\documentclass[12pt,leqno]{article}
\usepackage{amsmath,amssymb,amsthm,mathrsfs,url}

\def\res{\hbox{{\vrule height .22cm}{\leaders\hrule\hskip.2cm}}}

\DeclareMathOperator{\dist}{dist}
\DeclareMathOperator{\divergence}{div}	
\DeclareMathOperator{\Lip}{Lip}	
\DeclareMathOperator{\gph}{graph}

\DeclareMathOperator{\spt}{spt}
\DeclareMathOperator{\Clos}{Clos}

\newcommand{\dive}[1]{\divergence_{#1}}
\newcommand{\graph}[2]{\gph_{#1} #2}

\DeclareMathOperator{\ok}{k}
\DeclareMathOperator{\op}{p}
\DeclareMathOperator{\oq}{q}
\DeclareMathOperator{\prerot}{rot}

\newcommand{\bop}{\boldsymbol{\op}}
\newcommand{\boq}{\boldsymbol{\oq}}

\newcommand{\ETA}{\boldsymbol{\eta}}
\newcommand{\rot}{\boldsymbol{\prerot}}

\newcommand{\bB}{\bar{B}}

\newcommand{\bE}{\mathbf{E}}
\newcommand{\bL}{\mathbf{L}}
\newcommand{\bM}{\mathbf{M}}
\newcommand{\bV}{\mathbf{V}}
\newcommand{\bW}{\mathbf{W}}

\newcommand{\bbE}{\mathbb{E}}
\newcommand{\bbN}{\mathbb{N}}
\newcommand{\bbR}{\mathbb{R}}

\newcommand{\cH}{\mathcal{H}}
\newcommand{\cV}{\mathcal{V}}

\newcommand{\sR}{\mathscr{R}}
\newcommand{\sT}{\mathscr{T}}

\numberwithin{equation}{subsection}
\newtheorem{appendixlemma}[equation]{Lemma}

\newtheorem*{lemma}{Lemma}
\newtheorem*{definition}{Definition}
\newtheorem*{theorem}{Theorem}
\newtheorem*{corollary}{Corollary}

\begin{document}
\begin{flushleft}

AUTHOR: Leobardo Rosales 

\medskip

TITLE: Partial boundary regularity for co-dimension one area-minimizing currents at immersed $C^{1,\alpha}$ tangential boundary points.

\medskip

ABSTRACT: We give partial boundary regularity for co-dimension one absolutely area-minimizing currents at points where the boundary consists of a sum of $C^{1,\alpha}$ submanifolds, possibly with multiplicity, meeting tangentially, given that the current has a tangent cone supported in a hyperplane with constant orientation vector; this partial regularity is such that we can conclude the tangent cone is unique. The proof closely follows that giving the boundary regularity result of Hardt and Simon in \cite{HS79}.

\medskip


\medskip


\section{Introduction}

Through a careful modification of the work found in \cite{HS79}, we are able to give partial regularity for co-dimension one absolutely area-minimizing currents at points where the boundary is tangentially $C^{1,\alpha}$ immersed. Our main result, Theorem \ref{main}, can be heuristically stated as follows:

\medskip

{\bf Theorem \ref{main}} \emph{Suppose $T$ is an $n$-dimensional absolutely area-minimizing integer rectifiable current in an open subset of $\bbR^{n+1}$ containing the origin, and that near the origin $\partial T$ consists of a sum of $(n-1)$-dimensional $C^{1,\alpha}$ orientable submanifolds for some $\alpha \in (0,1]$, each possibly with multiplicity, meeting tangentially (with same orientation) at the origin. Suppose as well that $T$ has a tangent cone at the origin 
$$\mathbb{C} = M \bbE^{n} \res \{ (y_{1},\ldots,y_{n}) \in \bbR^{n} : y_{n} > 0 \} + m \bbE^{n} \res \{ (y_{1},\ldots,y_{n}) \in \bbR^{n}: y_{n} < 0 \}$$
where $M,m$ are positive integers with $m \leq M-1.$ Then near the origin, there is a large region of the horizontal hyperplane $\bbR^{n} \times \{0\}$ such that the support of $T$ over this region is the graph of a $C^{1,\frac{\alpha}{4n+6}}$ function.}

\medskip

Furthermore, the region is such that we can conclude $\mathbb{C}$ is the unique tangent cone of $T$ at the origin. Here, $\bbE^{n}$ is the current associated to the hyperplane $\bbR^{n} \times \{0\}$ with usual orientation; see 4.1.7 of \cite{F69}. See 4.3.16 of \cite{F69} for the definition of a tangent cone of a current.

\medskip

Theorem \ref{main} is precisely a generalization of Corollory 9.3 of \cite{HS79}, after applying the Hopf-type boundary point lemma given by Lemma 7  of \cite{FG57}, also appearing in \cite{HS79} as Lemma 10.1. We can get full boundary regularity via \cite{W83} in the special case that $\partial T$ is supported on exactly one $C^{1,\alpha}$ submanifold (if for example $m = M-1$), letting in this case $m \in \{0,\ldots,M-1\}$ and $M \geq 1.$ By \cite{W83} and the fact that the tangent cone of $T$ at the origin is $\mathbb{C}$ as above, if $m=0$ then $T$ corresponds to a $C^{1,\alpha}$ hypersurface-with-boundary, and if $m \geq 1$ then the support of $T$ near the origin is a real analytic hypersurface, with $T$ having multiplicity $M,m$ on either side of $\partial T.$  

\subsection{Modifying the work of Hardt and Simon}

To the reader thoroughly familiar with the entirety of \cite{HS79}, we can describe the key estimate and minor modification which allows us to carry over the proofs found in \cite{HS79}. 

\medskip

First, the key estimate needed is given by \eqref{appendixlemma2projectionmass}. This estimate is crucial to prove Lemma \ref{sec6.4}, analogous to Lemma 6.4 of \cite{HS79}. We use \eqref{appendixlemma2projectionmass} to prove Lemma \ref{sec6.4}, to show the function defined by taking the top sheet of the harmonic blowups (when the harmonic blowups are linear as in Lemma \ref{sec6.4}) is harmonic, leading to conclude the harmonic blowups are all given by the same linear function. See \eqref{6.4(15)}, where we specifically refer to \eqref{appendixlemma2projectionmass} in the proof of Lemma \ref{sec6.4}.

\medskip

The minor modification which must be made is seen in Lemma \ref{sec3.2}, which is analogous to Lemma 3.2 of \cite{HS79}. One can see in the right-hand side of the conclusion, we have replaced $c_{7} \tau^{-2} \kappa^{2}_{T}$ in Lemma 3.2 of \cite{HS79} with $c_{7} \kappa_{T}$ in Lemma \ref{sec3.2}. This difference arises from the fact that, as opposed to \cite{HS79}, the boundary $\partial T$ no longer corresponds to integrating over an embedded submanifold. As such, a slightly different proof must be given for Lemma \ref{sec3.2} than the proof of Lemma 3.2 of \cite{HS79}. We must subsequently take care that the rest of \cite{HS79} follows through keeping in mind Lemma \ref{sec3.2}; notably, we must check that Lemmas \ref{sec4.1},\ref{sec6.4} still hold.

\medskip

Besides these two points, the rest of \cite{HS79} passes through essentially without change, with only minor modifications due to the more general structure of $\partial T.$ For the reader already familiar with \cite{HS79} who wishes a more terse exposition, we reference \cite{R15}, which is a shorter version of this work in which only the differences with \cite{HS79} are explained. However, we take this opportunity to reintroduce the seminal work of \cite{HS79}, using more modern notation. We also include clarifying exposition, some of which is taken from \cite{DS93}, which extends the results of \cite{HS79} to minimizing currents with prescribed mean curvature.

\subsection{An application of Theorem \ref{main}} 

We note an application of Theorem \ref{main}, which in fact motivated the present work. Recently in \cite{R13} the author introduced the $c$-isoperimetric mass of currents, which is defined for each $c>0$ by $$\bM^{c}(T) = \bM(T) + c \bM (\partial T)^{\kappa}$$ whenever $T$ is an $n$-dimensional integer multiplicity rectifiable current in $\bbR^{n+k},$ $\bM$ is the usual mass on currents, and $\kappa = \frac{n}{n-1}$ is the isoperimetric exponent. 

\medskip

This leads to define and study a minimization problem. Let $\Gamma$ be an $(n-1)$-dimensional integer rectifiable current in $\bbR^{n+k}$ with compact support and $\partial \Gamma = 0,$ which we refer to as the fixed boundary. Define $\mathbf{I}_{\Gamma}(\bbR^{n+k})$ to be the set of $n$-dimensional integer rectifiable currents $T$ with compact support so that $\partial T = \Gamma + \Sigma$ where $\Gamma$ and $\Sigma$ have disjoint supports. We then say $\mathbf{T}_{c} \in \mathbf{I}_{\Gamma}(\bbR^{n+k})$ is a solution to the $c$-Plateau problem with respect to fixed boundary $\Gamma$ if $\mathbf{T}_{c}$ minimizes $\bM^{c}$ amongst all $T \in \mathbf{I}_{\Gamma}(\bbR^{n+k})$ (see Definition 3.3 of \cite{R13} with $U = \bbR^{n+k}$). For such $\mathbf{T}_{c},$ writing $\partial \mathbf{T}_{c} = \Gamma + \Sigma_{c}$ we refer to $\Sigma_{c}$ as the free boundary.

\medskip

Theorem 8.2 of \cite{R13} concludes there is no solution to the $c$-Plateau problem $\mathbf{T}_{c}$ with $\partial \mathbf{T}_{c} = \Gamma + \Sigma_{c}$ with nonzero free boundary $\Sigma_{c}$ a smooth embedded $(n-1)$-dimensional submanifold with parallel mean curvature (that is constant mean curvature in the sense of \cite{H73}) so that $\mathbf{T}_{c}$ near $\Sigma_{c}$ is a smooth submanifold-with-boundary. This can be used in Theorem 9.1 of \cite{R13} to show that in case the fixed boundary $\Gamma$ is one-dimensional in the plane, that is if $n=2,k=0,$ then free boundaries must always be empty. However, so-called non-trivial solutions in the limit can occur, as seen in Theorem 10.2 of \cite{R13} which shows that for small values of $c>0$ when $\Gamma$ is the square in the plane, the infimum of $\bM^{c}$ is attained in the limit by a sequence of currents in $\mathbf{I}_{\Gamma}(\bbR^{2})$ which converge to a nonempty current not in $\mathbf{I}_{\Gamma}(\bbR^{2}).$

\medskip

The author conjectures that this holds generally in $n=2,k=1:$ if the fixed boundary $\Gamma$ is one-dimensional in $\bbR^{3},$ then for each $c>0$ either every solution to the $c$-Plateau problem $\mathbf{T}_{c}$ with fixed boundary $\Gamma$ has empty free boundary, so that $\partial \mathbf{T}_{c} = \Gamma,$ or the infimum value of $\bM^{c}$ can only be attained in the limit by a sequence of currents in $\mathbf{I}_{\Gamma}(\bbR^{3}).$ Evidence for this is given by recent work by the author in \cite{R16a}, where it is proved (in case $n=2,k=1$) that $\mathbf{T}_{c}$ at singular points of the free boundary $\Sigma_{c}$ must have complicated topology; more specifically, $\mathbf{T}_{c}$ cannot be supported in a finite union of $C^{1}$ surfaces-with-boundary.

\subsection{More complete regularity}

The results of this work directly lead to the following regularity result, given by Theorem 3.18 of \cite{R16b}:

\medskip

{\bf Theorem 3.18 of \cite{R16b}:} \emph{Suppose $T$ is an $n$-dimensional absolutely area-minimizing integer rectifiable current in an open subset of $\bbR^{n+1}$ containing the origin, and that near the origin $\partial T$ consists of a sum of $(n-1)$-dimensional $C^{1,\alpha}$ orientable submanifolds for some $\alpha \in (0,1]$ through the origin, each possibly with multiplicity, which pairwise meet only tangentially (with same orientation). Suppose as well that $T$ has a tangent cone at the origin 
$$\mathbb{C} = M \bbE^{n} \res \{ (y_{1},\ldots,y_{n}) \in \bbR^{n} : y_{n} > 0 \} + m \bbE^{n} \res \{ (y_{1},\ldots,y_{n}) \in \bbR^{n}: y_{n} < 0 \}$$
where $M,m$ are positive integers with $m \leq M-1.$ Then the support of $T$ near the origin is the graph of a smooth solution to the minimal surface equation $u:\bbR^{n} \rightarrow \bbR,$ and the orientation vector of $T$ near the origin corresponds to the upward pointing unit normal of the graph of $u.$}

\medskip

We make clear that in Theorem 3.18 of \cite{R16b}, we not only assume that $\partial T$ consists of orientable submanifolds (with multiplicity) intersecting tangentially (with same orientation) at the origin, in fact we assume that anywhere a pair of these submanifolds intersect they do so tangentially. In \cite{R16b} we say such $T$ has \emph{$C^{1,\alpha}$ tangentially immersed boundary}, and there the author studies such $T$ more completely.

\medskip

As proving Theorem 3.18 of \cite{R16b} takes some effort in and of itself, we focus here on proving Theorem \ref{main}.

\subsection{Counterexamples}

The examples of stable branched minimal immersions given by \cite{SW07} and \cite{R10} show the absolutely area-minimizing hypothesis cannot be relaxed to stability. Indeed, Theorem 1 of \cite{SW07} holds that if $u_{0}$ is a solution to the two-valued minimal surface equation (see the operator $\mathcal{M}_{0}$ at the start of \S 3 of \cite{R10}) over the punctured unit disk in $\bbR^{2}$ which can be extended continuously across the origin, then $$G = \{ (re^{i \theta},u_{0}(r^{1/2} e^{i \theta/2})): r \in (0,1), \theta \in \bbR \}$$ is a stable minimal immersion with $C^{1,\alpha}$ branch point at $(0,u_{0}(0)),$ for some $\alpha \in (0,1).$ \cite{SW07} and \cite{R10} show a large non-trivial class of such solutions exist. We can thus show there is a solution $u_{0}$ to the two-valued minimal surface equation which can be extended continuously across the origin, so that $\{ (re^{i \theta},u(r^{1/2}e^{i\theta/2})): r \in (0,1), \theta \in (0,3\pi) \}$ satisfies the assumptions of Theorem \ref{main} (with $M=2,m=1,$ and with the absolutely area-minimizing condition replaced by stability) but fails to satisfy the partial regularity conclusions given there.  

\medskip

Neither does Theorem \ref{main} hold in higher co-dimensions. A counterexample is given by considering the region $\{ (re^{i \theta}, r^{3/2} e^{\frac{3i \theta}{2}}) : r > 0, \ \theta \in [0,3 \pi] \}$ of the holomorphic variety $\{ (z,w) : z^{3} = w^{2} \} \subset \mathbb{C} \times \mathbb{C} \cong \bbR^{4},$ which is still calibrated and hence area-minimizing. 

\medskip

The best general result in all co-dimensions is thus as in \cite{A75}, boundary regularity in case of currents with $C^{1,\alpha}$ embedded boundary at points of density near $1/2$; see Theorem 0.1 of \cite{DS02}, which concludes this in fact for almost minimizing currents of arbitrary co-dimension, or more generally, \cite{B10} which does this for stationary varifolds. Observe again, that the examples from \cite{SW07} and \cite{R10} show the density $=1/2$ assumption cannot be relaxed without the area-minimizing hypothesis. 

\subsection{Future work}

Clearly, we wish to extend Theorem \ref{main} to the case when $m=0,$ that is when $T$ has tangent cone at the origin
$$\mathbb{C} = M \bbE^{n} \res \{ (y_{1},\ldots,y_{n}) \in \bbR^{n}: y_{n}>0 \}$$
for $M$ a positive integer. What prevents us from modifying the work of \cite{HS79} in case $m=0$ is specifically \eqref{appendixlemma2projectionmass}, which is only generally true if $m \geq 1.$ 

\medskip

Naturally, we also wish to give a similar result to Theorem \ref{main} in all cases, without assuming $T$ has a tangent cone at the origin of a certain form. At least, we make the following:

\medskip

{\bf Conjecture:} \emph{Suppose $T$ is an $n$-dimensional absolutely area-minimizing integer rectifiable current in an open subset of $\bbR^{n+1}$ containing the origin, and that near the origin $\partial T$ consists of a sum of $(n-1)$-dimensional $C^{1,\alpha}$ orientable submanifolds for some $\alpha \in (0,1],$ each possibly with multiplicity, meeting tangentially (with same orientation) at the origin. Then $T$ has unique tangent cone at the origin.}

\medskip

To prove this, it may be necessary to apply the techniques of \cite{W14}, which gives a general regularity theory for co-dimension one stable minimal hypersurfaces. We suspect this as \cite{W14} proceeds by considering the geometric structure of a co-dimension one stable minimal hypersurface near points where such a hypersurface has a tangent cone consisting of a sum of half-hyperplanes meeting along a common co-dimension two subspace (the ``spine'' of the tangent cone). Moreover, \cite{W14} carries through this examination in part by generalizing the techniques of \cite{HS79}, albeit to an exceedingly sophisticated degree.

\subsection{Summary}

We now discuss the organization of this work. 

\medskip

Our aim is to extend Corollary 9.3 of \cite{HS79} to the conditions set forth by Theorem \ref{main}. This involves making small but ubiquitous changes to the proofs found in \cite{HS79}, up to Theorem 11.1 found therein. This task is undertaken in sections \ref{sec1}-\ref{proofmain}. In section \ref{sec1}, in particular in sections \ref{sec1.1,1.2,1.3},\ref{sec1.4}, we introduce our notation, which is different and more modern than the notation used in \cite{HS79}. To modify \cite{HS79}, we must rely on the calculations established in the Appendix, which contain the deeper differences between the present setting and the proof of \cite{HS79}. 

\medskip

Each section of \cite{HS79} is devoted to a large theoretical step, further divided into subsections, given either by closely related computations, lemma, or theorem. We follow the same general structure as well. We include minor, although clarifying, corrections to \cite{HS79}. We will also include expository comments, some taken from \cite{DS93}.

\medskip

Before all this, we state in section \ref{mainresults} the main result Theorem \ref{main}, giving exactly the assumptions necessary. Corollary \ref{uniquetangentcones} concludes uniqueness of tangent cones for $T$ satisfying the conditions of Theorem \ref{main}. We also remark in Theorem \ref{existencetangentcones}, using \cite{B76}, why currents $T$ as in Theorem \ref{main} must have tangent cones at such tangential boundary points.   

\section{Main results} \label{mainresults}

For the definition of absolutely area-minimizing, consult 5.1.6 of \cite{F69}. We denote $\bbE^{n-1}, \bbE^{n}$ to be the currents associated respectively to $\bbR^{n-1} \times \{0\},\bbR^{n} \times \{0\}$ in $\bbR^{n+1}$ each with usual orientation, as in 4.1.7 \cite{F69}. Given $r>0,$ we denote the homothety $\ETA_{0,\rho}(x) = x/\rho,$ and for a current $T$ we let $\ETA_{0,\rho \sharp} T$ be the push-forward of $T$ by $\ETA_{0,\rho}.$ Let also $\Clos A$ denote the closure of $A \subseteq \bbR^{n+1}.$ We now state our main result. 

\begin{theorem} \label{main} Suppose $\alpha \in (0,1]$ and $T$ is an $n$-dimensional absolutely area-minimizing locally rectifiable integer multiplicity current in $\{ x \in \bbR^{n+1}:|x| < 3\}.$ We also suppose $T$ satisfies the hypothesis:
\begin{enumerate} 
\item[($\ast$)] 
$\begin{aligned} \partial T \res \{ (x_{1},\ldots,& x_{n+1}) \in \bbR^{n+1}: |(x_{1},\ldots,x_{n-1})| < 2, \ |x_{n}| < 2 \} \\ 
& = (-1)^{n} \sum_{\ell=1}^{N} m_{\ell} \Phi_{T,\ell \sharp}(\bbE^{n-1} \res \{ z \in \bbR^{n-1}: |z| < 2 \}), 
\end{aligned}$ 
\\ where $m_{\ell}$ are positive integers, and for each $\ell \in \{1,\ldots,N\}$ and $z \in \bbR^{n-1}$ with $|z|<2$
$$\Phi_{T,\ell}(z) = (z,\varphi_{T,\ell}(z),\psi_{T,\ell}(z))$$ 
where $\varphi_{T,\ell},\psi_{T,\ell} \in C^{1,\alpha}(\{ z \in \bbR^{n-1}:|z|<2 \})$ with $\varphi_{T,\ell}(0) = 0 = \psi_{T,\ell}(0), \ D \varphi_{T,\ell}(0) = 0 = D \psi_{T,\ell}(0).$
\item[($\ast \ast$)] $T$ has a tangent cone 
$$M \bbE^{n} \res \{ (y_{1},\ldots,y_{n}) \in \bbR^{n}: y_{n} > 0 \} + m \bbE^{n} \res \{ (y_{1},\ldots,y_{n}) \in \bbR^{n}: y_{n} < 0 \}$$ 
at the origin, where $M \in \{ 2,3,\ldots \}$ and $m \in \{ 1,\ldots,M-1 \}.$
\end{enumerate}
Then there is a $\delta = \delta(n,M,m,\alpha) \in (0,1)$ sufficiently small, so that letting 
$$\begin{aligned} 
\tilde{V} & = \{ y=(y_{1},\ldots y_{n}) \in \bbR^{n} : y_{n} > |y|^{1+\frac{\alpha}{4n+6}}, \ |y| < \delta \} \\
\tilde{W} & = \{ y=(y_{1},\ldots y_{n}) \in \bbR^{n} : y_{n} < - |y|^{1+\frac{\alpha}{4n+6}}, \ |y| < \delta \},
\end{aligned}$$ 
then for $\rho>0$ sufficiently small depending on $T$ we have 
$$\begin{aligned} 
\bop^{-1}(\tilde{V}) \cap \spt \ETA_{0,\rho \sharp} T & = \graph{\tilde{V}}{\tilde{v}} \\
\bop^{-1}(\tilde{W}) \cap \spt \ETA_{0,\rho \sharp} T & = \graph{\tilde{W}}{\tilde{w}} \end{aligned}$$ 
for some $\tilde{v} \in C^{1,\frac{\alpha}{4n+6}}(\Clos \tilde{V}), \tilde{w} \in C^{1,\frac{\alpha}{4n+6}}(\Clos \tilde{W})$ such that $\tilde{v}|_{\tilde{V}},\tilde{w}|_{\tilde{W}}$ satisfy the minimal surface equation with $D \tilde{v}(0) = 0 = D \tilde{w}(0).$ Furthermore, we have 
$$\sup_{y \in \tilde{V}} \frac{|D^{2}\tilde{v}(y)|}{|y|^{\frac{\alpha}{4n+6}-1}} + \sup_{y \in \tilde{W}} \frac{|D^{2}\tilde{w}(y)|}{|y|^{\frac{\alpha}{4n+6}-1}} \leq c$$ 
for some $c=c(n,M,m) \in (0,\infty).$ 
\end{theorem}

Note that $M-m = \sum_{\ell=1}^{N} m_{\ell}.$ As noted in the introduction, the case $m=M-1$ is just Corollary 9.3 of \cite{HS79}, together with Lemma 10.1 of \cite{HS79}. Also, if $N=1,$ that is when $\partial T$ is an $(n-1)$-dimensional $C^{1,\alpha}$ submanifold with multiplicity, then Theorem \ref{main} follows in this case by the higher multiplicity boundary regularity given by \cite{W83}. Nonetheless, the proof we give below will cover all cases. The following corollary immediately follows.

\begin{corollary} \label{uniquetangentcones} 
If $T$ is as in Theorem \ref{main}, then $T$ has unique tangent cone 
$$M \bbE^{n} \res \{ (y_{1},\ldots,y_{n}) \in \bbR^{n}: y_{n} > 0 \} + m \bbE^{n} \res \{ (y_{1},\ldots,y_{n}) \in \bbR^{n}: y_{n} < 0 \}$$ 
at the origin. 
\end{corollary}

Before proceeding, we prove a lemma showing the existence of tangent cones to start.

\begin{lemma} \label{existencetangentcones} 
Suppose $\alpha \in (0,1]$ and $T$ is an $n$-dimensional absolutely area-minimizing locally rectifiable integer multiplicity current in $\{ x \in \bbR^{n+1}: |x| < 3 \}$ satisfying hypothesis $(\ast).$ Then $T$ has an oriented tangent cone at the origin, and every oriented tangent cone of $T$ at the origin is absolutely area minimizing with density at the origin equal to the density of $T$ at the origin. 
\end{lemma}

\begin{proof} 
By Theorems 3.6,3.3 of \cite{B76} we only need to check the finiteness of 
$$\nu_{1}^{\partial T}(0) = \int_{\{ x \in \bbR^{n+1}: |x| < 1 \}} \frac{|\vec{\partial T} \wedge x|}{|x|^{n}} \ d \mu_{\partial T} (x).$$ 
This follows by hypothesis $(\ast).$ 
\end{proof}

\medskip

Observe of course that $T$ may satisfy hypothesis $(\ast)$ but not ($\ast \ast$), if for example $T$ is a union of half-planes in space, appropriately oriented.

\section{Notation and preliminaries} \label{sec1}

While it is tempting to use the same notation as \cite{HS79}, we take this opportunity to offer cleaner, modern notation. This will not inconvenience the reader, as this article can be read independently of \cite{HS79}.

\medskip

For $\bbN = \{ 1,2,\ldots \},$ in what follows we fix numbers
$$n,M \in \bbN \text{ with } n \geq 2, \ m \in \{0\} \cup \bbN \text{ with } m \leq M-1, \ \alpha \in (0,1].$$
In this section we will also use $\tilde{n} \in \{ 1,\ldots,n+1 \}.$ Observe that while we will only show Theorem \ref{main} holds with $m \in \bbN,$ many of the calculations here hold with $m=0.$ In particular, we only use $m \geq 1$ in Lemma \ref{sec6.4}, Theorem \ref{sec7.3}, section \ref{sec8}, Theorem \ref{sec9.1}, section \ref{proofmain}, and in \eqref{appendixlemma2projectionmass}.

\medskip

We will introduce and use constants $c_{1},\ldots,c_{49},$ a few of which are exactly as in \cite{HS79}. We must take care to show that $c_{1},\ldots,c_{46}$ depend only on $n,M,m$ in order to apply the iterative arguments used to prove the crucial Lemma \ref{sec9.2}. The constants $c_{47},c_{48},c_{49}$ will only be used in proving the general Hopf-type boundary point Lemma \ref{sec10.1}; as such, $c_{47},c_{48},c_{49}$ are exactly as in \cite{HS79}.

\subsection{Notation associated to Euclidean space.} \label{sec1.1,1.2,1.3}

This section combines sections 1.1,1.2,1.3 of \cite{HS79}. We will mostly work in the Euclidean spaces $\bbR^{n-1},\bbR^{n},\bbR^{n+1}.$
\begin{itemize}
 \item We shall typically write points in the three Euclidean spaces
$$\begin{aligned}
x&=(x_{1},\ldots,x_{n+1}) \in \bbR^{n+1}, \\
y&=(y_{1},\ldots,y_{n}) \in \bbR^{n}, \\
z&=(z_{1},\ldots,z_{n-1}) \in \bbR^{n-1}.
\end{aligned}$$
We shall as well use $\tilde{x} \in \bbR^{n}, \tilde{y} \in \bbR^{n}, \tilde{z} \in \bbR^{n-1}.$ We will also identify $\bbR^{n}$ with $\bbR^{n} \times \{ 0 \} \subset \bbR^{n+1}$ by identifying $y \in \bbR^{n}$ with $(y,0) \in \bbR^{n+1}$; we do the same for $\bbR^{n-1}$ in both $\bbR^{n}$ and $\bbR^{n+1}.$
 \item If $A \subseteq \bbR^{n+1},$ then we shall denote the closure of $A$ by $\Clos{A},$ and we denote the boundary of $A$ by $\partial A = \Clos{A} \setminus A.$ 
 \item Define the following functions. Fix $\rho \in (0,\infty)$ and $\theta \in \bbR.$
$$\begin{aligned}
\ETA_{\rho}:&\bbR^{n+1} \rightarrow \bbR^{n+1}, \ \ETA_{\rho}(x) = x/\rho. \\
\rot_{\theta}:&\bbR^{n+1} \rightarrow \bbR^{n+1}, \\
& \rot_{\theta}(x) = (x_{1},\ldots,x_{n-1},x_{n} \cos \theta - x_{n+1} \sin \theta,x_{n} \sin \theta + x_{n+1} \cos \theta). \\
\bop:&\bbR^{n+1} \rightarrow \bbR^{n}, \ \bop(x) = (x_{1},\ldots,x_{n}). \\
\boq_{0}:&\bbR^{n+1} \rightarrow \bbR, \ \boq_{0}(x)= x_{n}. \\
\boq_{1}:&\bbR^{n+1} \rightarrow \bbR, \ \boq_{1}(x)= x_{n+1}.
\end{aligned}$$
Note that the notation for $\ETA_{\rho}$ is derived from the usual notation for the more general homothety $\ETA_{\tilde{x},\rho}(x)=\frac{x-\tilde{x}}{\rho}.$ Since we will only once need the homothety $\ETA_{-x,1}$ (translation by $x,$ in the proof of Theorem \ref{sec5.4}), then we use $\ETA_{\rho}$ for rescaling. 
 \item We denote the following subsets of $\bbR^{n+1}.$ For $x \in \bbR^{n+1}$ and $\rho \in (0,\infty),$ we shall denote the open and closed ball by
$$\begin{aligned}
B_{\rho}(x) & = \{ \tilde{x} \in \bbR^{n+1}: |\tilde{x}-x| < \rho \} \\
\bB_{\rho}(x) & = \{ \tilde{x} \in \bbR^{n+1}: |\tilde{x}-x| \leq \rho \}.
\end{aligned}$$
We will mainly consider balls centered at the origin, and so we denote
$$B_{\rho} = B_{\rho}(0) \text{ and } \bB_{\rho} = \bB_{\rho}(0).$$
Of great importance will be vertical closed cylinders centered at the origin, so we denote
$$C_{\rho} = \{ \tilde{x} \in \bbR^{n+1}: |\bop(\tilde{x})| \leq \rho \}.$$
 \item Given $z \in \bbR^{n-1}$ and $y \in \bbR^{n},$ we denote the open and closed balls in $\bbR^{n-1},\bbR^{n}$ with $\rho \in (0,\infty)$
$$\begin{array}{cc}
B^{n-1}_{\rho}(z) = B_{\rho}(z) \cap \bbR^{n-1}, & \bB^{n-1}_{\rho}(z) = \bB_{\rho}(z) \cap \bbR^{n-1} \\
B^{n}_{\rho}(y) = B_{\rho}(y) \cap \bbR^{n}, & \bB^{n}_{\rho}(y) = \bB_{\rho}(y) \cap \bbR^{n}.
\end{array}$$
We will specifically use balls centered at the origin, denoted by
$$\begin{array}{cc}
B^{n-1}_{\rho} = B_{\rho} \cap \bbR^{n-1}, & \bB^{n-1}_{\rho} = \bB_{\rho} \cap \bbR^{n-1} \\
B^{n}_{\rho} = B_{\rho} \cap \bbR^{n}, & \bB^{n}_{\rho} = \bB_{\rho} \cap \bbR^{n}
\end{array}$$
The following sets will also be important, so we denote them as in \cite{HS79}:
$$\bL = B^{n-1}_{1}, \bV = \{ y \in B^{n}_{1} :y_{n}>0 \}, \bW = \{ y \in B^{n}_{1} :y_{n}<0 \}$$ 
$$\bV_{\sigma} = \{ y \in \bV: \dist(y, \partial \bV) > \sigma \}, \bW_{\sigma} = \{ y \in \bW: \dist(y, \partial \bW) > \sigma \}$$
for $\sigma \in (0,1).$
 \item Let $e_{1},\ldots,e_{n+1} \in \bbR^{n+1}$ be the standard basis vectors; we will also denote the standard basis $e_{1},\ldots,e_{n-1} \in \bbR^{n-1}$ and $e_{1},\ldots,e_{n} \in \bbR^{n}.$ We let $\ast: \bigwedge_{n} \bbR^{n+1} \rightarrow \bbR^{n+1}$ be the Hopf map
 $$\ast \left( \sum_{i=1}^{n+1} x_{k} (-1)^{k-1} e_{1} \wedge \ldots \wedge e_{k-1} \wedge e_{k+1} \wedge \ldots \wedge e_{n+1} \right) = \sum_{k=1}^{n+1} x_{k} e_{k}.$$
 Note that $\ast (e_{1} \wedge \ldots \wedge e_{n}) = (-1)^{n} e_{n+1}.$
 \item Let $\cH^{\tilde{n}}$ be $\tilde{n}$-dimensional Hausdorff measure in $\bbR^{n+1}.$ Denote $\varpi_{\tilde{n}} = \cH^{\tilde{n}}(B_{1} \cap \bbR^{\tilde{n}}).$ 
 \item $D$ will denote differentiation over Euclidean space, with dimension clear by context, and $D_{k}$ partial differentiation with respect to the $k^{\text{th}}$ variable.
 
\medskip

$\Lip(\phi)$ will denote the Lipschitz constant for a Lipschitz function $\phi: \bbR^{n+1} \rightarrow \bbR^{n+1}.$

\medskip

We will use various spaces of H$\ddot{\text{o}}$lder continuously differentiable functions. For example, $C^{k,a}_{c}(A,\tilde{A})$ will denote the $k$ times H$\ddot{\text{o}}$lder continuously differentiable functions of order $a$ with compact support in $A$ and values in $\tilde{A},$ where $k \in \bbN$ and $a \in [0,1].$
\end{itemize}

\subsection{Notation associated to $T \in \sR_{n}(\bbR^{n+1}).$} \label{sec1.4}

This is section 1.4 of \cite{HS79}. For a thorough introduction to currents, see \cite{F69},\cite{S83}.
\begin{itemize}
 \item Recall that $\mathcal{D}^{\tilde{n}}(\bbR^{n+1})$ denotes the smooth compactly supported $\tilde{n}$-forms. The $\tilde{n}$-dimensional currents over $\bbR^{n+1}$ are then the set of linear functionals over $\mathcal{D}^{\tilde{n}}(\bbR^{n+1}).$
 
 \medskip

Given an $\tilde{n}$-dimensional current $T,$ as usual $\partial T$ will denote the associated boundary of $T.$ So, $\partial T$ is the $(\tilde{n}-1)$-dimensional current defined by $\partial T(\omega)=T(d \omega)$ for $\omega \in \mathcal{D}^{\tilde{n}-1}(\bbR^{n+1}).$
 \item For $T$ a current over $\bbR^{n+1}$ and $\phi:\bbR^{n+1} \rightarrow \bbR^{n+1},$ we denote $\phi_{\sharp} T$ the push-forward current of $T$ by $\phi$; we shall frequently make use of $\ETA_{\rho \sharp} T.$
 \item Denote by $\bbE^{\tilde{n}}$ the $\tilde{n}$-dimensional current in $\bbR^{n+1}$ given by $\bbE^{\tilde{n}}(\omega) = \int_{\bbR^{\tilde{n}}} \langle \omega, e_{1} \wedge \ldots \wedge e_{\tilde{n}} \rangle \ d \cH^{\tilde{n}}$ for $\omega \in \mathcal{D}^{\tilde{n}}(\bbR^{n+1}).$
 \item For $T$ an $\tilde{n}$-dimensional current in $\bbR^{n+1},$ we let $\mu_{T}$ denote the associated mass measure of $T$; This is given by $\mu_{T}(U) = \sup_{\omega \in \mathcal{D}^{\tilde{n}}(U),|\omega| \leq 1} T(\omega)$ when $U$ is an open subset of $\bbR^{n+1}.$ We define the support of $T$ by $\spt T = \spt \mu_{T}.$ Let $\vec{T}$ be the $\bigwedge_{\tilde{n}}(\bbR^{n+1})$-valued orientation of $T.$ Thus, $T(\omega) = \int \langle \omega,\vec{T} \rangle \ d \mu_{T}$ for $\omega \in \mathcal{D}^{\tilde{n}}(\bbR^{n+1}).$ 

\medskip

We denote the mass of $T$ by $\bM(T)=\mu_{T}(\bbR^{n+1}).$

\medskip
 
 For $A$ a $\mu_{T}$-measurable set, we let $T \res A$ denote the restriction current $(T \res A)(\omega) = \int_{A} \langle \omega,\vec{T} \rangle \ d \mu_{T}$ for $\omega \in \mathcal{D}^{\tilde{n}}(\bbR^{n+1}).$ In some instances we may use $\bM(T \res A)$ instead of $\mu_{T}(A),$ in order to avoid cluttered notation.

\medskip
 
We denote the density of $T$ at $x \in \bbR^{n+1}$ by $\Theta_{T}(x) = \lim_{\rho \searrow 0} \frac{\mu_{T}(B_{\rho}(x))}{\varpi_{\tilde{n}} \rho^{\tilde{n}}},$ whenever this limit exists.   
 \item We shall let $\sR_{\tilde{n}}(\bbR^{n+1})$ denote the $\tilde{n}$-dimensional (integer) rectifiable currents over $\bbR^{n+1}.$ 
 
 \medskip
 
 For $T \in \sR^{\tilde{n}}(\bbR^{n+1}),$ denote the approximate ($\tilde{n}$-dimensional) tangent space of $T$ at $x$ by $T_{x}T,$ for the $\mu_{T}$-almost-every $x \in \bbR^{n+1}$ such that this space exists.
 \item For $T \in \sR_{n}(\bbR^{n+1}),$ we denote the associated generalized unit normal vector field $\cV^{T}$ defined by $\cV^{T}(x) = \ast \vec{T}(x)$ for $\mu_{T}$-almost-every $x$; note that $\cV^{T}(x) \perp T_{x}T.$ We will also write the components $\cV^{T}=(\cV^{T}_{1},\ldots,\cV^{T}_{n+1}).$
 \item For $f \in C^{1}(\bbR^{n+1})$ let $\nabla^{T} f = Df - (Df \cdot \cV^{T}) \cV^{T}$ denote tangential differentiation over $T.$ For $X \in C^{1}(\bbR^{n+1};\bbR^{n+1})$ let $\dive{T}{X} = \dive{\bbR^{n+1}}{X} - \cV^{T} \cdot (DX \cV^{T})$ denote the divergence over $T.$
 \item We say $T \in \sR_{n}(\bbR^{n+1})$ is (absolutely) area minimizing if for any $R \in \sR_{n}(\bbR^{n+1})$ with $\partial R = \partial T$ we have $\bM(T) \leq \bM(R).$
\end{itemize}

\medskip

We now define the cylindrical and spherical excess of a current.

\medskip

\begin{definition} \label{excess} Suppose $T \in \sR_{n}(\bbR^{n+1}).$ We define the \emph{cylindrical excess} of $T$ at radius $r \in (0,\infty)$ to be
$$\bE_{C}(T,r) = \frac{\mu_{T}(C_{r}) - \mu_{\bop_{\sharp} T}(C_{r})}{r^{n}}.$$
If $\Theta_{T}(0)$ exists, we define the \emph{spherical excess} of $T$ at radius $r \in (0,\infty)$ by
$$\bE_{S}(T,r) = \frac{\mu_{T}(\bB_{r})}{r^{n}}-\varpi_{n} \Theta_{T}(0).$$
\end{definition}

From the definition of cylindrical excess, as well as Remark 27.2(3) of \cite{S83}, we get the following monotonicity for $0<r<s<\infty$
\begin{equation} \label{1.4(1)} \begin{aligned} 
r^{n} \bE_{C}(T,r) & = \int_{C_{r}} 1-|\cV^{T}_{n+1}| \ d \mu_{T} \\
& \leq \int_{C_{s}} 1-|\cV^{T}_{n+1}| \ d \mu_{T} = s^{n} \bE_{C}(T,s). 
\end{aligned} \end{equation}
This is 1.4(1) of \cite{HS79}, with a slight correction.

\subsection{The family $\sT.$} \label{sec1.5,1.6}

This is sections 1.5,1.6 of \cite{HS79} with substantial changes. Most notably, we replace section 1.5 of \cite{HS79} with assumption \eqref{defprojection}. 

\medskip

\begin{definition} Let $M \in \bbN,$ $m \in \{ 0,\ldots,M-1 \},$ and $\alpha \in (0,1].$ Define $\sT = \sT(M,m,\alpha)$ to be the collection of (absolutely) area minimizing $T \in \sR_{n}(\bbR^{n+1})$ such that the following four hold.

\medskip

First, $T$ satisfies the basic support, mass, and density identities
\begin{equation} \label{defmass}
\begin{aligned} 
\spt T & \subset \bB_{3} \\
\bM(T) & \leq 3^{n} (1+ M \varpi_{n}) \\
\Theta_{T}(0) & = \frac{M+m}{2}. 
\end{aligned}
\end{equation}

\medskip

Second, there are $N,m_{1},\ldots,m_{N} \in \bbN$ so that
\begin{equation} \label{defboundary}
\begin{aligned}
\sum_{\ell=1}^{N} m_{\ell} & = M-m \text{ and } \\
\partial T \res \{ x \in \bbR^{n+1}: & |(x_{1},\ldots,x_{n-1})|<2, \ |x_{n}|<2 \} \\
& = (-1)^{n} \sum_{\ell=1}^{N} m_{\ell} \Phi_{T,\ell \sharp} \left( \bbE^{n-1} \res B^{n-1}_{2} \right)
\end{aligned}
\end{equation}
where for each $\ell =1,\ldots,N$ we have $\Phi_{T,\ell}(z) = (z,\varphi_{T,\ell}(z),\psi_{T,\ell}(z))$ for functions $\varphi_{T,\ell},\psi_{T,\ell} \in C^{1,\alpha}(B^{n-1}_{2})$ satisfying $\varphi_{T,\ell}(0)=0=\psi_{T,\ell}(0),$ $D\varphi_{T,\ell}(0)=0=D\psi_{T,\ell}(0).$

\medskip

Third, and moreover, defining
\begin{equation} \label{defderivative}
\begin{aligned}
& \kappa_{T} = \frac{2}{\alpha} \max_{\ell=1,\ldots,N} \sup_{z \neq \tilde{z}} \frac{|(D \varphi_{T,\ell}(z),D \psi_{T,\ell}(z))-(D \varphi_{T,\ell}(\tilde{z}),D \psi_{T,\ell}(\tilde{z}))|}{|z-\tilde{z}|^{\alpha}} \\
& \text{then } \kappa_{T} \leq 1.
\end{aligned}
\end{equation} 

\medskip

Fourth and final, if we define $\varphi^{max}_{T},\varphi^{min}_{T}: B^{n-1}_{2} \rightarrow \bbR$ by
$$\varphi^{max}_{T}(z) = \max_{\ell = 1,\ldots,N} \varphi_{T,\ell}(z), \hspace{.25in} \ \varphi^{min}_{T}(z) = \min_{\ell = 1,\ldots,N} \varphi_{T,\ell}(z),$$
then
\begin{equation} \label{defprojection}
\begin{aligned}
\bop_{\sharp} (T \res C_{2}) \res & \{ y \in B^{n}_{2}: y_{n} \notin [\varphi_{T}^{min}(y_{1},\ldots,y_{n-1}), \varphi_{T}^{max}(y_{1},\ldots,y_{n-1})] \} \\ 
= & M \bbE^{n} \res \{ y \in B^{n}_{2}: y_{n} > \varphi^{max}_{T}(y_{1},\ldots,y_{n-1}) \} \\ 
& + m \bbE^{n} \res \{ y \in B^{n}_{2}: y_{n} < \varphi^{min}_{T}(y_{1},\ldots,y_{n-1}) \}. 
\end{aligned}
\end{equation}
\end{definition}

Using \eqref{appendixlemma2identity} and then \eqref{appendixlemma2retraction}, we compute for $r \in (0,2]$
\begin{equation} \label{1.6(1)} 
\begin{aligned} 
\bE_{S}(T,r) & \leq \frac{\mu_{T}(C_{r})}{r^{n}} - \left( \frac{M+m}{2} \right) \varpi_{n} \\
& \leq \bE_{C}(T,r) + \frac{\mu_{\bop_{\sharp} T}(C_{r})}{r^{n}} - \left( \frac{M+m}{2} \right) \varpi_{n} \\
& \leq \bE_{C}(T,r) + (M-m) \varpi_{n-1} r^{\alpha} \kappa_{T}.
\end{aligned} 
\end{equation}
This is a modification of 1.6(1) of \cite{HS79}, a bound on the spherical excess by the cylindrical excess. Later, in Lemma \ref{sec7.1} and Theorem \ref{sec7.3} we must essentially bound the cylindrical excess by the spherical excess.

\section{First variation and monotonicity} \label{sec2}

This is section 2 of \cite{HS79}, with mostly only changes in presentation. Monotonicity formulas are computed in this section, via the first variation. We introduce the constants $c_{1},\ldots,c_{5}.$

\medskip

Throughout this section (and after) we shall write $\sT = \sT(M,m,\alpha)$ with $M \in \bbN,$ $m \in \{ 0,\ldots,M-1 \},$ and $\alpha \in (0,1].$ Recall that we need $m \geq 1$ only in Lemma \ref{sec6.4}, Theorem \ref{sec7.3}, section \ref{sec8}, Theorem \ref{sec9.1}, section \ref{proofmain}, and in \eqref{appendixlemma2projectionmass}.

\medskip

Starting with this section, if the reader subtracts two to the section number, then one gets the counterpart section, subsection, lemma, theorem, and equation of \cite{HS79}.

\subsection{First variation} \label{sec2.1} 

The following first variation formula is well-known; we sketch the proof. We thus give a more precise version of the first variation formula given in section 2.1 of \cite{HS79}.

\begin{lemma} For any $T \in \sT,$ there is a $\mu_{\partial T}$-measurable vector field $\nu_{T}:(B^{n}_{2} \times \bbR) \rightarrow \bbR^{n+1}$ with $\nu_{T}(x) \perp T_{x} \partial T$ and $|\nu_{T}(x)| \leq 1$ for $\mu_{\partial T}$-almost-every $x \in B^{n}_{2} \times \bbR$ such that
$$\int \dive{T}{X} \ d \mu_{T} = \int \nu_{T} \cdot X \ d \mu_{\partial T}$$
for every $X \in C_{c}^{1}(B^{n}_{2} \times \bbR).$
\end{lemma}

Note that (as opposed to \cite{HS79}) we have used $\cV^{T}$ to denote the generalized unit normal vector field of $T,$ given by $\cV^{T} = \ast \vec{T}.$ In this lemma and here throughout, $\nu_{T}$ denotes the generalized co-normal of $\partial T$ with respect to $T$. The proof here is taken from the proof of (2.10) of \cite{E89}.

\medskip

\begin{proof} Examining the proof of Lemma 3.1 of \cite{B76}, we conclude $|\int \dive{T}{X} \ d \mu_{T}| \leq \int |X \wedge \vec{\partial T}| \ d \mu_{\partial T}$; to see this, note that in our setting we may take $\lambda=0$ in the proof of Lemma 3.1 of \cite{B76}. This together with section 39 of \cite{S83} implies the lemma. \end{proof}

\subsection{Monotonicity} \label{sec2.2} 

We give the monotonicity formulas of section 2.2 of \cite{HS79}. The conclusions bear only changes in notation.

\begin{lemma} \label{2.2} There are constants $c_{1},c_{2},c_{3},c_{4}$ depending on $n,M,m$ so that for any $T \in \sT,$ $0 < r \leq s < 2,$ and $\varkappa \in (0,1)$
\begin{equation} \label{2.2(1)}
\frac{\mu_{T}(\bB_{r})}{r^{n}} e^{c_{1} \kappa_{T} r^{\alpha}} \text{ is nondecreasing in } r,
\end{equation}
\begin{equation} \label{2.2(2)}
\frac{1}{c_{2}} \leq \frac{\mu_{T}(\bB_{r})}{r^{n}} \leq c_{2},
\end{equation}
\begin{equation} \label{2.2(3)}
\int_{\bB_{2}} \frac{1}{|x|^{(1-\varkappa)n}} d \mu_{T}(x) \leq \frac{c_{3}}{\varkappa},
\end{equation}
\begin{equation} \label{2.2(4)}
\begin{aligned}
\left| \frac{\mu_{T}(\bB_{s})}{s^{n}} - \frac{\mu_{T}(\bB_{r})}{r^{n}} \right. & \left. - \int_{\bB_{s} \setminus \bB_{r}} \frac{| x \cdot \cV^{T}(x)|^{2}}{|x|^{n+2}} \ d \mu_{T}(x) \right| \\
& \leq \frac{c_{4} \kappa_{T}}{2} \left[ (s^{\alpha}-r^{\alpha}) + \alpha \left(1- \left( \frac{r}{s} \right)^{n} \right) r^{\alpha} \right],
\end{aligned}
\end{equation}
\begin{equation} \label{2.2(5)}
\left| \bE_{S}(T,s) - \int_{\bB_{s}} \frac{|x \cdot \cV^{T}(x)|^{2}}{|x|^{n+2}} \ d \mu_{T}(x) \right| \leq c_{4} \kappa_{T}.
\end{equation}
\end{lemma}

These are respectively formulas 2.2(1)-(5) of \cite{HS79}, but where $c_{1},c_{2},c_{3},c_{4}$ now depend on $n,M,m.$ Observe that our use of $\varkappa \in (0,1)$ in \eqref{2.2(3)} differs from that of $\beta \in (0,n)$ as used in 2.2(3) of \cite{HS79}; in fact, we won't be needing \eqref{2.2(3)}, but present it for potential future use, and to keep our references to equations as close to \cite{HS79} as feasible.

\medskip

\begin{proof} We first prove \eqref{2.2(4)}. Consider the Lipschitz vector field
$$X(x) = \left\{ 
\begin{array}{cl}  
(r^{-n}-s^{-n}) x & \text{ for } |x| \leq r, \\
(|x|^{-n}-s^{-n}) x & \text{ for } r < |x| \leq s, \\
0 & \text{ for } s < |x|,
\end{array} \right.$$
and recall that by the previous section we have $\int \dive{T}{X} \ d \mu_{T} = \int \nu_{T} \cdot X \ d \mu_{\partial T}.$  Using \eqref{defboundary} we conclude
$$\begin{aligned}
\left| \int \dive{T}{X} \ d \mu_{T} \right| & \leq \int |X \wedge \vec{\partial T}| \ d \mu_{\partial T} \\
& = \sum_{\ell=1}^{N} m_{\ell} \int_{\Phi_{T,\ell}(B^{n-1}_{2})} |X \wedge \vec{\partial T}| \ d \cH^{n-1}
\end{aligned}$$
Furthermore, for each $\ell=1,\ldots,N$ we have by \eqref{defderivative} that for $\cH^{n-1}$-almost-every $x \in \Phi_{T,\ell}(B^{n-1}_{2})$ (as in 2.2(6) of \cite{HS79}) 
\begin{equation} \label{2.2(6)}
\begin{aligned}
& |x| \leq (1+\kappa_{T})|(x_{1},\ldots,x_{n-1})| \leq 2 |(x_{1},\ldots,x_{n-1})|, \\
& |x \wedge \vec{\partial T}(x)| \leq c_{5} \kappa \alpha |(x_{1},\ldots,x_{n-1})|^{1+\alpha},
\end{aligned}
\end{equation}
where $c_{5}$ depends in fact only on $n.$ Thus, \eqref{2.2(4)} follows by integrating with respect to $(x_{1},\ldots,x_{n-1})$, noting that
$$\begin{aligned}
\int_{r}^{s} (t^{-n}-s^{-n}) t^{n-1+\alpha} \ dt & \leq \alpha^{-1} (s^{\alpha}-r^{\alpha}) \\
(r^{-n}-s^{-n}) \int_{0}^{r} t^{n-1+\alpha} \ dt & \leq (1-(r/s)^{n}) r^{\alpha},
\end{aligned}$$
and recalling by \eqref{defboundary} that $\sum_{\ell=1}^{N} m_{\ell} = M-m$; from this we get $c_{4}$ depending on $n,M,m.$

\medskip

Second, define $L:\bbR^{n+1} \rightarrow \bbR^{n}$ (as in \cite{HS79}) by 
$$L(x) = (x_{1},\ldots,x_{n-1},|(x_{n},x_{n+1})|) \text{ for } x \in \bbR^{n+1}.$$
Note that $\Lip(L)=1$ and that $L_{\sharp} (T \res B_{2}) \neq 0$ because
$$\big( \partial L_{\sharp}(T \res B_{2}) \big) \res B^{n}_{2} = L_{\sharp} (\partial T \res B_{2}) \neq 0.$$
From \eqref{2.2(6)} and 4.1.31 of \cite{F69} we infer that for $r \in (0,2)$
$$\begin{aligned}
r^{-n} \mu_{T}(\bB_{r}) & \geq r^{-n} \mu_{L_{\sharp} T}(\bB_{r}) \\
& \geq r^{-n} \bM ( \bbE^{n} \res \{ y \in \bB^{n}_{r}: 2|(y_{1},\ldots,y_{n-1})| \leq |y|, \ y_{n}>0 \}) \\
& = \cH^{n} (\{ y \in \bB^{n}_{r}: 2|(y_{1},\ldots,y_{n-1})| \leq |y|, \ y_{n}>0 \}).
\end{aligned}$$
This along with \eqref{2.2(4)} implies the first inequality in \eqref{2.2(2)}. Letting $\phi(r) = r^{-n} \mu_{T}(\bB_{r}),$ we deduce from this inequality and \eqref{2.2(4)} that
$$(n+2) c_{2}c_{4} \kappa_{T} \alpha r^{\alpha-1} \phi(r) + \liminf_{s \searrow r} \frac{\phi(s)-\phi(r)}{s-r} \geq 0,$$
which implies \eqref{2.2(1)} with $c_{1}=(n+2) c_{2}c_{4}.$ 

\medskip

By increasing $c_{2}$ if necessary, the second inequality in \eqref{2.2(2)} now follows from \eqref{2.2(1)} and \eqref{defmass}; we can also show \eqref{2.2(3)} from \eqref{2.2(2)}. Meanwhile, \eqref{2.2(5)} is verified by letting $r \searrow 0$ in \eqref{2.2(4)}. \end{proof}

\subsection{Remark} \label{sec2.3}

We remark here as in section 2.3 of \cite{HS79}. If $T \in \sT$ and $\lambda \in (0,1/3],$ then \eqref{2.2(4)} applied with $r = 3 \lambda$ and $s=1$, \eqref{defmass}, and \eqref{1.6(1)} imply
$$\begin{aligned}
\bM( (\ETA_{\lambda \sharp} T) \res \bB_{3}) = & 3^{n} \big( (3 \lambda)^{-n} \mu_{T}(\bB_{3 \lambda}) \big) \\
\leq & 3^{n} \big( \mu_{T}(\bB_{1}) +c_{4} \kappa_{T} \big) \\
< & 3^{n} \left( \bE_{C}(T,1) + \left( \frac{M+m}{2} \right) \varpi_{n} \right) \\
& + 3^{n} ((M-m) \varpi_{n-1} + c_{4}) \kappa_{T}.
\end{aligned}$$
We can hence choose $c_{4}$ depending on $n,M,m$ so that $\bE_{C}(T,1) + \kappa_{T} \leq (1+c_{4})^{-1}$ implies $\bM( (\ETA_{\lambda \sharp}T) \res \bB_{3}) < 3^{n} (1+ M \varpi_{n}).$ Thus
\begin{equation} \label{2.3(1)}
(\ETA_{\lambda \sharp} T) \res B_{3} \in \sT \text{ if } \bE_{C}(T,1)+\kappa_{T} \leq (1+c_{4})^{-1} \text{ and } \lambda \in (0,1/3],
\end{equation}
moreover with
\begin{equation} \label{2.3(2)}
\kappa_{(\ETA_{\lambda \sharp} T) \res B_{3}} \leq \lambda^{\alpha} \kappa_{T}.
\end{equation}
These are 2.3(1)(2) of \cite{HS79}, with only differences in notation. These equations will be used to iteratively apply results for $T \in \sT$ to rescalings of $T,$ most notable in the proof of Lemma \ref{sec7.1}.

\section{An area comparison lemma} \label{sec3}

\medskip

The results of this section shall be used in the next to conclude preliminary bounds on the excess. This section is analogous to section 3 of \cite{HS79}. However, we must make a serious change to section 3.2 of \cite{HS79}. 

\medskip

The last constant introduced in section \ref{sec2} was $c_{5},$ analogous to the last constant $c_{5}$ introduced in section 2 of \cite{HS79}. We note that the first constant introduced in section 3 of \cite{HS79} is $c_{7}$; in other words, the constant $c_{6}$ is mistakenly skipped. To make it easier for the reader to compare our current work to \cite{HS79}, we as well skip $c_{6}$ and introduce $c_{7},c_{8},c_{9}.$ Also, while section 3 of \cite{HS79} also introduced $c_{10}$ and $c_{11},$ we will not be needing them.

\subsection{Remark} \label{sec3.1}

A general fact about exterior algebras is stated, as in section 3.1 of \cite{HS79}. 

\medskip

If $F \in C^{1}(\bbR^{n+1};\bbR^{n+1})$ and $T \in \sT,$ then for $\mu_{T}$-almost-every $x \in \bbR^{n+1}$
$$\ast \big( (\wedge_{n} DF)(x) \vec{T}(x) \big) = \cV^{T}(x) \Delta (x)$$
where $\Delta(x)$ is the $(n+1) \times (n+1)$ matrix with $(i,j)^{\text{th}}$ entry $(-1)^{i+j}$ times the determinant of the $n \times n$ matrix obtained by deleting the $i^{\text{th}}$ row and $j^{\text{th}}$ column from the $(n+1) \times (n+1)$ matrix with $i^{\text{th}}$ row $D_{i}F(x).$

\subsection{Lemma} \label{sec3.2} 

We give a lemma analogous to that of section 3.2 of \cite{HS79}. However, the conclusion of our lemma is slightly different. As such, we must take care that any applications of Lemma \ref{sec5.2} still follow through as for Lemma 3.2 of \cite{HS79}. Recall that $\boq_{1}:\bbR^{n+1} \rightarrow \bbR$ is given by $\boq_{1}(x)=x_{n+1}$ for $x \in \bbR^{n+1}.$ This lemma uses the first variation Lemma \ref{sec2.1} with ``vertical'' deformation vector fields, to estimate the change in the mass of $T$ when pushforwarded by certain maps $F:\bbR^{n+1} \rightarrow \bbR^{n+1}$ fixing cylinders $C_{\rho}.$ 

\begin{lemma}
There are constants $c_{7},c_{8}$ depending on $n,M,m$ such that if $T \in \sT,$ $\rho \in (0,\infty),$ $\tau \in (0,1),$ $A \subseteq C_{1+\tau}$ is a Borel set satisfying $A = \op^{-1}[\op(A)],$ and $\phi \in C^{1}(\bbR^{n};[0,1])$ with $\sup_{\op(A)} |D \phi| \leq \rho/\tau,$ then 
$$\bM (F_{\sharp}(T \res A)) - \bM (T \res A) \leq c_{7} \kappa_{T} + \frac{c_{8}(1+\rho^{2})}{\tau^{2}} \int_{A_{\tau}} \boq_{1}^{2} \ d \mu_{T},$$
where $F \in C^{1}(\bbR^{n+1};\bbR^{n+1})$ is given by
$$F(x) = (\bop(x),\phi(\bop(x)) x_{n+1}),$$
and where $A_{\tau} = \{ x \in \bbR^{n+1} : \dist (x,A) < \tau \}.$
\end{lemma}

\begin{proof} Letting $\Delta$ be given as in Remark \ref{sec3.1}, we compute
$$|\cV^{T} \Delta|^{2} = \sum_{i=1}^{n} \big( \cV^{T}_{i} \phi - \boq_{1} \cV^{T}_{i} D_{i} \phi \big)^{2} + (\cV^{T}_{n+1})^{2}$$
$\mu_{T}$-almost-everywhere. Using 4.1.30 of \cite{F69}, Remark \ref{sec3.1}, the inequalities
$$\phi^{2} \leq 1, \ |\cV^{T}_{n+1}| \leq 1, \ \sum_{i=1}^{n} (\cV^{T}_{i})^{2}=1-(\cV^{T}_{n+1})^{2},$$
Schwartz's inequality, and Cauchy's inequality ($2ab \leq a^{2}+b^{2}$), we estimate
\begin{equation} \label{3.2(1)}
\begin{aligned}
\bM & (F_{\sharp} (T \res A)) - \bM (T \res A) \\
\leq & \int_{A} |(\wedge_{n} DF)(x) \vec{T}(x)| \ d \mu_{T}(x) - \mu_{T}(A) \\
= & \int_{A} (|\cV^{T} \Delta|-1) \ d \mu_{T} = \int_{A} \frac{|\cV^{T} \Delta|^{2}-1}{|\cV^{T} \Delta|+1} \ d \mu_{T} \\
\leq & \int_{A} \Big| \sum_{i=1}^{n} (\cV^{T}_{i})^{2} \phi^{2}- \big( 1-(\cV^{T}_{n+1})^{2} \big) \Big| d \mu_{T} \\
& + \int_{A} 2 \Big| \boq_{1} \phi \cV^{T}_{n+1} \sum_{i=1}^{n} \cV^{T}_{i} D_{i} \phi \Big| + ( \boq_{1} \cV^{T}_{n+1})^{2} |D\phi|^{2} \ d\mu_{T} \\
\leq & \int_{A} \big( 1-(\cV^{T}_{n+1})^{2} \big) + 2|\boq_{1}| \sqrt{1-(\cV^{T}_{n+1})^{2}} \ |D\phi| + \boq_{1}^{2} |D\phi|^{2} \ d \mu_{T} \\
\leq & 2 \int_{A} \big( 1-(\cV^{T}_{n+1})^{2} \big) \ d \mu_{T} + 2 \int_{A} \boq_{1}^{2} |D\phi|^{2} \ d \mu_{T};
\end{aligned}
\end{equation}
this is 3.2(1) of \cite{HS79}. The second term is $\leq \frac{2 \rho^{2}}{\tau^{2}} \int_{A_{\tau}} \boq_{1}^{2} \ d \mu_{T}$ as is needed.

\medskip

Consider the first variation formula from Lemma \ref{sec2.1} with the vector field $X(x) =  \zeta(x)^{2} \boq_{1}(x) e_{n+1}$ where $\zeta \in C^{1}(\bbR^{n+1};[0,1])$ satisfies $\spt \zeta \subset A_{\tau},$ $\zeta|_{A}=1,$ and $\sup |D\zeta| \leq \frac{c_{9}}{\tau}$ with $c_{9}=c_{9}(n).$ We can compute
$$\begin{aligned} 
\int & \big( 1-(\cV^{T}_{n+1})^{2} \big) \zeta^{2} \ d \mu_{T} \\ 
& = \int (-2 \zeta \boq_{1} \nabla^{T} \zeta) \cdot e_{n+1} \ d \mu_{T} + \int \zeta^{2} \boq_{1} (e_{n+1} \cdot \nu_{T}) \ d \mu_{\partial T} \\ 
& \leq \int \frac{1}{2} \big( 1-(\cV^{T}_{n+1})^{2} \big) + 2 |D\zeta|^{2} \boq_{1}^{2} \ d \mu_{T} + \int \zeta^{2} |\boq_{1}| \ d \mu_{\partial T},
\end{aligned}$$ 
where recall that $\nu_{T}$ is the co-normal of $\partial T$ with respect to $T.$ Thus
$$\int \big( 1-(\cV^{T}_{n+1})^{2} \big) \zeta^{2} \ d \mu_{T} \leq \frac{4 c_{9}^{2}}{\tau^{2}} \int_{A_{\tau}} \boq_{1}^{2} \ d \mu_{T} + 2 \int \zeta^{2} |\boq_{1}| \ d \mu_{\partial T}.$$
We can also compute
$$\begin{aligned} 
\int \zeta^{2} |\boq_{1}| \ d \mu_{\partial T} & \leq \int_{C_{1+2\tau}} |\boq_{1}| \ d \mu_{\partial T} \\ 
& \leq \left(\frac{\alpha}{2}\right) \kappa_{T} (1+2\tau)^{1+\alpha} \mu_{\partial T}(C_{1+2\tau}) \\
& \leq 2^{-1} 3^{n+1} \left( 1+\frac{3^{2}}{4}+\frac{3^{4}}{16} \right)^{\frac{1}{2}} (M-m) \varpi_{n-1} \kappa_{T},
\end{aligned}$$ 
using \eqref{defboundary},\eqref{defderivative}, and $\alpha,\kappa_{T},\tau \in (0,1]$ as well as
$$\begin{aligned}
\mu_{\partial T} & (C_{1+2\tau}) \\
\leq & (M-m) \varpi_{n-1} \left( 1+ \frac{\alpha^{2} \kappa_{T}^{2}}{4} (1+2\tau)^{2 \alpha} + \frac{\alpha^{4} \kappa_{T}^{4}}{16} (1+2\tau)^{4\alpha} \right)^{1/2} \\
& \times (1+2\tau)^{n-1}.
\end{aligned}$$
We conclude the lemma with $c_{7}$ depending on $n,M,m$ and $c_{8}$ actually just depending on $c_{9}$ (and hence, only on $n$). \end{proof}

\section{Some preliminary bounds on excess} \label{sec4}

This section compares the cylindrical excess to the height excess, using subharmonicity while referring to either Theorem 7.5(6) of \cite{A72} or Theorem 3.4 of \cite{MS73}. The proofs and results here are the same as in section 4 of \cite{HS79}, although we make a slight clarification to the proof found in section 4.1 of \cite{HS79}, given here in section \ref{sec4.1}. In this section we introduce $c_{12},\ldots,c_{16}$; in the previous section we introduced $c_{7},c_{8},c_{9}$ while skipping $c_{10},c_{11},$ which as opposed to \cite{HS79} we did not need.

\subsection{Lemma} \label{sec4.1} 

\begin{lemma}
There are positive constants $c_{12},c_{13},c_{14},c_{15} \geq 1$ depending only on $n,M,m$ so that for all $\sigma \in (0,1)$ and $T \in \sT,$
\begin{equation} \label{4.1(1)}
\begin{aligned}
c_{12}^{-1} \sigma^{2} \bE_{C}(T,1) - \kappa_{T} & \leq \int_{C_{1+\sigma}} \boq_{1}^{2} \ d \mu_{T} \\
& \leq c_{13} \sup_{C_{1+\sigma} \cap \spt T} \boq_{1}^{2};
\end{aligned}
\end{equation}
\begin{equation} \label{4.1(2)}
\begin{aligned}
c_{14}^{-1} \sigma^{n} \sup_{C_{1-\sigma} \cap \spt T} \boq_{1}^{2} - \kappa_{T}^{2} & \leq \int_{C_{1-\frac{\sigma}{2}}} \boq_{1}^{2} \ d \mu_{T} \\
& \leq c_{15} \sigma^{-n-1} \left( \bE_{C}(T,1)+\kappa_{T} \right).
\end{aligned}
\end{equation}
\end{lemma}

The equations \eqref{4.1(1)},\eqref{4.1(2)} are 4.1(1)(2) of \cite{HS79}. We as well remark that this lemma means that $\bE_{C}(T,1)$ can be bounded above and below by fixed multiples of $\sup_{C_{1+\sigma} \cap \spt T} \boq_{1}^{2}$ and $\sup_{C_{1-\sigma} \cap \spt T} \boq_{1}^{2},$ respectively, up to a term which is a fixed multiple of $\kappa_{T}.$ As remarked in \cite{DS93}, we will repeatedly need to estimate the cylindrical excess by the height. The first equation \eqref{4.1(1)} follows by the area comparison Lemma \ref{sec3.2} and a comparison argument based on the minimality of $T,$ while the proof of \eqref{4.1(2)} uses Allard's technique of Moser interation; see Theorem 7.5(6) of \cite{A72}.

\medskip

\begin{proof} The second inequality in \eqref{4.1(1)} follows immediately from \eqref{defmass} with $c_{13} = 3^{n}(1+M \varpi_{n}).$ To prove the first inequality in \eqref{4.1(1)}, first observe that $\kappa_{T} > 3^{n} \big( 1+M \varpi_{n} \big) \sigma^{2}$ implies by \eqref{defmass} that 
$$c_{12}^{-1} \sigma^{2} \bE_{C}(T,1) - \kappa_{T} < 0$$
so long as we choose $c_{12} \geq 1.$ 

\medskip

Now assume $\kappa_{T} \leq 3^{n} \big( 1+M \varpi_{n} \big) \sigma^{2}.$ With $\tau = \sigma/2$ let $\phi,$ $F,$ $h,$ and $R_{T}$ be as in Lemma \ref{appendixlemma3}. Using the homotopy formula (see 4.1.8-9 of \cite{F69} or 26.22 of \cite{S83}) we can thus compute
$$\partial \big( (T \res C_{1+\tau}) - F_{\sharp}(T \res C_{1+\tau}) - R_{T} \big) = \partial \big( T-F_{\sharp}T-R_{T} \big) = 0.$$
Since $T$ is area-minimizing, then this implies
$$\bM(T \res C_{1+\tau}) \leq \bM(F_{\sharp}(T \res C_{1+\tau})+R_{T}).$$
Since $F(x) = \bop(x)$ for $x \in C_{1},$ then letting $A = C_{1+\tau} \setminus C_{1}$ we compute
$$\begin{aligned} 
\bE_{C}(T,1) & = \bM(T \res C_{1}) - \bM(F_{\sharp}(T \res C_{1})) \\
& \leq \bM(F_{\sharp}(T \res C_{1+\tau})+R_{T}) - \bM(T \res A) - \bM(F_{\sharp}(T \res C_{1})) \\
& \leq \bM(F_{\sharp}(T \res A)) - \bM(T \res A) + \bM(R_{T}).
\end{aligned}$$ 
Using \eqref{appendixlemma3RT} (since $\kappa_{T} \leq 4 \cdot 3^{n} (1+M \varpi_{n}) \tau^{2}$), Lemma \ref{sec3.2} (with $\rho = 3$ and $A = C_{1+\tau} \setminus C_{1}$), and $\tau= \sigma/2,$ gives 
$$\begin{aligned} 
\bE_{C}(T,1) \leq & c_{7} \kappa_{T} + \frac{40c_{8}}{\sigma^{2}} \int_{C_{1+\sigma}} \boq_{1}^{2} \ d \mu_{T} \\ 
& + \left( \frac{\sqrt{21}}{8}+ 2^{\frac{9n-7}{2}} 3^{n^{2}-\frac{1}{2}} \right) (M-m) \varpi_{n-1} (1+M \varpi_{n})^{n-1} \kappa_{T}. 
\end{aligned}$$ 
We can then choose $c_{12} \geq 1,$ depending on $n,M,m$ so that the first inequality of \eqref{4.1(1)} holds.  

\medskip

We now prove the first inequality in \eqref{4.1(2)}. For this, we will show the function $\max \{ \boq_{1}-\kappa_{T},0 \}^{2}$ is $T$-subharmonic, in the sense that
$$\int \nabla^{T} (\max \{ \boq_{1}-\kappa_{T},0 \}^{2}) \cdot \nabla^{T} \zeta \ d \mu_{T} \leq 0$$
for all $\zeta \in C^{1}_{c}(\bbR^{n+1} \setminus \spt \partial T;[0,\infty)).$ In fact, by the first variation formula \ref{sec2.1}, for any such $\zeta$
$$0 = \int \dive{T}{(\zeta e_{n+1})} \ d \mu_{T} = \int e_{n+1} \cdot \nabla^{T} \zeta \ d \mu_{T} = \int \nabla^{T} \boq_{1} \cdot \nabla^{T} \zeta \ d \mu_{T}.$$
Since $\max \{t-\kappa_{T},0 \}^{2}$ is a nondecreasing convex function and $\boq_{1}$ is $T$-harmonic, then by Lemma 7.5(3) of \cite{A72} the function $\max \{ \boq_{1}-\kappa_{T},0 \}^{2}$ is $T$-subharmonic.

\medskip

Having shown $\max \{ \boq_{1}-\kappa_{T},0 \}^{2}$ is subharmonic, we now use the mean value theorem for subharmonic functions; for this, apply either the argument of Theorem 7.5(6) of \cite{A72} (to the varifold associated with $T$, see Example 4.8(4) of \cite{A72}), or the argument of Theorem 3.4 of \cite{MS73} (with $\tilde{g}^{ij} = \delta_{ij} - \cV^{T}_{i} \cV^{T}_{j},$ $\mu = \mu_{T},$ and $M=U=B^{n}_{2}(0) \times \bbR$). We deduce the bound
\begin{equation} \label{4.1(3)}
\begin{aligned}
\sup_{C_{1-\sigma} \cap \spt T} & \max \{ \boq_{1}-\kappa_{T},0 \}^{2} \\
= & \sup_{x \in \bop^{-1}(\{0\})} \sup_{B_{1-\sigma}(x) \cap \spt T} \max \{ \boq_{1}-\kappa_{T},0 \}^{2} \\
\leq & \sup_{x \in \bop^{-1}(\{0\})} \frac{c_{16}}{\sigma^{n}} \int_{B_{1-\frac{\sigma}{2}}(x)} \max \{ \boq_{1}-\kappa_{T},0 \}^{2} \ d \mu_{T} \\
\leq & \frac{c_{16}}{\sigma^{n}} \int_{C_{1-\frac{\sigma}{2}}} \max \{ \boq_{1}-\kappa_{T},0 \}^{2} \ d \mu_{T}
\end{aligned}
\end{equation}
where $c_{16}=c_{16}(n) \geq 1;$ this is 4.1(3) of \cite{HS79}. We similarly verify
$$\sup_{C_{1-\sigma} \cap \spt T} \min \{ -\kappa_{T}-\boq_{1},0 \}^{2} \leq \frac{c_{16}}{\sigma^{n}} \int_{C_{1-\frac{\sigma}{2}}} \min \{ -\kappa_{T}-\boq_{1},0 \}^{2} \ d \mu_{T}.$$
Combining the latter two estimates with Cauchy's inequality ($2ab \leq a^{2}+b^{2}$) gives the first inequality in \eqref{4.1(2)}. 

\medskip

To prove the second inequality in \eqref{4.1(2)}, we follow \cite{HS79} but with small changes in constants. In particular, we now let
$$c_{15} = 16 \cdot 3^{3n+3} (3+c_{4}+ (M-m) \varpi_{n-1})(1+M \varpi_{n}) \max\{ 1,\varpi_{n}^{-1} \} c_{16}.$$
We presently may assume 
\begin{equation} \label{4.1(4)} 
\bE_{C}(T,1) + \kappa_{T} < 3^{n+2} (1+M \varpi_{n}) c_{15}^{-1} \sigma^{n+1}, 
\end{equation} 
analogous to 4.1(4) of \cite{HS79}. Otherwise using \eqref{defmass} we get
$$\int_{C_{1-\frac{\sigma}{2}}} \boq_{1}^{2} \ d \mu_{T} \leq 3^{n+2}(1+M \varpi_{n}) \leq c_{15} \sigma^{-n-1} (\bE_{C}(T,1)+\kappa_{T}),$$
which is the second inequality in \eqref{4.1(2)}. Assuming \eqref{4.1(4)}, then we use \eqref{2.2(5)}, \eqref{1.4(1)}, \eqref{1.6(1)} (and assuming $c_{4} \geq 3$) to get
\begin{equation} \label{4.1(5)} 
\begin{aligned} 
\int_{\bB_{1}} \boq_{1}^{2} \ d \mu_{T} & = \int_{\bB_{1}} \left( x \cdot \cV^{T}(x) + x \cdot (e_{n+1}-\cV^{T}(x)) \right)^{2} \ d \mu_{T}(x) \\
& \leq 2 \int_{\bB_{1}} (x \cdot \cV^{T}(x))^{2} + |e_{n+1}-\cV^{T}(x)|^{2} \ d \mu_{T}(x) \\
& \leq 2 \bE_{S}(T,1) + 2 c_{4} \kappa_{T} + 8 \bE_{C}(T,1) \\ 
& \leq 2 (2+c_{4}+(M-m) \varpi_{n-1})(\bE_{C}(T,1) + \kappa_{T}),
\end{aligned} 
\end{equation} 
as in 4.1(5) of \cite{HS79}. Recalling the definition of $c_{15}$ and $\sigma \in (0,1)$ we conclude
$$\int_{\bB_{1}} \boq_{1}^{2} \ d \mu_{T} \leq c_{15} \sigma^{-n-1} (\bE_{C}(T,1)+\kappa_{T}).$$ 
As in \cite{HS79}, we conclude the proof by showing
\begin{equation} \label{4.1(6)}
C_{1-\frac{\sigma}{2}} \cap \spt T \subset \bB_{1}, 
\end{equation}
as in 4.1(6) of \cite{HS79}. For this, by the mean value theorem for subharmonic functions as used in \eqref{4.1(3)}, Cauchy's inequality, \eqref{4.1(5)}, and \eqref{4.1(4)} we have
$$\begin{aligned} 
\sup_{\bB_{1-\frac{\sigma}{6}} \cap \spt T} \boq_{1}^{2} & \leq 2 \cdot 6^{n} c_{16} \sigma^{-n} \int_{\bB_{1}} \boq_{1}^{2} \ d \mu_{T} + 4 \kappa_{T}^{2} \\ 
& \leq 4 \cdot 6^{n} (3+c_{4}+(M-m) \varpi_{n-1}) c_{16} \sigma^{-n} (\bE_{C}(T,1)+\kappa_{T}) \\ 
& \leq \sigma/12. 
\end{aligned}$$ 
This together with \eqref{4.1(4)} and \eqref{defderivative} implies
$$\partial (T \res B_{1-\frac{\sigma}{6}}) \res C_{1-\frac{\sigma}{3}} = (\partial T) \res C_{1-\frac{\sigma}{3}}.$$
Recalling \eqref{4.1(4)} (and the definition of $c_{15}$), we additionally conclude  
$$\bop_{\sharp} \big( (T \res B_{1-\frac{\sigma}{6}}) \res C_{1-\frac{\sigma}{3}} \big) = \bop_{\sharp}(T \res C_{1-\frac{\sigma}{3}})$$
using Lemma \ref{appendixlemma4}. This together with \eqref{1.4(1)}, the interior monotonicity formula (see 5.4.5(2) of \cite{F69} or Theorem 17.6 of \cite{S83}), and \eqref{4.1(4)} imply that for any $x \in (C_{1-\frac{\sigma}{2}} \setminus \bB_{1}) \cap \spt T$
$$\begin{aligned}
\bE_{C}(T,1) \geq & (1-\sigma/3) \bE_{C}(T,1-\sigma/3) \\
\geq & (2/3)^{n} \mu_{T}(B_{\frac{\sigma}{6}}(x)) \\
\geq & (2/3)^{n} \varpi_{n} (\sigma/6)^{n} \geq 3^{-2n}  \varpi_{n} \sigma^{n+1} > \bE_{C}(T,1).
\end{aligned}$$
As this is a contradiction, we must have $(C_{1-\frac{\sigma}{2}} \setminus \bB_{1}) \cap \spt T = \emptyset.$ \end{proof}

\medskip

\subsection{Remark} \label{sec4.2} 

Applying \eqref{4.1(2)} with $\sigma = \frac{1}{4}$ gives
$$\sup_{C_{\frac{3}{4}} \cap \spt T} |\boq_{1}| \leq \frac{1}{8}$$
whenever $T \in \sT$ and $\bE_{C}(T,1)+\kappa_{T} \leq 4^{-2n-4} c_{15}^{-1}(1+c_{14})^{-1}.$ Recall the function $\rot_{\theta}$ as given in section \ref{sec1.1,1.2,1.3}, and observe that the sets
$$\{ x \in \bbR^{n+1}: |(x_{1},\ldots,x_{n-1})| \leq 1/2, |x_{n}| < 1/2 \}$$ 
$$\rot_{\theta}(\boq_{1}^{-1}([-1/8,1/8]) \cap \partial B_{3/4})$$
do not intersect whenever $|\theta| \leq 1/8.$ We conclude from Remark \ref{sec2.3}
\begin{equation} \label{4.2(1)}
\begin{aligned}
& (\ETA_{\frac{1}{4} \sharp} \rot_{\theta \sharp} T) \res B_{3} = (\rot_{\theta \sharp} \ETA_{\frac{1}{4} \sharp} T) \res B_{3} \in \sT \\
\text{whenever} & \ |\theta| \leq 1/8, \ T \in \sT, \text{ and} \\
& \bE_{C}(T,1)+\kappa_{T} \leq \min \left\{ \frac{1}{(1+c_{4})},\frac{1}{4^{2n+4}c_{15}(1+c_{14})} \right\};
\end{aligned}
\end{equation} 
this is 4.2(1) of \cite{HS79}. For such $\theta$ and $T,$ we may use \eqref{2.3(2)}, \eqref{4.1(1)} (with $\sigma \nearrow 1$), and \eqref{4.1(2)} (with $\sigma = \frac{1}{4}$) to estimate
\begin{equation} \label{4.2(2)}
\kappa_{(\ETA_{\frac{1}{4} \sharp} \rot_{\theta \sharp} T) \res B_{3}} = \kappa_{(\ETA_{\frac{1}{4} \sharp} T) \res B_{3}} \leq 4^{-\alpha} \kappa_{T},
\end{equation}
\begin{equation} \label{4.2(3)}
\begin{aligned}
\bE_{C} ((\ETA_{\frac{1}{4} \sharp} \rot_{\theta \sharp} T) \res B_{3},1) \leq & c_{12} c_{13} \sup_{C_{2} \cap \spt (\ETA_{\frac{1}{4} \sharp} \rot_{\theta \sharp} T)} \boq_{1}^{2} + c_{12} 4^{-\alpha} \kappa_{T} \\
\leq & 16 c_{12} c_{13} \sup_{C_{\frac{1}{2}} \cap \spt (\rot_{\theta \sharp} T)} \boq_{1}^{2} + c_{12} 4^{-\alpha} \kappa_{T} \\
\leq & 16 c_{12} c_{13} ( \theta^{2} + \sup_{C_{\frac{3}{4}} \cap \spt T} \boq_{1}^{2} ) + c_{12} 4^{-\alpha} \kappa_{T} \\
\leq & 4^{2n+4} c_{12}(1+c_{13})(1+c_{14})(1+c_{15}) \\
& \times (\theta^{2} + \bE_{C}(T,1) + \kappa_{T}),
\end{aligned}
\end{equation}
as in 4.2(2)(3) of \cite{HS79}. Finally, combining \eqref{4.2(1)},\eqref{4.2(2)},\eqref{4.2(3)} and Remark \ref{sec2.3} we conclude that
\begin{equation} \label{4.2(4)}
\begin{aligned}
& (\ETA_{\lambda \sharp} \rot_{\theta \sharp} T) \res B_{3} \in \sT \text{ with } \kappa_{(\ETA_{\lambda \sharp} \rot_{\theta \sharp} T) \res B_{3}} \leq \lambda^{\alpha} \kappa_{T} \\
\text{whenever} & \ T \in \sT, \ \lambda \in (0,1/12), \text{ and} \\
& \theta^{2}+ \bE_{C}(T,1)+ \kappa_{T} \leq c_{16}^{-1}
\end{aligned}
\end{equation}
where $c_{16} = c_{16}(n,M,m) = 4^{2n+5}(1+c_{4})(1+c_{12})(1+c_{13})(1+c_{14})(1+c_{15})$; this is as in 4.2(4) of \cite{HS79}.

\medskip

As remarked in \cite{DS93}, one typically uses Lemma \ref{sec4.1} to see that if $T \in \sT$ with $\bE_{C}(T,1)+\kappa_{T}$ small, then slightly tilting and rescaling $T$ yields another member of $\sT$ will small excess of order $\bE_{C}(T,1).$

\section{Interior nonparametric estimates} \label{sec5}

Section 5.1 of \cite{HS79} proves a general decomposition theorem, while sections 5.2,5.3 of \cite{HS79} state the well-known gradient estimates for solutions to the minimal surface equation. Section 5.4 of \cite{HS79}, which proves an approximate graphical decomposition for $T \in \sT$ with sufficiently small cylindrical excess, passes with no serious changes. Sections \ref{sec5.1}-\ref{sec5.4} are direct counterparts to sections 5.1-5.4 of \cite{HS79}, with only minor mostly notational changes. We introduce in this section $c_{17},\ldots,c_{25}.$

\subsection{Lemma} \label{sec5.1} 

The following lemma concerns general rectifiable currents, and so passes completely unchanged as in section 5.1 of \cite{HS79}. We will use this to prove Theorem \ref{sec5.4}, which states that $T$ with $\bE_{C}(T,1)+\kappa_{T}$ sufficiently small can be respectively decomposed into a sum of graphs over large regions of $\bV$ and $\bW.$

\begin{lemma}
If $V$ is an open subset of $\bbR^{n},$ $S \in \sR_{n}(\bbR^{n+1}),$ $S \res \bop^{-1}(V)=S,$ $(\partial S) \res \bop^{-1}(V)=0,$ $\mathfrak{m} \in \bbN,$ $\bop_{\sharp} S = \mathfrak{m} (\bbE^{n} \res V),$ and $\bM(S)-\bM(\bop_{\sharp}S) < \cH^{n}(V),$ then for each $i \in \{ 1,\ldots,\mathfrak{m} \}$ there exists $S_{i} \in \sR_{n}(\bbR^{n+1})$ so that
$$\bop^{-1}(V) \cap \spt \partial S_{k} = \emptyset, \ \bop_{\sharp} S_{k} = \bbE^{n} \res V, \ S = \sum_{i=1}^{\mathfrak{m}} S_{i}, \ \mu_{S} = \sum_{i=1}^{\mathfrak{m}} \mu_{S_{i}}.$$
\end{lemma}

\begin{proof} Choose $s > \sup_{\spt S} \boq_{1}$ and let $h:\bbR \times \bbR^{n+1} \rightarrow \bbR^{n+1}$ be given by
$$h(t,x) = (\bop(x),(1-t) s + t x_{n+1})$$
for $(t,x) \in \bbR \times \bbR^{n+1}.$ Then by 4.1.9,4.5.17 of \cite{F69} we can choose Lebesgue measurable sets $M_{i} \subset \bbR^{n+1},$ for each integer $i,$ so that $M_{i} \subset M_{i-1}$ and
$$\begin{aligned}  
\partial h_{\sharp} ( (\bbE^{1} \res [0,1]) \times S) & = \sum_{i=-\infty}^{\infty} \partial (\bbE^{n+1} \res M_{i}) \\
\mu_{\partial h_{\sharp} ( (\bbE^{1} \res [0,1]) \times S) } & = \sum_{i=-\infty}^{\infty} \mu_{\partial (\bbE^{n+1} \res M_{i})};
\end{aligned}$$
see 4.1.8 of \cite{F69} or Definition 26.16 of \cite{S83}. Defining 
$$R_{i} = \partial (\bbE^{n+1} \res M_{i}) \res \bop^{-1}(V) \cap \{ x \in \bbR^{n+1}: \boq_{1}(x)<s \},$$
where recall $\boq_{1}(x) = x_{n+1}$ for $x \in \bbR^{n+1},$ then
$$\begin{aligned}
\bop_{\sharp} R_{i} & = - \bop_{\sharp} \big( \partial ( \bbE^{n+1} \res M_{i}) \res \{ x \in \bbR^{n+1}: \boq_{1}(x) \geq s \} \big) \in \{ 0, \pm \bbE^{n} \res V \}, \\
S & = \sum_{i=-\infty}^{\infty} R_{i}, \ \mu_{S} = \sum_{i=-\infty}^{\infty} \mu_{R_{i}}.
\end{aligned}$$
Letting $I = \{ i: \bop_{\sharp} R_{i} \neq 0 \},$ we infer that $\# I = \mathfrak{m}$ and hence $\bop_{\sharp} R_{i} = \bbE^{n} \res V$ for $i \in I$, because
$$\begin{aligned}
\mathfrak{m} \cH^{n}(V) & = \bM ( \bop_{\sharp} S ) = \bM \left( \sum_{i \in I} \bop_{\sharp} R_{i} \right) \leq \sum_{i \in I} \bM( \bop_{\sharp} R_{i} ) \\
& = (\# I) \cH^{n}(V) \leq \sum_{i \in I} \bM(R_{i}) = \bM(S) < (\mathfrak{m}+1) \cH^{n}(V).
\end{aligned}$$
Moreover, $M_{i} \subset M_{i-1}$ for all $i$ implies $I = \{ \mathfrak{i}+1,\ldots,\mathfrak{i}+\mathfrak{m} \},$ where $\mathfrak{i} = (\inf_{i \in I} i)-1.$ Setting
$$S_{1} = \sum_{i=-\infty}^{\mathfrak{i}+1} R_{i}, \ S_{2}=R_{\mathfrak{i}+2}, \ldots, S_{\mathfrak{m}-1}=R_{\mathfrak{i}+\mathfrak{m}-1}, \ S_{\mathfrak{m}}=\sum_{i=\mathfrak{i}+\mathfrak{m}}^{\infty} R_{i},$$
we conclude the lemma. \end{proof}

\subsection{Remark} \label{sec5.2} 

We introduce standard $L^{2}$ gradient and DeGiorgi-Nash H\"older continuity estimates for uniformly elliptic partial differential equations. We introduce constants $c_{17},c_{18}$ as in \cite{HS79}, which in fact depend only on $n.$ These estimates are now well-known, but we give them again for convenience.

\begin{lemma} 
There exist constants $c_{17},c_{18}$ depending only on $n$ such that if $\lambda \in (0,1),$ $\rho \in (0,\infty),$ $y \in \bbR^{n},$ $a_{kl}: B^{n}_{\rho}(y) \rightarrow \bbR$ are Lebesgue measurable functions for each $k,l \in \{ 1,\ldots,n \}$ satisfying $a_{kl}=a_{lk},$
\begin{equation} \label{5.2(1)}
\lambda |\xi|^{2} \leq \sum_{k,l=1}^{n} a_{kl} \xi_{k} \xi_{l} \leq \lambda^{-1} |\xi|^{2} \text{ whenever } \xi=(\xi_{1},\ldots,\xi_{n}) \in \bbR^{n},
\end{equation}
and $u$ is a weak solution of
$$\sum_{k,l=1}^{n} D_{l} (a_{kl} D_{k} u) = 0$$
(with $u,Du$ locally square integrable) over $B^{n}_{\rho}(y),$ then
\begin{equation} \label{5.2(2)}
\int_{\bB^{n}_{\frac{3 \rho}{4}}(y)} |Du|^{2} \ d \cH^{n} \leq \frac{c_{17}}{\lambda^{4} \rho^{2}} \int_{B^{n}_{\rho}(y)} u^{2} \ d \cH^{n}, 
\end{equation}
\begin{equation} \label{5.2(3)}
\sup_{\bB^{n}_{\frac{\rho}{2}}(y)} u^{2} \leq \frac{c_{17}}{\lambda^{2n} \rho^{n}} \int_{B^{n}_{\frac{3 \rho}{4}}(y)} u^{2} \ d \cH^{n}, \text{ and }
\end{equation}
\begin{equation} \label{5.2(4)}
|u(w)-u(\tilde{w})| \leq c_{18} \Big( \sup_{B^{n}_{\frac{\rho}{2}}(y)} |u| \Big) \left( \frac{|w-\tilde{w}|}{\rho} \right)^{\varkappa} \text{ for } w,\tilde{w} \in \bB^{n}_{\frac{\rho}{4}}(y)
\end{equation}

where $\varkappa = \varkappa(n,\lambda) \in (0,1).$ 
\end{lemma}

\begin{proof} See Lemma 1,Theorem 1 of \cite{MJ60}. \end{proof}

\medskip

We remark that \eqref{5.2(1)}-\eqref{5.2(4)} are 5.2(1)-(4) of \cite{HS79}. In particular, inequality \eqref{5.2(4)} is the DeGiorgi-Nash estimate.

\subsection{Lemma} \label{sec5.3} 

This section, analogous to section 5.3 of \cite{HS79} introduces the well-known gradient estimates for solutions to the minimal surface equation. In \cite{HS79} these estimates are derived for the sake of completeness using the results of section 5.2 there in; we do the same. The constants $c_{19},c_{20},c_{21},c_{22},c_{23}$ are introduced in this section, which remain unchanged as in \cite{HS79}, depending only on $n.$ We give the following lemma, which is merely Remark 5.3 of \cite{HS79}.

\begin{lemma}
There exist constants $c_{19},c_{20} \geq 1$ depending on $n$ so that for any solution $v$ of the minimal surface equation on an open subset $\Omega$ of $\bbR^{n}$ and any point $y \in \Omega,$
\begin{equation} \label{5.3(1)}
\sup_{\bB^{n}_{\frac{\rho}{2}}(y)} |Dv| < c_{19} \left( \sup_{\Omega} |v|/\rho \right) \exp(c_{20} \sup_{\Omega} |v|/\rho)
\end{equation}
where $\rho = \dist (y,\partial \Omega).$
\end{lemma}

We note that \eqref{5.3(1)} is 5.3(1) of \cite{HS79}. For a more modern reference, see Theorem 16.5 of \cite{GT83}.

\medskip

\begin{proof} For $n=2$ this was proved in Theorem 1 of \cite{JS63}. In general, for $n \geq 2,$ we consider two cases.

\medskip

First, suppose $\sup_{\Omega} |v|/\rho \geq 1.$ Here, \eqref{5.3(1)} follows precisely from the main estimate of \cite{BDM69} (see also \cite{T72},\cite{S76}).

\medskip

Second, suppose $\sup_{\Omega} |v|/\rho < 1.$ Here the main estimate of \cite{BDM69} implies that $\sup_{\Omega} |Dv| \leq c_{21}$ for some $c_{21}=c_{21}(n).$ We may then choose an appropriate $\lambda = \lambda(c_{21}) \in (0,1)$ so that we may first apply \eqref{5.2(3)} with $u=D_{\tilde{k}}v,$ for each $\tilde{k} = 1,\ldots,n,$ and
$$a_{kl} = \frac{1}{\sqrt{1+|Dv|^{2}}} \left( \delta_{kl} - \frac{D_{k}vD_{l}v}{1+|Dv|^{2}} \right),$$
and second apply \eqref{5.2(2)} with $u=v$ and $a_{kl} = \frac{\delta_{kl}}{\sqrt{1+|Dv|^{2}}}$ to conclude that
$$\begin{aligned}
\sup_{\bB^{n}_{\frac{\rho}{2}}(y)} |Dv|^{2} & \leq \frac{c_{22}}{\rho^{n}} \int_{B_{\frac{3 \rho}{4}}(y)} |Dv|^{2} \ d \cH^{n} \\
& \leq \frac{c_{23}}{\rho^{n+2}} \int_{B^{n}_{\rho}(y)} |v|^{2} \ d \cH^{n} \leq c_{23} \varpi_{n} \sup_{\Omega} |v|^{2}/\rho^{2},
\end{aligned}$$
where $c_{22},c_{23}$ depend only on $n.$ \end{proof}

\subsection{Theorem} \label{sec5.4} 

We give a theorem as in section 5.4 of \cite{HS79}. Note that we make a small, technical correction in defining $\bV_{T}$ and $\bW_{T}$ below. Also, here we conclude the existence of functions $v^{T}_{1} \leq v^{T}_{2} \leq \ldots \leq v^{T}_{M}$ and $w^{T}_{1} \leq \ldots \leq w^{T}_{m}.$ As in \cite{HS79}, we introduce $c_{24},c_{25} \geq 1$ but now depending on $n,M,m.$ We make some clarifications and simplifications to the proof found in \cite{HS79}.

\begin{theorem}
If $m \geq 1,$ then there are constants $c_{24},c_{25} \geq 1$ depending on $n,M,m$ so that for any $T \in \sT$ with
$$\sigma_{T} = c_{24} (\bE_{C}(T,1)+\kappa_{T})^{\frac{1}{2n+3}} \leq \frac{1}{4},$$
then with 
$$\bV_{T} = B^{n}_{\frac{1}{4}} \cap \bV_{\sigma_{T}} \text{ and } \bW_{T} = B^{n}_{\frac{1}{4}} \cap \bW_{\sigma_{T}}$$
(recall $\bV_{\sigma},\bW_{\sigma}$ defined in section \ref{sec1.1,1.2,1.3}) we have
$$\begin{aligned}
\bop^{-1}(\bV_{T}) \cap \spt T & = \bigcup_{i=1}^{M} \graph{\bV_{T}}{v^{T}_{i}} \\
\bop^{-1}(\bW_{T}) \cap \spt T & = \bigcup_{j=1}^{m} \graph{\bW_{T}}{w^{T}_{j}} \\
\end{aligned}$$
for some analytic functions $v^{T}_{i} \in C^{\infty}(\bV_{T})$ and $w^{T}_{j} \in C^{\infty}(\bW_{T})$ satisfying the minimal surface equation, and such that 
$$v^{T}_{1} \leq v^{T}_{2} \leq \ldots \leq v^{T}_{M} \text{ and } w^{T}_{1} \leq w^{T}_{2} \leq \ldots \leq w^{T}_{m}.$$ 
Furthermore, for each $i \in \{ 1,\ldots,M \},$ $j \in \{ 1,\ldots,m \},$ and $l \in \{ 1,2,3 \}$
\begin{equation} \label{5.4(1)}
|D^{l} v^{T}_{i}(y)| \leq c_{25} \frac{(\bE_{C}(T,1)+\kappa_{T})^{\frac{1}{2}}}{\dist(y,\partial \bV)^{l}} \text{ for } y \in \bV_{T},
\end{equation}
\begin{equation} \label{5.4(2)}
 |D^{l} w^{T}_{j}(y)| \leq c_{25} \frac{(\bE_{C}(T,1)+\kappa_{T})^{\frac{1}{2}}}{\dist(y,\partial \bW)^{l}} \text{ for } y \in \bW_{T},
\end{equation}
\begin{equation} \label{5.4(3)}
\begin{aligned}
\int_{\bV_{T}} & \left( \frac{\partial}{\partial r} \left( \frac{v^{T}_{i}(y)}{|y|} \right) \right)^{2} |y|^{2-n} \ d \cH^{n}(y) \\
& + \int_{\bW_{T}} \left( \frac{\partial}{\partial r} \left( \frac{w^{T}_{j}(y)}{|y|} \right) \right)^{2} |y|^{2-n} \ d \cH^{n}(y) \\
\leq & 4 (\bE_{S}(T,1)+c_{4} \kappa_{T}) \leq 4 \bE_{C}(T,1) + 4 ((M-m) \varpi_{n-1} + c_{4} ) \kappa_{T},
\end{aligned}
\end{equation}
where $\frac{\partial}{\partial r} f(y) = \frac{y}{|y|} \cdot Df(y)$ and $c_{4}$ is as in \eqref{2.2(4)}.

\medskip

If $m = 0,$ then we still conclude the existence of $v^{T}_{1},\ldots,v^{T}_{M} \in C^{\infty}(\bV_{T})$ satisfying the corresponding properties above.
\end{theorem} 

Note that the equations here are analogous to 5.4(1)(2)(3) of \cite{HS79}.

\medskip

\begin{proof} Suppose $m \geq 1,$ and let $\epsilon \in (0,1)$ be as in 5.3.14 of \cite{F69} with $\lambda,\kappa,m,n$ replaced by $1,1,n,n+1$. Recall $c_{14},c_{15}$ from \eqref{4.1(2)}, which depend only on $n,M,m.$ Also recall $c_{19},c_{20} \geq 1$ from Remark \ref{sec5.3}, which depend only on $n.$ We can thus choose $c_{24}=c_{24}(n,M,m) \geq 1$ so that 
$$\left( \frac{\sigma}{c_{24}} \right)^{2n+3} < \min \left\{ \frac{\sigma^{2n+3}}{c_{14}(1+c_{15}) c^{2}_{19} \exp(c_{20})}, \ \cH^{n}(\bV_{\frac{\sigma}{3}}), \ \frac{\epsilon}{2^{n}} \right\}$$
for any $\sigma \in (0,1).$ With this choice of $c_{24}$ we now fix a current $T \in \sT$ for which $\sigma = \sigma_{T} = c_{24} (\bE_{C}(T,1)+\kappa_{T})^{\frac{1}{2n+3}} \leq \frac{1}{4}.$

\medskip

First, we wish to apply Lemma \ref{sec5.1}. Using $c_{24} \geq 1$ and $\sigma \leq \frac{1}{4},$ we find that $\kappa_{T} < \frac{\sigma}{3}.$ We conclude
$$(\spt \partial T) \cap \bop^{-1}(\bV_{\frac{\sigma}{3}} \cap \bW_{\frac{\sigma}{3}}) = \emptyset.$$
Estimating $\bE_{C}(T,1)+\kappa_{T}=(\sigma/c_{24})^{2n+3}$ by the second quantity in the minimum above, we apply Lemma \ref{sec5.1} with $V,\mathfrak{m},S$ replaced by $\bV_{\frac{\sigma}{3}},M,T \res \bop^{-1}(\bV_{\frac{\sigma}{3}})$ (and respectively, $\bW_{\frac{\sigma}{3}},m,T \res \bop^{-1}(\bW_{\frac{\sigma}{3}})$) to obtain corresponding $S_{i}$ for $i=1,\ldots,M$ (respectively $S_{j}$ for $j=1,\ldots,m$) which are each absolutely area minimizing.

\medskip 

Second, we wish to apply interior regularity for area-minimizing currents. Estimating by the first quantity, we get using \eqref{4.1(2)} and $\sigma \leq \frac{1}{4}$
$$\sup_{C_{1-\frac{\sigma}{3} \cap \spt T}} \boq_{1} \leq (3^{n+1} \sigma^{n+2})^{\frac{1}{2}} < 1/2.$$
We now estimate by the third quantity and apply 5.3.15 of \cite{F69} with $\lambda,\kappa,m,n,r,s,S$ replaced by 
$$1,1,n,n+1,\left( \frac{\sigma}{c_{24}} \right)^{2n+3},\frac{1}{2}\left( \frac{\sigma}{c_{24}} \right)^{2n+3},\ETA_{-x,1 \sharp} S_{i}$$
for $x \in \bop^{-1}(\bV_{\frac{2\sigma}{3}}) \cap \spt S_{i}$ (and respectively, $\ETA_{-x,1 \sharp} S_{j}$ for $x \in \bop^{-1}(\bW_{\frac{2\sigma}{3}}) \cap \spt S_{j}$) to conclude that $\bop^{-1}(\bV_{\frac{2\sigma}{3}}) \cap \spt T$ (respectively, $\bop^{-1}(\bW_{\frac{2 \sigma}{3}}) \cap \spt T$) partitions into graphs of (at most) $M$ (respectively, $m$) solutions of the minimal surface equation over $\bV_{\frac{2 \sigma}{3}}$ (respectively, $\bW_{\frac{2\sigma}{3}}$); note that estimating by the first quantity and $\sigma \leq \frac{1}{4}$ gives $(\sigma/c_{24})^{2n+3} \leq \frac{\sigma}{3},$ and that $\ETA_{-x,1}$ is translation by $x.$

\medskip

Third, we show \eqref{5.4(1)} (respectively, similarly \eqref{5.4(2)}). Fix $i \in \{ 1,\ldots,M \}$ and abbreviate $v = v^{T}_{i}.$ Using \eqref{4.1(2)} (with $\sigma = \frac{1}{2}$ therein) gives
$$\sup_{B^{n}_{\frac{1}{2}} \cap \bV_{\frac{2 \sigma}{3}}} |v| \leq 2^{\frac{n+1}{2}} c_{14}^{\frac{1}{2}} (1+c_{15})^{\frac{1}{2}} (\bE_{C}(T,1)+ \kappa_{T})^{\frac{1}{2}}$$
Choose any $y \in \bV_{T} = B^{n}_{\frac{1}{2}} \cap \bV_{\sigma_{T}},$ then using \eqref{5.3(1)} with $\Omega = \bV_{\frac{2 \sigma}{3}}$ and $\rho = \dist(y,\partial (B^{n}_{\frac{1}{2}} \cap \bV_{\frac{2 \sigma}{3}})) \geq \max \{ \frac{\sigma}{3}, \frac{1}{3} \dist(y,\partial \bV) \},$ as well as using $\sigma = c_{24} (\bE_{C}(T,1)+\kappa_{T})^{\frac{1}{2n+3}} \leq \frac{1}{4}$ and $c_{24} \geq 1,$ we conclude
$$\begin{aligned}
|Dv(y)| < & c_{19} \Big( \sup_{B^{n}_{\frac{1}{2}} \cap \bV_{\frac{2 \sigma}{3}}} |v|/\rho \Big) \exp \Big( c_{20} \sup_{B^{n}_{\frac{1}{2}} \cap \bV_{\frac{2 \sigma}{3}}} |v|/\rho \Big) \\
\leq & \left( \frac{ 3c_{19} \sup_{B^{n}_{\frac{1}{2}} \cap \bV_{\frac{2 \sigma}{3}}} |v|}{ \dist(y,\partial \bV) } \right) \exp \Big( 3c_{20} \sup_{B^{n}_{\frac{1}{2}} \cap \bV_{\frac{2 \sigma}{3}}} |v|/\sigma \Big) \\
\leq & \left( \frac{ 3 \cdot 2^{\frac{n+1}{2}} c_{14}^{\frac{1}{2}} (1+c_{15})^{\frac{1}{2}} c_{19} (\bE_{C}(T,1)+ \kappa_{T})^{\frac{1}{2}}}{ \dist(y,\partial \bV) } \right) \\
& \times \exp \Big( 3 \cdot 2^{\frac{n+1}{2}} c_{14}^{\frac{1}{2}} (1+c_{15})^{\frac{1}{2}} c_{20} (\bE_{C}(T,1)+ \kappa_{T})^{\frac{1}{2}}/\sigma \Big) \\
\leq & \left( \frac{ 3 \cdot 2^{\frac{n+1}{2}} c_{14}^{\frac{1}{2}} (1+c_{15})^{\frac{1}{2}} c_{19} (\bE_{C}(T,1)+ \kappa_{T})^{\frac{1}{2}}}{ \dist(y,\partial \bV) } \right) \\
& \times \exp \Big( 3 \cdot 2^{\frac{n+1}{2}} c_{14}^{\frac{1}{2}} (1+c_{15})^{\frac{1}{2}} c_{20} \sigma^{n+\frac{1}{2}}c_{24}^{-n-\frac{3}{2}} \Big) \\
\leq & \left( \frac{ 3 \cdot 2^{\frac{n+1}{2}} c_{14}^{\frac{1}{2}} (1+c_{15})^{\frac{1}{2}} c_{19} (\bE_{C}(T,1)+ \kappa_{T})^{\frac{1}{2}}}{ \dist(y,\partial \bV) } \right) \\
& \times \exp \left( c_{14}^{\frac{1}{2}} (1+c_{15})^{\frac{1}{2}} c_{20} \right)
\end{aligned}$$
Thus, \eqref{5.4(1)} holds with $k=1$ if we choose $c_{25}$ depending on $n,c_{14},c_{15},c_{19},c_{20},c_{24}$ (and hence on $n,M,m$). Moreover, each of the partial derivatives $D_{l}v$ satisfies a linear divergence structure equation of the type treated in Remark \ref{sec5.2}, with constant $\lambda \in (0,1)$ depending only on $n,M,m$ (see for example the proof of Remark \ref{sec5.3}). The inequality \eqref{5.4(1)} with $l=2,3$ thus follows from \eqref{5.2(4)} and the interior Schauder theory (see section 6.3 of \cite{GT83}) for uniformly elliptic equations with H$\ddot{\text{o}}$lder continuous coefficients (respectively, \eqref{5.4(2)} also similarly holds). 

\medskip

To prove \eqref{5.4(3)}, consider again $v=v^{T}_{i}$ with fixed $i \in \{ 1,\ldots,M \}.$ For any $y \in \bV_{T},$ we compute using \eqref{4.1(2)} (as above with $\sigma=\frac{1}{2}$ therein), $\sigma \leq \frac{1}{4},$ $c_{19},c_{20} \geq 1,$ $|y| \in (\sigma,\frac{1}{4}),$ \eqref{5.3(1)} (as in showing \eqref{5.4(1)}), and estimating by the first quantity to give 
$$\begin{aligned}
|y|^{2}+|v(y)|^{2} \leq & \ |y|^{2} + 2^{n+1} c_{14} (1+c_{15}) (\bE_{C}(T,1)+ \kappa_{T}) \\
\leq & \ |y|^{2}+ 2^{-3n-1} \sigma^{2} \leq 2^{\frac{3}{(n+2)}} |y|^{2} < 1, \\
|Dv(y)|^{2} \leq & \ 9 \cdot 2^{n+1} c_{14} (1+c_{15}) c^{2}_{19} (\sigma/c_{24})^{2n+3}\sigma^{-2} \\
& \times \exp \left( 3 \cdot 2^{\frac{n+3}{2}} c_{14}^{\frac{1}{2}} (1+c_{15})^{\frac{1}{2}} c_{20} (\sigma/c_{24})^{\frac{2n+3}{2}}\sigma^{-1} \right) \leq 1. \\
\end{aligned}$$
Since $\cV^{T}(y,v(y)) = \ast \vec{T}(y,v(y)) = (-1)^{n} \frac{(-Dv(y),1)}{\sqrt{1+|Dv(y)|^{2}}}$ for $y \in \bV_{T},$ we compute 
\begin{equation} \label{5.4(4)}
\begin{aligned}
\int_{\bV_{T}} & \left( \frac{\partial}{\partial r} \left( \frac{v(y)}{|y|} \right) \right)^{2} |y|^{2-n} \ d \cH^{n}(y) \\
& = \int_{\bV_{T}} \frac{(y \cdot Dv(y) - v(y))^{2}}{|y|^{n+2}} \ d \cH^{n}(y) \\
& \leq 4 \int_{\bV_{T}} \frac{|(y,v(y)) \cdot \cV^{T}(y,v(y))|^{2}}{(|y|^{2}+|v(y)|^{2})^{\frac{n}{2}+1}} \sqrt{1+|Dv(y)|^{2}} \ d \cH^{n}(y) \\
& \leq 4 \int_{\bB_{1} \cap \bop^{-1}(\bV)} \frac{|x \cdot \cV^{T}(x)|^{2}}{|x|^{n+2}} \ d \mu_{T}(x);
\end{aligned}
\end{equation}
this is exactly as in 5.4(4) of \cite{HS79}. We also verify (as in 5.4(5) of \cite{HS79}) 
\begin{equation} \label{5.4(5)}
\begin{aligned}
\int_{\bW_{T}} \left( \frac{\partial}{\partial r} \left( \frac{w^{T}_{j}(y)}{|y|} \right) \right)^{2} & |y|^{2-n} \ d \cH^{n}(y) \\
& \leq 4 \int_{\bB_{1} \cap \bop^{-1}(\bW)} \frac{|x \cdot \cV^{T}(x)|^{2}}{|x|^{n+2}} \ d \mu_{T}(x)
\end{aligned}
\end{equation}
for each $j \in \{ 1,\ldots,m \}$. Since \eqref{2.2(5)} and \eqref{1.6(1)} imply that
$$\begin{aligned}
\int_{\bB_{1}} \frac{|x \cdot \cV^{T}(x)|^{2}}{|x|^{n+2}} \ d \mu_{T}(x) & \leq \bE_{S}(T,1)+c_{4} \kappa_{T} \\
& \leq \bE_{C}(T,1) + ((M-m) \varpi_{n-1} + c_{4}) \kappa_{T},
\end{aligned}$$
then \eqref{5.4(3)} follows from \eqref{5.4(4)} and \eqref{5.4(5)}. 

\medskip

The case $m=0$ follows similarly. \end{proof}

\section{Blowup sequences and harmonic blowups} \label{sec6}

This section introduces blowup sequences and harmonic blowups, with the aim to prove the necessary rigidity result Lemma \ref{sec6.4}. More precisely (as remarked in \cite{DS93}), we study the limit functions $f_{i}:\bV \rightarrow \bbR$ and $g_{j}:\bW \rightarrow \bbR$ obtained from the local graph representations $v^{(\ok)}_{i}: \bV_{\sigma_{T_{\ok}}} \rightarrow \bbR$ and $w^{(\ok)}_{j}: \bW_{\sigma_{T_{\ok}}} \rightarrow \bbR$ of a sequence of currents $T_{\ok} \in \sT$, with excess $\bE_{C}(T_{\ok},1) \rightarrow 0$ and $\bE_{C}(T_{\ok},1)^{-1} \kappa_{T_{\ok}} \rightarrow 0$ as $\ok \rightarrow \infty,$ by rescaling with a factor $\bE_{C}(T_{\ok},1)^{-1}$ in the vertical direction and passing to a convergent subsequence. Our aim is to prove the rigidity result Lemma \ref{sec6.4}, which allows us to relate the height of $T_{\ok},$  for large $\ok \in \bbN$ with the corresponding quantity associated with the set of limit functions. 

\medskip

This section is analogous to section 6 of \cite{HS79}. Only minor, mostly notational changes must be made. The only serious change is seen in justifying \eqref{6.4(13)}, which is analogous to 6.4(13) of \cite{HS79}. Whereas 6.4(13) of \cite{HS79} is proved using section 3.2 of \cite{HS79}, we must use Lemma \ref{sec3.2} (which differs from Lemma 3.2 of \cite{HS79}) to show \eqref{6.4(13)}. We introduce specifically in section \ref{sec6.4} constants $c_{26},\ldots,c_{33}$ depending on $n,M,m.$

\medskip

In section \ref{sec6.4} it will be necessary to assume $m \geq 1.$

\subsection{Definition} \label{sec6.1} 

We give the same definition of a blowup sequence and harmonic blowup as in \cite{HS79}. In this case, we must take functions $v_{i}^{(\ok)},f_{i}$ and $w_{j}^{(\ok)},g_{j}$ respectively with $i \in \{1,\ldots,M\}$ and $j \in \{1,\ldots,m\},$ but still require 6.1(1)-(4) from \cite{HS79} to hold.  

\medskip

Recalling Definition \ref{sec1.5,1.6} and Theorem \ref{sec5.4}, with $m \geq 1$ suppose that for each $\ok \in \bbN,$ $i \in \{ 1,\ldots,M \},$ and $j \in \{ 1,\ldots,m \},$ we have
$$T_{\ok} \in \sT, \varepsilon_{\ok} = \bE_{C}(T_{\ok},1)^{\frac{1}{2}}, \kappa_{\ok} = \kappa_{T_{\ok}}$$
$$\begin{array}{ccc}
v^{(\ok)}_{i}: \bV \rightarrow \bbR, & v^{(\ok)}_{i}|_{\bV_{T_{\ok}}} = v^{T_{\ok}}_{i}, & v^{(\ok)}_{i}|_{\bV \setminus \bV_{T_{\ok}}} = 0 \\
w^{(\ok)}_{j}: \bW \rightarrow \bbR, & w^{(\ok)}_{j}|_{\bW_{T_{\ok}}} = w^{T_{\ok}}_{j}, & w^{(\ok)}_{j}|_{\bW \setminus \bW_{T_{\ok}}} = 0;
\end{array}$$
in case $m=0$ we simply define $w^{(\ok)}_{j} \equiv 0$ over $\bW.$ This leads to the following definition.
\begin{definition}
We say that the sequence $\{ T_{\ok} \}_{\ok \in \bbN} \subset \sT$ is a blowup sequence with associated harmonic blowups $f_{i},g_{j}$ if as $\ok \rightarrow \infty$
\begin{equation} \label{6.1(1)}
\varepsilon_{\ok} \text{ converges to zero,}
\end{equation}
\begin{equation} \label{6.1(2)}
\varepsilon^{-2}_{\ok} \kappa_{\ok} \text{ converges to zero,}
\end{equation}
\begin{equation} \label{6.1(3)}
\varepsilon^{-2}_{\ok} v^{(\ok)}_{i} \text{ converges uniformly on compact subsets of } \bV \text{ to } f_{i},
\end{equation}
\begin{equation} \label{6.1(4)}
\varepsilon^{-2}_{\ok} w^{(\ok)}_{j} \text{ converges uniformly on compact subsets of } \bW \text{ to } g_{j}.
\end{equation}
\end{definition} 

These requirements are the same as 6.1(1)-(4) of \cite{HS79}. Note that when $m=0$ we simply have $g_{j} \equiv 0.$ It readily follows from the estimates of \eqref{5.4(1)},\eqref{5.4(2)}, or from 5.3.7 of \cite{F69}, that the functions $f_{i}: \bV \rightarrow \bbR,$ $g_{j} : \bW \rightarrow \bbR$ are harmonic. Moreover, by \eqref{4.1(2)},
\begin{equation} \label{6.1(5)}
\sup_{\bV \cap B^{n}_{\rho}} |f_{i}|^{2} + \sup_{\bW \cap B^{n}_{\rho}} |g_{j}|^{2} \leq 2 c_{14}(1+c_{15})(1-\rho)^{-2n-1}
\end{equation}
for each $\rho \in (0,1)$; see 6.1(5) of \cite{HS79}. We will frequently use, by \eqref{6.1(5)}, \eqref{5.4(1)},\eqref{5.4(2)}, and the Arzela-Ascoli theorem, the following fact:

\begin{lemma}
Every sequence $\{ T_{\ok} \}_{\ok \in \bbN} \subset \sT$ for which
$$\lim_{\ok \rightarrow \infty} \big( \bE_{C}(T_{\ok},1)+\bE_{C}(T_{\ok},1)^{-1} \kappa_{T_{\ok}} \big) = 0$$
contains a blowup sequence.
\end{lemma}

Our aim in sections \ref{sec6},\ref{sec7},\ref{sec8} is to show that for any blowup sequence $\{ T_{\ok} \}_{\ok \in \bbN}$ with $m \geq 1$ the associated harmonic blowups $f_{i},g_{j}$ are represented by two functions $f \in C^{2}(\bV \cup \bL),g \in C^{2}(\bW \cup \bL)$ so that
$$\begin{array}{cc}
\begin{aligned}
f|_{\bV} & = f_{1}=f_{2}=\ldots=f_{M}, \\
f|_{\bL} & =0=g|_{\bL},
\end{aligned}
&
\begin{aligned}
g|_{\bW}&=g_{1}=\ldots=g_{m}, \\
Df(0) & =Dg(0).
\end{aligned}
\end{array}$$

\subsection{Lemma} \label{sec6.2} 

We make a slight modification to the proof of this lemma as presented in \cite{HS79}, to make clearer the application to 5.3.7 of \cite{F69}. This lemma will be used for example in Theorem \ref{sec7.3} to show certain blowups have zero trace over $\bL.$ The proof requires referencing 5.3.7 of \cite{F69}, which gives essentially the same result at the interior for minimizers of an elliptic integrand.

\begin{lemma}
For every blowup sequence $\{ T_{\ok} \}_{\ok \in \bbN}$ with associate blowups $f_{i},g_{j},$ the two functions $\Pi,\Psi:B^{n}_{1} \rightarrow \bbR$ defined by
$$\begin{aligned}
\Pi|_{\bV} &= \sum_{i=1}^{M} f_{i}, \ \Pi|_{\bW} = \sum_{j=1}^{m} g_{j}, \ \Pi|_{\bL}=0, \\
\Psi|_{\bV} &= \min \{ |f_{1}|,|f_{2}|,\ldots,|f_{M}| \}, \ \Psi|_{\bW \cup \bL}=0,
\end{aligned}$$
both have locally square integrable weak gradients. Hence the function
$$\min \{ |f_{1}|,|f_{2}|,\ldots,|f_{M}| \}: \bV \rightarrow \bbR$$
has zero trace on $\bL.$
\end{lemma} 

\begin{proof} We take as in \cite{HS79} and above 
$$\varepsilon_{\ok} = \bE_{C}(T_{\ok},1)^{\frac{1}{2}}, \ \kappa_{\ok} = \kappa_{T_{\ok}}.$$
Define $q: \bbR \times \bbR^{n+1} \rightarrow \bbR^{n+1}$ by
$$q(t,x) = (x_{1},\ldots,x_{n-1},t x_{n},t x_{n+1})$$
for $(t,x) \in \bbR \times \bbR^{n+1}$ and subsequently define
$$Q_{\ok} = Q_{T_{\ok}} = q_{\sharp} \big( (\bbE^{1} \res [0,1]) \times (\partial T \res C_{2}) \big),$$ 
as in Lemma \ref{appendixlemma2}. Fix any $r \in (0,1)$ and define 
$$S_{\ok} = \big(T_{\ok} - Q_{\ok} + (M-m)(\bbE^{n} \res \bW) \big) \res C_{r}.$$ 

\medskip

Now, the second identity of \eqref{appendixlemma2identity} implies
$$\bM(\bop_{\sharp} T_{\ok} \res \bB^{n}_{r}) \leq M \cH^{n}(\bV \cap \bB^{n}_{r}) + m \cH^{n}(\bW \cap \bB^{n}_{r}) + \bM(\bop_{\sharp}Q_{\ok} \res \bB^{n}_{r}).$$
Using this together with \eqref{appendixlemma2identity} and \eqref{appendixlemma2retraction} we compute
$$\begin{aligned} 
\partial S_{\ok} \res (B^{n}_{r} \times \bbR) = & 0, \ \bop_{\sharp} S_{\ok} \res B^{n}_{r} = M \bbE^{n} \res B^{n}_{r}, \\
\bE_{C}(S_{\ok},r) \leq & r^{-n} \bM(T_{\ok} \res C_{r}) + r^{-n} \bM(Q_{\ok} \res C_{1}) + (M-m) r^{-n} \cH^{n}(\bW \cap \bB^{n}_{r}) \\
& - r^{-n} \bM \big( (M \bE^{n} \res \bV + M \bE^{n} \res \bW) \res \bB^{n}_{r} \big) \\
= & r^{-n} \bM(T_{\ok} \res C_{r}) + r^{-n} \bM(Q_{\ok} \res C_{r}) \\
& -r^{-n} \left( M \cH^{n}(\bV \cap \bB^{n}_{r}) + m \cH^{n} (\bW \cap \bB^{n}_{r}) \right) \\
\leq & \bE_{C}(T_{\ok},r) + r^{-n} \bM(Q_{\ok} \res C_{r}) + r^{-n} \bM(\bop_{\sharp} Q_{\ok} \res \bB^{n}_{r}) \\
\leq & r^{-n} \varepsilon_{\ok} + (M-m) \kappa_{\ok} \varpi_{n-1} r^{\alpha} \left( \frac{\alpha}{2} \right) \left(1+\frac{\alpha^{2} \kappa_{\ok}^{2}}{4} r^{2 \alpha} + \frac{\alpha^{4} \kappa_{\ok}^{4}}{16} r^{4 \alpha} \right)^{\frac{1}{2}} \\
& + (M-m) \kappa_{\ok} \varpi_{n-1} r^{\alpha}. 
\end{aligned}$$ 

\medskip

We can hence apply Lemma \ref{sec5.1} with $V,S,\mathfrak{m}$ replaced with $B^{n}_{r},S_{\ok},M$ to obtain corresponding $S_{\ok,1},\ldots,S_{\ok,M}.$ Observe that for each $i \in \{ 1,\ldots,M \}$
$$\limsup_{\ok \rightarrow \infty} \varepsilon^{-2}_{\ok} \bE_{C}(S_{\ok,i},r) \leq \limsup_{\ok \rightarrow \infty} \varepsilon^{-2}_{\ok} \bE_{C}(S_{\ok},r) \leq r^{-n}$$
by Definition \eqref{sec6.1}. Using \eqref{6.1(2)}, we see that the proof of 5.3.7 of \cite{F69} carries over with 
$$\kappa,\lambda,\Psi_{\nu},m,n,r_{\nu},s_{\nu},\varepsilon_{\nu},\alpha$$
replaced by
$$1,1,\text{the parameteric area integrand},n,n+1,r,1,\varepsilon_{\ok},r^{-n}.$$
To see this, we make two observations: first, by \eqref{4.1(2)} with $\sigma = (1-r)$ and \eqref{defderivative} together with 4.1.8 of \cite{F69} and Lemma 26.25 of \cite{S83}, we have
$$\begin{aligned}
& \spt S_{k} \\
& \subset \left\{ x \in C_{r}: |x_{n+1}| \leq \max \left\{ \frac{(c_{14}(1+c_{15}))^{\frac{1}{2}}}{(1-r)^{n+\frac{1}{2}}} ( \varepsilon_{\ok}+\kappa_{\ok}^{\frac{1}{2}}),\left( \frac{\alpha}{2} \right) \kappa_{\ok} r^{1+\alpha} \right\} \right\} \\
\end{aligned}$$
which is contained in $\{ x \in C_{r}: |x_{n+1}| < 1 \},$ for all sufficiently large $\ok$; second, note that in 5.3.7 of \cite{F69} the only terms which use the hypothesis
$$\lim_{\nu \rightarrow \infty} (\varepsilon^{-1}_{\nu} r_{\nu}+\varepsilon^{-1}_{\nu} s_{\nu}) = 0$$
are $M_{\nu,1}(t)$ and $M_{\nu,3}(t),$ both of which vanish because the area integrand is a constant coefficient integrand. 

\medskip

Choosing locally integrable functions $\Pi_{i}: B^{n}_{r} \rightarrow \bbR$ so that
$$\int \Pi_{i} \cdot \phi \ d \cH^{n} = \lim_{\ok \rightarrow \infty} \varepsilon^{-1}_{\ok} S_{\ok,i} \big( (\phi \circ \bop) \boq_{1} dx_{1} \wedge \ldots \wedge dx_{n} \big)$$
for any $\phi \in C^{\infty}_{c}(B^{n}_{r}),$ we deduce that each $\Pi_{i}$ has locally square integrable weak gradient. By Definition \ref{sec6.1} and the choice of $S_{\ok,i},$ the set of values $\{ \Pi_{i}(y),\ldots,\Pi_{M}(y) \}$ coincides for Lebesgue almost all $y \in \bW \cap B^{n}_{r}$ with the (unordered) set $\{ 0,g_{1}(y),\ldots,g_{m}(y) \}.$ After changing each $\Pi_{i}$ on an $\cH^{n}$-null set, we conclude that
$$\Pi|_{B^{n}_{r}} = \sum_{i=1}^{M} \Pi_{i}, \ \Psi|_{B^{n}_{r}} = \min \{ |\Pi_{1}|,\ldots,|\Pi_{M}| \}.$$
Since this holds for all $r \in (0,1),$ then both $\Pi,\Psi$ have locally square integrable weak gradients over $B^{n}_{1}.$

\medskip

The last part of the lemma follows from section 26 of \cite{T75}. \end{proof}

\subsection{Lemma} \label{sec6.3} 

As noted in \cite{DS93}, one crucial step in \cite{HS79} is to obtain information about the height of currents in a blowup sequence from bounds on the harmonic blowups. Proving this depends on a well-known barrier argument, see Corollary 4.3 of \cite{H77}.

\begin{lemma}
Suppose $\{ T_{\ok} \}_{\ok \in \bbN}$ is a blowup sequence with associated harmonic blowups $f_{i},g_{j}$ and let $\varepsilon_{\ok} = \bE_{C}(T_{\ok},1)^{\frac{1}{2}}.$ If $\sigma \in (0,1/2),$ $z \in B^{n-1}_{1-2\sigma},$ and $K$ is a compact subset of $\bop^{-1}(B^{n}_{\sigma}(z)),$ then
$$\begin{aligned}
\limsup_{\ok \rightarrow \infty} \sup_{K \cap \spt T_{\ok}} \varepsilon^{-1}_{\ok} \boq_{1} & \leq \max \{ \sup_{ \bV \cap \partial B^{n}_{\sigma}(z)} f_{M}, \sup_{ \bW \cap \partial B^{n}_{\sigma}(z)} g_{m},0 \}, \\
\liminf_{\ok \rightarrow \infty} \inf_{K \cap \spt T_{\ok}} \varepsilon^{-1}_{\ok} \boq_{1} & \geq \min \{ \inf_{ \bV \cap \partial B^{n}_{\sigma}(z)} f_{1}, \inf_{ \bW \cap \partial B^{n}_{\sigma}(z)} g_{1},0 \}.
\end{aligned}$$
\end{lemma} 

\begin{proof} For each $\ok \in \bbN,$ let $\kappa_{\ok} = \kappa_{T_{\ok}},$ $\sigma_{\ok}=\sigma_{T_{\ok}},$ and $v^{(\ok)}_{i},w^{(\ok)}_{j}$ be as in sections \ref{sec1.5,1.6},\ref{sec5.4},\ref{sec6.1}. We may assume $2 \sigma_{\ok} < \sigma.$

\medskip

To prove the first inequality, we use \eqref{4.1(2)} and Theorem \ref{sec5.4} to choose functions $b_{\ok} \in C^{2}(\partial B^{n}_{\sigma}(z))$ so that
$$\begin{aligned}
\max \{ v^{(\ok)}_{M}(y),\kappa_{\ok} \} \leq b_{\ok}(y) & \leq \max \{ v^{(\ok)}_{M}(y),\kappa_{\ok} \} + \frac{1}{\ok} \text{ for } y \in \bV_{2 \sigma_{\ok}} \cap \partial B^{n}_{\sigma}(z) \\
\max \{ w^{(\ok)}_{m}(y),\kappa_{\ok} \} \leq b_{\ok}(y) & \leq \max \{ w^{(\ok)}_{m}(y),\kappa_{\ok} \} + \frac{1}{\ok} \text{ for } y \in \bW_{2 \sigma_{\ok}} \cap \partial B^{n}_{\sigma}(z), \\
\max \{ \sup_{\bop^{-1}(y) \cap \spt T_{\ok}} \boq_{1}, \kappa_{\ok} \} & \leq b_{\ok}(y) \\
& \leq \frac{2(c_{14}(1+c_{15}))^{\frac{1}{2}}}{\sigma^{n+\frac{1}{2}}} (\varepsilon_{\ok}+\kappa^{\frac{1}{2}}_{\ok}) \text{ for } y \in \partial B^{n}_{\sigma}(z).
\end{aligned}$$

\medskip

Next we solve the Dirichlet problem (see Theorem 16.9 of \cite{GT83}) to obtain $u_{\ok} \in C^{2}(\bB^{n}_{\sigma}(z))$ so that $u_{\ok}|_{\partial B^{n}_{\sigma}(z)} = b_{\ok}$ and $u_{\ok}$ satisfies the minimal surface equation over $B^{n}_{\sigma}(z)$. The maximum principle (see section 3.6 of \cite{GT83}) implies $u_{\ok} \geq \kappa_{\ok}.$ Also, a well-known barrier argument (see Corollary 4.3 of \cite{H77}) can be used to show
\begin{equation} \label{6.3(1)}
x_{n+1} \leq u_{\ok}(\bop(x)) \text{ for } x \in \bop^{-1}(B^{n}_{\sigma}(z)) \cap \spt T_{\ok};
\end{equation}
this is 6.3(1) of \cite{HS79}. Using these two facts we conclude
$$\begin{aligned}
\spt \partial \big( T_{\ok} \res \bop^{-1}(B^{n}_{\sigma}(z)) \big) & \subseteq \big( \bop^{-1}(\bB^{n}_{\sigma}(z)) \cap \spt \partial T_{\ok} \big) \cup \big( \bop^{-1}(\partial B^{n}_{\sigma}(z)) \cap \spt T_{\ok} \big) \\
& \subset \bop^{-1}(\bB^{n}_{\sigma}(z)) \cap \{ x \in \bbR^{n+1}: x_{n+1} \leq u_{\ok}(\bop(x)) \}.
\end{aligned}$$
Using Lemma \ref{sec5.3} and \eqref{4.1(2)} to obtain local interior gradient bounds on $\varepsilon^{-1}_{\ok} u_{\ok},$ independent of $\ok,$ and using the interior H$\ddot{\text{o}}$lder estimates for the gradient (see Chapter 13 of \cite{GT83}), we find a bounded harmonic function $u$ on $B^{n}_{\sigma}(z)$ and an increasing sequence $\{ \ok_{l} \}_{l \in \bbN}$ so that as $l \rightarrow \infty$
\begin{equation} \label{6.3(2)}
\varepsilon^{-1}_{\ok_{l}} u_{\ok_{l}} \rightarrow u \text{ uniformly on compact subsets of } B^{n}_{\sigma}(z);
\end{equation}
this is as in 6.3(2) of \cite{HS79}. Recalling \eqref{6.1(1)}-\eqref{6.1(4)} and using linear barriers, we readily verify that
$$\begin{aligned}
\lim_{\tilde{y} \in B^{n}_{\sigma}(z), \tilde{y} \rightarrow y} u(\tilde{y}) &= \max \{ f_{M}(y),0 \} \text{ for } y \in \bV \cap \partial B^{n}_{\sigma}(z), \\
\lim_{\tilde{y} \in B^{n}_{\sigma}(z), \tilde{y} \rightarrow y} u(\tilde{y}) &= \max \{ g_{m}(y),0 \} \text{ for } y \in \bW \cap \partial B^{n}_{\sigma}(z).
\end{aligned}$$
From the Poisson integral formula, we then obtain the inequality
$$\sup_{B^{n}_{\sigma}(z)} u \leq \max \{ \sup_{\bV \cap \partial B^{n}_{\sigma}(z)} f_{M}, \sup_{\bW \cap \partial B^{n}_{\sigma}(z)} g_{m},0 \},$$
which along with \eqref{6.3(1)},\eqref{6.3(2)} establishes the first inequality of the conclusion.

\medskip

The second inequality similarly follows using a lower barrier. \end{proof}

\subsection{Lemma} \label{sec6.4}

We here show that if the harmonic blowups $f_{i}$ and $g_{j}$ are each restrictions to $\bV$ and respectively $\bW$ of some multiple of $\boq_{0},$ then in fact $f_{i}$ and $g_{j}$ are the restriction to $\bV$ and respectively $\bW$ of $\beta \boq_{0}$ for one value of $\beta \in \bbR.$  We will use this lemma most notably in the proof of Theorem \ref{sec9.1}, to show that if $T_{\ok}$ is a blowup sequence with linear harmonic blowups as described above, then for sufficiently large $\ok \in \bbN,$ if we tilt $T_{\ok}$ by $\beta$ and rescale by $\tau = \tau(n,M,m,\alpha),$ then the new cylindrical excess is proportional to $\tau$ by a factor depending on $\bE_{C}(T_{\ok},1)$ and $\kappa_{T_{\ok}}.$

\medskip

In this section we need $m \geq 1.$

\begin{lemma}  
Suppose $m \geq 1.$ If $\beta_{1} \leq \beta_{2} \leq \ldots \leq \beta_{M}$ and $\gamma_{1} \geq \ldots \geq \gamma_{m},$ while $\{ T_{\ok} \}_{\ok \in \bbN}$ is a blowup sequence with associated harmonic blowups
$$\begin{aligned}
f_{i}(y) & = \beta_{i} y_{n} \text{ for } y \in \bV, \\
g_{j}(y) & = \gamma_{j} y_{n} \text{ for } y \in \bW,
\end{aligned}$$
then
$$\beta_{1} = \beta_{2} = \ldots = \beta_{M} = \gamma_{1} = \ldots = \gamma_{m},$$
and 
$$\lim_{\ok \rightarrow \infty} \sup_{x \in C_{\rho} \cap \spt T_{\ok}} |\varepsilon^{-1}_{\ok} x_{n+1}-\beta_{1} x_{n} |=0$$
for $\rho \in (0,1),$ where $\varepsilon_{\ok} = \bE_{C}(T_{\ok},1)^{\frac{1}{2}}.$
\end{lemma}

The proof proceeds partly by proving that the function $b : \bB^{n}_{\frac{1}{2}} \rightarrow \bbR$ defined by $b(y) = \beta_{M} y_{n}$ for $y \in \bB^{n}_{\frac{1}{2}} \cap \Clos \bV$ and $b(y) = \gamma_{m} y_{n}$ for $y \in \bB^{n}_{\frac{1}{2}} \cap \Clos \bW$ is harmonic; $b$ is the ``top sheet'' of the harmonic blowups. Showing this is done by using the area minimality of each $T_{\ok}$ in the blowup sequence, to directly show $b$ minimizes the Dirichlet integral. To do this, we must crucially use \eqref{appendixlemma2projectionmass}, which means we must assume $m \geq 1.$

\begin{proof} While this lemma is analogous to Lemma 6.4 of \cite{HS79}, we must make some changes to the proof. In particular, we observe our use of Lemma \ref{sec3.2}, which differs from the corresponding Lemma 3.2 of \cite{HS79}.

\medskip

Let $v^{(\ok)}_{i},w^{(\ok)}_{j}$ be as in Definition \ref{sec6.1}, and let
$$\begin{aligned}
\Lambda & = \max \{|\beta_{1}|,|\beta_{M}|,|\gamma_{1}|,|\gamma_{m}| \}, \\
\lambda & = \min \big( \{ 1 \} \cup \{ \beta_{i+1}-\beta_{i}: \beta_{i} < \beta_{i+1} \} \cup \{ \gamma_{j}-\gamma_{j+1}: \gamma_{j} > \gamma_{j+1} \} \big).
\end{aligned}$$ 
For each $\sigma \in (0,\min \{ \frac{\lambda}{2},\frac{1}{16} \}),$ we use \eqref{5.4(1)},\eqref{5.4(2)}, \eqref{6.1(1)}-\eqref{6.1(4)}, and Lemma \ref{sec6.3} to find $N_{\sigma} \in \bbN$ so that for $\ok \geq N_{\sigma}$ 
\begin{equation} \label{6.4(1)}
\sigma_{T_{\ok}} < \frac{\sigma}{4}, \ \varepsilon^{2}_{\ok} < \sigma, \ \kappa_{T_{\ok}} < \sigma^{3} \varepsilon^{2}_{\ok}
\end{equation}
\begin{equation} \label{6.4(2)}
\sup_{\bV_{\sigma/2}} |v^{(\ok)}_{i} - \varepsilon_{\ok} \beta_{i} \boq_{0}|^{2} \leq \sigma^{n+4} \varepsilon^{2}_{\ok} \text{ for } i \in \{ 1,\ldots,M \},
\end{equation}
\begin{equation} \label{6.4(3)}
\sup_{\bW_{\sigma/2}} |w^{(\ok)}_{j} - \varepsilon_{\ok} \gamma_{j} \boq_{0}|^{2} \leq \sigma^{n+4} \varepsilon^{2}_{\ok} \text{ for } j \in \{ 1,\ldots,m \},
\end{equation}
\begin{equation} \label{6.4(4)}
\sup_{C_{\sigma+\frac{3}{4}} \cap (\spt T_{\ok}) \setminus \bop^{-1}(\bV_{2 \sigma} \cup \bW_{2 \sigma})} |\boq_{0}| \leq 2 \Lambda \sigma \varepsilon_{\ok};
\end{equation}
these are as in equations 6.4(1)-(4) of \cite{HS79}.

\medskip

Next, if $\ok \geq N_{\sigma}$ then \eqref{5.4(1)} and \eqref{6.4(1)} imply $|D v_{i}^{(\ok)}| \leq c_{25}$ for $y \in \bV_{T_{\ok}}.$ Using this, for each $i \in \{ 1,\ldots,m \}$ and $\jmath \in \{ 1,\ldots,n \}$ we apply \eqref{5.2(3)} with 
$$\begin{aligned}
u & = D_{\jmath} (v^{(\ok)} - \varepsilon_{\ok} \beta_{i} \boq_{0}) = D_{\jmath} v^{(\ok)}_{i} - \varepsilon_{\ok} \beta_{i} \delta_{n \jmath} \\
a_{kl} & = \frac{1}{\sqrt{1+|p|^{2}}} \left( \delta_{kl} - \frac{p_{k}p_{l}}{1+|p|^{2}} \right) \Big|_{p = Dv^{\ok}_{i}},
\end{aligned}$$
then apply \eqref{5.2(2)} with $u = v^{(\ok)}_{i} - \varepsilon_{\ok} \beta_{i} \boq_{0}$ and
$$a_{kl} = \int_{0}^{1} \frac{1}{\sqrt{1+|p|^{2}}} \left( \delta_{kl} - \frac{p_{k} p_{l}}{1+|p|^{2}} \right) \Big|_{p= tDv^{(\ok)}_{i}+(1-t) \varepsilon_{\ok} \beta_{i} e_{n}} \ dt$$
in order to conclude
\begin{equation} \label{6.4(5)}
\begin{aligned}
\sup_{\bV_{\sigma}} |D & ( v^{(\ok)}_{i} - \varepsilon_{\ok} \beta_{i} \boq_{0} )| \\
\leq & \frac{c_{26}}{\sigma^{n}} \int_{\bV_{\frac{\sigma}{2}}} |D( v^{(\ok)}_{i} - \varepsilon_{\ok} \beta_{i} \boq_{0} )|^{2} \ d \cH^{n} \\
\leq & \frac{c_{27}}{\sigma^{n+2}} \int_{\bV_{\frac{\sigma}{4}}} |v^{(\ok)}_{i} - \varepsilon_{\ok} \beta_{i} \boq_{0}|^{2} \ d \cH^{n} \leq c_{27} \varpi_{n} \sigma^{2} \varepsilon^{2}_{\ok}
\end{aligned}
\end{equation}
with $c_{26},c_{27}$ depending in $n,M,m;$ this is 6.4(5) of \cite{HS79}. We similarly verify 
$$\sup_{\bW_{\sigma}} |D(w^{(\ok)}_{j} - \varepsilon_{\ok} \gamma_{j} \boq_{0})|^{2} \leq c_{27} \varpi_{n} \sigma^{2} \varepsilon^{2}_{\ok}$$
for each $j \in \{1,\ldots,m\}.$ 

\medskip

Next, with $\sigma \in (0,\min \{ \frac{\lambda}{2},\frac{1}{16} \}),$ $i \in \{ 1,\ldots,M \},$ and $j \in \{ 1,\ldots,m \}$ define 
$$\begin{aligned}
H^{\sigma} & = \{ x \in \bbR^{n+1}: |x_{n}| \leq \sigma \}, \\
I^{\sigma}_{i,\ok} &= \{ x \in \bop^{-1}(\bV_{\sigma}): |\varepsilon^{-1}_{\ok} x_{n+1}- \beta_{i} x_{n}| < \lambda \sigma/2 \}, \\
J^{\sigma}_{j,\ok} &= \{ x \in \bop^{-1}(\bW_{\sigma}): |\varepsilon^{-1}_{\ok} x_{n+1}- \gamma_{j} x_{n}| < \lambda \sigma/2 \};
\end{aligned}$$
note that
$$\begin{aligned}
I^{\sigma}_{i,\ok} \cap I^{\sigma}_{\tilde{i},\ok} & \neq \emptyset \text{ if and only if } \beta_{i} = \beta_{\tilde{i}}, \\
J^{\sigma}_{j,\ok} \cap J^{\sigma}_{\tilde{j},\ok} & \neq \emptyset \text{ if and only if } \gamma_{j} = \gamma_{\tilde{j}}.
\end{aligned}$$
For each $\ok \geq N_{\sigma}$ we also define the set
$$G^{\sigma}_{\ok} = H^{\sigma} \cup \bigcup_{i=1}^{M} I^{\sigma}_{i,\ok} \cup \bigcup_{j=1}^{m} J^{\sigma}_{j,\ok},$$
and the Lipschitz map $\Lambda^{\sigma}_{\ok}:G^{\sigma}_{\ok} \rightarrow \bbR^{n+1}$ given by
$$\begin{array}{ll} 
L^{\sigma}_{\ok}(x) = \bop(x) & \text{ for } x \in H^{\sigma}, \\
L^{\sigma}_{\ok}(x) = (\bop(x),\varepsilon_{\ok} \beta_{i}(x_{n}-\sigma)) & \text{ for } x \in I^{\sigma}_{i,\ok}, \\
L^{\sigma}_{\ok}(x) = (\bop(x),\varepsilon_{\ok} \gamma_{j}(x_{n}+\sigma)) & \text{ for } x \in J^{\sigma}_{j,\ok}.
\end{array}$$
Finally, let $\phi \in C^{1}(\bB^{n}_{1};[0,1])$ be a function with
$$\begin{array}{ccc}
\phi|_{\bB^{n}_{1/2}} \equiv 0, & \phi|_{\bB^{n}_{1} \setminus B^{n}_{3/4}} \equiv 1, & \sup |D\phi| \leq 5,
\end{array}$$
and define the Lipschitz mapping
$$F^{\sigma}_{\ok}: G^{\sigma}_{\ok} \cup (\bbR^{n+1} \setminus C_{\frac{3}{4}}) \rightarrow \bbR^{n+1},$$
by
$$\begin{array}{ll}
F^{\sigma}_{\ok}(x) = x & \text{for } x \in \bbR^{n+1} \setminus C_{\frac{3}{4}}, \\
F^{\sigma}_{\ok}(x) = (1-(\phi \circ \bop)) L^{\sigma}_{\ok}(x) + (\phi \circ \bop)(x) & \text{for } x \in G^{\sigma}_{\ok}.
\end{array}$$

\medskip

We wish to estimate $\bM (F^{\sigma}_{\ok \sharp} T_{\ok}) - \bM (T_{\ok}).$ First notice that
\begin{equation} \label{6.4(6)}
\bop^{-1}(\bV_{\sigma}) \cap \spt T_{\ok} = \cup_{i=1}^{M} \graph{\bV_{\sigma}}{v^{(\ok)}_{i}} \subset G^{\sigma}_{\ok},
\end{equation}
\begin{equation} \label{6.4(7)}
\bop^{-1}(\bV_{\sigma}) \cap F^{\sigma}_{\ok \sharp} \spt T_{\ok} = \cup_{i=1}^{M} \graph{\bV_{\sigma}}{u^{(\ok)}_{i}} \subset G^{\sigma}_{\ok},
\end{equation}
as in 6.4(6)(7) of \cite{HS79}, where we define
$$\begin{aligned}
u^{(\ok)}_{i} &= (1-\phi) \varepsilon_{\ok} \beta_{i} (\boq_{0}-\sigma)+\phi v^{(\ok)}_{i} \\
&= (1-\phi) \varepsilon_{\ok} \beta_{i} \boq_{0} + \phi v^{(\ok)}_{i} - (1-\phi) \varepsilon_{\ok} \beta_{i} \sigma
\end{aligned}$$
over $\bV_{\sigma}$; hence,
\begin{equation} \label{6.4(8)}
\begin{aligned}
\sup_{\bV_{\sigma}} \left( |Du^{(\ok)}_{i}|+|Dv^{(\ok)}_{i}| \right) \leq &  c_{28} (1+\Lambda) \varepsilon_{\ok}, \\
u^{(\ok)}_{i} - v^{(\ok)}_{i} = & (1-\phi)( \varepsilon_{\ok} \beta_{i} \boq_{0} - v^{(\ok)}_{i}) \\
& - (1-\phi) \varepsilon_{\ok} \beta_{i} \sigma \text{ over } \bV_{\sigma}.
\end{aligned}
\end{equation}
\begin{equation} \label{6.4(9)}
\begin{aligned}
\sup_{\bV_{\sigma}} & \left| Du^{(\ok)}_{i}+Dv^{(\ok)}_{i} \right| \\
\leq & \sup_{\bV_{\sigma}} \left| (1-\phi) D(\varepsilon_{\ok} \beta_{i} \boq_{0} - v^{(\ok)}_{i})+(\varepsilon_{\ok} \beta_{i} \boq_{0} - v^{(\ok)}_{i} + \epsilon_{\ok} \beta_{i} \sigma ) D\phi \right| \\
\leq & c_{29} (1+\Lambda) \sigma \varepsilon_{\ok}
\end{aligned}
\end{equation}
by \eqref{6.4(2)} and \eqref{6.4(5)}, with $c_{28},c_{29}$ depending on $n,M,m$; these are as in 6.4(8)(9) of \cite{HS79}. Using \eqref{6.4(6)},\eqref{6.4(7)},\eqref{6.4(8)}, and \eqref{6.4(9)} we estimate
\begin{equation} \label{6.4(10)}
\begin{aligned}
\bM(F^{\sigma}_{\ok \sharp} & ( T_{\ok} \res \bop^{-1}(\bV_{\sigma}))) - \bM(T_{\ok} \res \bop^{-1}(\bV_{\sigma})) \\
= & \sum_{i=1}^{M} \int_{\bV_{\sigma}} \left( (1+|Du^{(\ok)}_{i}|^{2})^{1/2} - (1+|Dv^{(\ok)}_{i}|^{2})^{1/2} \right) \ d \cH^{n} \\
\leq & \sum_{i=1}^{M} \int_{\bV_{\sigma}} |Du^{(\ok)}_{i}-Dv^{(\ok)}_{i}| \left( |Du^{(\ok)}_{i}| + |Dv^{(\ok)}_{i}| \right) \ d \cH^{n} \\
\leq & c_{30} (1+\Lambda)^{2} \sigma \varepsilon^{2}_{\ok}
\end{aligned}
\end{equation}
with $c_{30}$ depending on $n,M,m$; this is as in 6.4(10) of \cite{HS79}. Similarly, we verify that
\begin{equation} \label{6.4(11)}
\bM(F^{\sigma}_{\ok \sharp} ( T_{\ok} \res \bop^{-1}(\bW_{\sigma}))) - \bM(T_{\ok} \res \bop^{-1}(\bW_{\sigma})) \leq c_{30} (1+\Lambda)^{2} \sigma \varepsilon^{2}_{\ok} 
\end{equation}
as in 6.4(11) of \cite{HS79}.

\medskip

Next we infer from \eqref{1.4(1)}  that
$$\begin{aligned} 
\mu_{T_{\ok}}(H^{2 \sigma} \cap C_{\frac{3}{4}+\sigma}) 
& \leq \bM(\bop_{\sharp}(T_{\ok} \res H^{2 \sigma} \cap C_{\frac{3}{4}+\sigma})) + (3/4+\sigma)^{n} \bE_{C}(T,3/4+\sigma) \\ 
& \leq \bM(\bop_{\sharp}(T_{\ok} \res H^{2 \sigma} \cap C_{\frac{3}{4}+\sigma})) + \varepsilon_{\ok}^{2}.
\end{aligned}$$ 
Using \eqref{appendixlemma2identity} and \eqref{appendixlemma2retraction} with $r = (\frac{3}{4}+\sigma)$ we compute 
$$\begin{aligned} 
\bM(\bop_{\sharp}(T_{\ok} \res H^{2 \sigma} \cap C_{\frac{3}{4}+\sigma})) \leq & (M+m) \varpi_{n-1} (3/4+\sigma)^{n-1} (2\sigma) \\ 
& + (M-m) \kappa_{T_{\ok}} \varpi_{n-1} (3/4+\sigma)^{n+\alpha}.
\end{aligned}$$
The previous two calculations, $\sigma \leq \frac{1}{16},$ and \eqref{6.4(1)} imply that for $\ok \geq N_{\sigma}$
\begin{equation} \label{6.4(12)} 
\mu_{T_{\ok}}(H^{2 \sigma} \cap C_{\frac{3}{4}+\sigma}) \leq c_{31} (\sigma + \kappa_{T_{\ok}}) + \varepsilon_{\ok}^{2} \leq 2c_{31} \sigma, 
\end{equation} 
with $c_{31}$ depending on $n,M,m$; this is as in 6.4(12) of \cite{HS79}.

\medskip

Noting that
$$F^{\sigma}_{\ok}(x) = (\bop(x),\phi(\bop(x))x_{n+1}) \text{ for all } x \in H^{\sigma} \cap C_{3/4},$$
then Lemma \ref{sec3.2} (with $A = H^{\sigma} \cap C_{\frac{3}{4}}$ and $\tau = \sigma$) and \eqref{6.4(6)},\eqref{6.4(12)} give for $c_{32},c_{33}$ depending on $n,M,m$ 
\begin{equation}
\begin{aligned} \label{6.4(13)}
\bM(F^{\sigma}_{\ok \sharp} (T_{\ok} \res H^{\sigma})) & - \bM (T_{\ok} \res H^{\sigma}) \\
& = \bM(F^{\sigma}_{\ok \sharp} (T_{\ok} \res H^{\sigma})) - \bM(T_{\ok} \res H^{\sigma}) \\ 
& \leq \frac{c_{32}}{\sigma^{2}} \left( \kappa_{T_{\ok}} + \int_{H^{2 \sigma} \cap C_{\frac{3}{4}+\sigma}} \boq_{1}^{2} \ d \mu_{T_{\ok}} \right) \\ 
& \leq c_{33} (1+\Lambda) \sigma \varepsilon^{2}_{\ok}, 
\end{aligned}
\end{equation} 
using as well $\kappa_{T_{\ok}} < \sigma^{3} \varepsilon^{2}_{\ok}$ from \eqref{6.4(1)}; this is as in 6.4(13) of \cite{HS79}. 

\medskip

Combining \eqref{6.4(10)},\eqref{6.4(11)}, and \eqref{6.4(13)} gives the desired estimate
\begin{equation} \label{6.4(14)}
\bM(F^{\sigma}_{\ok \sharp} T_{\ok}) - \bM(T_{\ok}) \leq (2c_{30}+c_{33})(1+\Lambda)^{2} \sigma \varepsilon^{2}_{\ok}
\end{equation}
for all $\ok \geq N_{\sigma}.$

\medskip

We will use this estimate to show that the function $b : \bB^{n}_{\frac{1}{2}} \rightarrow \bbR$
$$\begin{aligned} 
b(y) & = \beta_{M} y_{n} &\text{ for } y  \in \bB^{n}_{\frac{1}{2}} \cap \Clos{\bV} \\ 
b(y) & = \gamma_{m} y_{n} & \text{ for } y \in \bB^{n}_{\frac{1}{2}} \cap \Clos{\bW}. 
\end{aligned}$$ 
is harmonic in $B^{n}_{\frac{1}{2}}.$ For this purpose let $\zeta:\bB^{n}_{\frac{1}{2}} \rightarrow \bbR$ be any Lipschitz function with $\zeta|_{\partial \bB^{n}_{\frac{1}{2}}}=b|_{\partial \bB^{n}_{\frac{1}{2}}},$ and let $\{ \sigma_{l} \}_{l=1}^{\infty}$ be any decreasing sequence of positive numbers with limit zero and $\sigma_{1} < \min \{ \frac{\lambda}{2},\frac{1}{16} \}.$ For each $l \in \bbN,$ let $\ok_{l} = N_{\sigma_{l}}$ and define $b_{l}: \bB^{n}_{\frac{1}{2}} \rightarrow \bbR$ by
$$\begin{array}{ll}
b_{l}(y) = \beta_{M}(y_{n}-\sigma_{l}) & \text{ for } y \in \bB^{n}_{\frac{1}{2}} \cap \bV_{\sigma_{l}} \\
b_{l}(y) = \gamma_{m}(y_{n}-\sigma_{l}) & \text{ for } y \in \bB^{n}_{\frac{1}{2}} \cap \bW_{\sigma_{l}} \\
b_{l}(y) = 0 & \text{ for } y \in \bB^{n}_{\frac{1}{2}} \setminus (\bV_{\sigma_{l}} \cup \bW_{\sigma_{l}}),
\end{array}$$
and choose $\zeta_{l}: \bB^{n}_{\frac{1}{2}} \rightarrow \bbR$ Lipschitz so that $\zeta_{l}|_{\partial \bB^{n}_{\frac{1}{2}}} = b_{l}|_{\partial \bB^{n}_{\frac{1}{2}}},$
$$\begin{array}{cc}
\lim_{l \rightarrow \infty} \int_{\bB^{n}_{\frac{1}{2}}} |D\zeta_{l}-D\zeta|^{2} \ d \cH^{n} = 0, & \limsup_{l \rightarrow \infty} \sup |D\zeta_{l}| < \infty.
\end{array}$$
Also, define 
$$\begin{aligned} 
R_{l} & = (-1)^{n-1} \left( \partial (\bbE^{n+1} \res \{ x \in C_{\frac{1}{2}}: x_{n+1} > \varepsilon_{\ok_{l}} b_{l}(\bop(x)) \}) \right) \res C_{\frac{1}{2}} \\ 
S_{l} & = (-1)^{n-1} \left( \partial (\bbE^{n+1} \res \{ x \in C_{\frac{1}{2}}: x_{n+1} > \varepsilon_{\ok_{l}} \zeta_{l}(\bop(x)) \}) \right) \res C_{\frac{1}{2}}, 
\end{aligned}$$ 
(note the slight difference from the corresponding definitions in \cite{HS79}). Using \eqref{6.4(7)} with $\sigma = \sigma_{l},$ $\ok=\ok_{l}$ (as well as the analogous identity over $\bW_{\sigma_{l}}$), \eqref{6.4(1)}, and \eqref{appendixlemma2projectionmass} with $T=T_{\ok_{l}},$ $\sigma=\sigma_{l}$ we compute
\begin{equation} \label{6.4(15)}
\bM(F^{\sigma_{l}}_{\ok_{l} \sharp}(T_{\ok_{l}} \res C_{1})) = \bM(F^{\sigma_{l}}_{\ok_{l} \sharp}(T_{\ok_{l}} \res C_{1})-R_{l}) + \bM(R_{l})
\end{equation}
because $F^{\sigma_{l}}_{\ok_{l}}|_{C_{\frac{1}{2}}} = L^{\sigma_{l}}_{\ok_{l}}|_{C_{\frac{1}{2}}}$; this is as in 6.4(15) of \cite{HS79}.

\medskip

From the area minimality of $T_{\ok_{l}}$ and the equations
$$\partial (T_{\ok_{l}} \res C_{1}) = \partial F^{\sigma_{\ok}}_{\ok_{l}}(T_{\ok_{l}} \res C_{1}) + \partial P_{l}, \ \partial (R_{l}-S_{l}) = 0,$$
where
$$P_{l} = Q_{\ok_{l}} - F^{\sigma_{l}}_{\ok_{l} \sharp} Q_{\ok_{l}}$$
with $Q_{\ok_{l}}$ as in the proof of Lemma \ref{sec6.2} or Lemma \ref{appendixlemma2}, we deduce that
\begin{equation} \label{6.4(16)}
\begin{aligned}
\bM(T_{\ok_{l}} \res C_{1}) \leq & \bM \left( F^{\sigma_{l}}_{\ok_{l} \sharp}(T_{\ok_{l}} \res C_{1}) + P_{l} - R_{l} + S_{l} \right) \\
\leq & \bM \left( F^{\sigma_{l}}_{\ok_{l} \sharp}(T_{\ok_{l}} \res C_{1})-R_{l} \right) + \bM(P_{l}) + \bM(S_{l});
\end{aligned}
\end{equation}
this is as in 6.4(16) of \cite{HS79}. Combining \eqref{6.4(14)},\eqref{6.4(15)}, and \eqref{6.4(16)} gives the inequality
\begin{equation} \label{6.4(17)}
\bM(R_{l})-\bM(S_{l}) \leq \bM(P_{l}) + (2c_{30}+c_{33})(1+\Lambda)^{2} \sigma_{l} \varepsilon^{2}_{\ok_{l}},
\end{equation}
as in 6.4(17) of \cite{HS79}. Using the definition of $F^{\sigma_{l}}_{\ok_{l}},$ Lemma 26.25 of \cite{S83} or 4.1.14 of \cite{F69}, and \eqref{appendixlemma2retraction} with $r=\frac{3}{4}$ gives 
$$\limsup_{l \rightarrow \infty} \varepsilon^{-2}_{\ok_{l}} \bM(P_{l}) \leq \limsup_{l \rightarrow \infty} \left( 1+ (\Lip(F^{\sigma_{l}}_{\ok_{l}}))^{n} \right) \varepsilon^{-2}_{\ok_{l}} \bM(Q_{\ok_{l}} \res C_{\frac{3}{4}}) = 0$$
by \eqref{6.1(2)}. We deduce from \eqref{6.4(17)} that
$$\begin{aligned}
0 \geq & \limsup_{k \rightarrow \infty} \varepsilon^{-2}_{\ok_{l}} (\bM(R_{l})-\bM(S_{l})) \\
= & \limsup_{l \rightarrow \infty} \varepsilon^{-2}_{\ok_{l}} \left( \int_{\bB^{n}_{\frac{1}{2}}} \sqrt{1+ \varepsilon^{2}_{\ok_{l}} |D b_{l}|^{2}} \ d \cH^{n} - \int_{\bB^{n}_{\frac{1}{2}}} \sqrt{1+ \varepsilon^{2}_{\ok_{l}} |D \zeta_{l}|^{2}} \ d \cH^{n} \right) \\
= & \limsup_{l \rightarrow \infty} \varepsilon^{-2}_{\ok_{l}} \int_{\bB^{n}_{\frac{1}{2}}} \frac{|D b_{l}|^{2} - |D \zeta_{l}|^{2}}{\sqrt{1+\varepsilon^{2}_{\ok_{l}} |D b_{l}|^{2}} + \sqrt{1+\varepsilon^{2}_{\ok_{l}} |D \zeta_{l}|^{2}}} \ d \cH^{n} \\
= & \frac{1}{2} \int_{\bB^{n}_{\frac{1}{2}}} |D b|^{2} - |D \zeta|^{2} \ d \cH^{n}.
\end{aligned}$$
Since this holds for all Lipschitz $\zeta:\bB^{n}_{\frac{1}{2}} \rightarrow \bbR$ with $\zeta|_{\partial \bB^{n}_{\frac{1}{2}}} = b,$ then we conclude $b$ minimizes the Dirichlet integral and so is a harmonic function. In particular, $b$ is differentiable, and hence $\beta_{M} = \gamma_{m}.$

\medskip

Since $\beta_{1} \leq \beta_{2} \leq \cdots \leq \beta_{M} = \gamma_{m} \leq \ldots \leq \gamma_{1},$ then we can use a similar argument to show $\beta_{1} = \gamma_{1}$ to complete the proof of the first conclusion.

\medskip

To prove the second conclusion, we assume $\sigma \in (0,\frac{(1-\rho)}{2})$ and deduce from \eqref{6.4(2)} and \eqref{6.4(3)} that
$$\limsup_{\ok \rightarrow \infty} \sup_{(\spt T_{\ok}) \setminus H^{\sigma/2}} |\varepsilon^{-1}_{\ok} \boq_{1} - \beta_{1} \boq_{0}| = 0$$
and from Lemma \ref{sec6.3} that
$$\limsup_{\ok \rightarrow \infty} \sup_{H^{\sigma/2} \cap C_{\rho} \cap \spt T_{\ok}} |\varepsilon^{-1}_{\ok} \boq_{1} - \beta_{1} \boq_{0}| \leq \frac{|\beta_{1}| \sigma}{2} + \frac{|\beta_{1}| \sigma}{2},$$
and let $\sigma \rightarrow 0.$ \end{proof}

\section{Comparison of spherical and cylindrical excess} \label{sec7}

Sections \ref{sec7.1},\ref{sec7.3}, analogous to sections 7.1,7.3 of \cite{HS79}, give bounds for the cylindrical excess (at smaller radii) in terms of the spherical excess, for $T \in \sT$ with small cylindrical excess. We use Lemma \ref{sec7.1} and Theorem \ref{sec7.3} to prove $C^{2}$ boundary regularity for harmonic blowups in section \ref{sec8}. Section \ref{sec7.2}, analogous to section 7.2 of \cite{HS79}, gives a general lemma about homogeneous degree one harmonic functions over $\bV.$ We introduce $c_{34},\ldots,c_{37}.$

\medskip

Theorem \ref{sec7.3} requires $m \geq 1,$ but otherwise we assume $m \in \{ 0,\ldots,M-1 \}.$

\subsection{Lemma} \label{sec7.1}

\begin{lemma} 
There exists positive constants $c_{34} \geq 1+c_{14},$ $c_{35},$ and $c_{36},$ all depending on $n,M,m,$ so that if $T \in \sT,$ 
$$\begin{aligned}
\bE_{C}(T,1) + \kappa_{T} \leq & c_{34}^{-1}, \\
\sup_{C_{\frac{1}{4}} \cap \spt T} \boq_{1}^{2} \leq & c_{35}^{-1} \bE_{S}(T,1),
\end{aligned}$$
then
$$\bE_{C}(T,1/3) \leq c_{36} (\bE_{S}(T,1)+\kappa_{T}).$$
\end{lemma}

We note that we correct a typo in the statement appearing in Lemma 7.1 of \cite{HS79}, where in the second inequality of the hypothesis we need $\bE_{S}(T,1)$ and not $\bE_{C}(T,1)$ as it appears in \cite{HS79}. The proof follows using the first variation Lemma \eqref{sec4.1}.

\medskip

\begin{proof} The calculations carry over exactly from section 7.1 of \cite{HS79}, although we make a clarification for the reader. 

\medskip

With $c_{4}$ as in \eqref{2.2(5)} and $c_{12},c_{14},c_{15}$ from Lemma \ref{sec4.1}, all now depending on $n,M,m,$ we let
$$\begin{aligned} 
c_{34} & = 2^{2n+2} (1+c_{4})(1+c_{14})(1+c_{15}) \\
c_{35} & = 3^{2n+8} (1+M \varpi_{n}) c_{12}, \\ 
c_{36} & = 4^{3n+6} (1+M \varpi_{n} + c_{4}) c_{12}, 
\end{aligned}$$ 
We now assume for contradiction that $T \in \sT$ satisfies the hypothesis of the lemma, but that 
$$\bE_{S}(T,1) + \kappa_{T} < c_{36}^{-1} \bE_{C}(T,1/3).$$ 
Using \eqref{4.1(2)} with $\sigma=\frac{1}{2}$ and $\bE_{C}(T,1)+\kappa_{T} \leq c_{34}^{-1}$ we compute
$$\sup_{C_{\frac{1}{2}} \cap \spt T} \boq_{1}^{2} \leq 2^{2n+1} c_{14}(1+c_{15})(\bE_{C}(T,1)+\kappa_{T}) \leq \frac{1}{2},$$
hence
\begin{equation} \label{7.1(1)}
C_{\frac{1}{2}} \cap \spt T \subset B_{1};
\end{equation}
this is equation 7.1(1) of \cite{HS79}.

\medskip

Let $\kappa = \kappa_{T},$ $\varepsilon = \bE_{C}(T,1/3)^{\frac{1}{2}},$ 
$$\chi(x) = \max \left\{ 0, \frac{x_{n+1}}{|x|} - \frac{4 \varepsilon}{\sqrt{c_{35}}} - \kappa \right\} \text{ for } x \in \bbR^{n+1}$$
and for each $k \in \bbN$ let $\zeta_{k} \in C^{1}(\bbR)$ be such that
$$\zeta_{k}(t) = \max \{ 0, t^{-n}-1 \}^{1+\frac{1}{k}} \text{ for } t \geq 1/4.$$

\medskip

The assumptions $\sup_{C_{\frac{1}{4}} \cap \spt T} \boq_{1}^{2} \leq c_{35}^{-1} \bE_{S}(T,1)$ and $\bE_{S}(T,1)+\kappa_{T} \leq c_{36}^{-1} \bE_{C}(T,1/3)$ then imply 
$$\sup_{x \in (\spt T) \cap C_{\frac{1}{4}} \setminus C_{\frac{1}{8}}} \frac{x_{n+1}}{|x|} \leq 8 c_{35}^{-\frac{1}{2}} \bE_{S}(T,1)^{\frac{1}{2}} \leq \frac{8 \varepsilon}{\sqrt{c_{35} c_{36}}}.$$ 
We can therefore choose a vector field $X_{k} \in C^{1}(\bbR^{n+1};\bbR^{n+1})$ so that
$$\begin{array}{ll}
X_{k}(x) = 0 & \text{ for } x \in \bB_{\frac{1}{4}} \cap \spt T \\ 
X_{k}(x) = \chi^{2}(x) \zeta_{k}(|x|) x & \text{ for } x \in \bbR^{n+1} \setminus \bB_{\frac{1}{4}}. 
\end{array}$$ 
As in \cite{HS79}, $X_{k}$ vanishes on 
$$(\spt \partial T) \cup (\bB_{\frac{1}{4}} \cap \spt T) \cup \left\{ x \in \bbR^{n+1} : x_{n+1} \leq |x| \left( \frac{4 \varepsilon}{\sqrt{c_{35}}} + \kappa \right) \right\}.$$
Using the first variation formula Lemma \ref{sec2.1} we compute
$$\begin{aligned}
0 = & \int \dive{T}{X_{k}} \ d \mu_{T} \\
= & \int_{B_{1} \setminus \bB_{\frac{1}{4}}} \chi^{2}(x) \left( |x| \zeta_{k}'(|x|) \left( 1- \left( \frac{x}{|x|} \cdot \cV^{T} \right)^{2} \right) - n \zeta_{k}(|x|)  \right) \ d \mu_{T}(x) \\
& + \int_{B_{1} \setminus \bB_{\frac{1}{4}}} 2 \chi(x) \zeta_{k}(|x|) x \cdot \nabla^{T} \left( \frac{x_{n+1}}{|x|} \right) \ d \mu_{T}(x).
\end{aligned}$$

Computing as well

$$\left| x \cdot \nabla^{T} \left( \frac{x_{n+1}}{x} \right) \right| = \left| \frac{x_{n+1}}{|x|} \left( \frac{x}{|x|} \cdot \cV^{T} \right)^{2} - \cV^{T}_{n+1} \left( \frac{x}{|x|} \cdot \cV^{T} \right) \right| \leq 2 \left| \frac{x}{|x|} \cdot \cV^{T} \right|,$$

we let $k \rightarrow \infty$ and deduce that

$$\begin{aligned}
\frac{n}{2^{n}} \int_{B_{1} \setminus \bB_{\frac{1}{4}}} \chi^{2} \ d \mu_{T} \leq & n \int_{B_{1} \setminus \bB_{\frac{1}{4}}} \frac{\chi^{2}(x)}{|x|^{n}} \left( \frac{x}{|x|} \cdot \cV^{T} \right)^{2} \ d \mu_{T}(x) \\
& + 4 \int_{B_{1} \setminus \bB_{\frac{1}{4}}} \frac{\chi(x)}{|x|^{n}} \left| \frac{x}{|x|} \cdot \cV^{T} \right| 
 d \mu_{T}(x) \\
\leq & 5n4^{n} \int_{B_{1} \setminus \bB_{\frac{1}{4}}} \chi(x) \left| \frac{x}{|x|} \cdot \cV^{T} \right| \ d \mu_{T}(x). 
\end{aligned}$$
Similarly, we verify that
$$\begin{aligned}
\int_{B_{1} \setminus \bB_{\frac{1}{4}}} \min & \left\{ 0, \frac{x_{n+1}}{|x|} + \frac{4 \varepsilon}{\sqrt{c_{35}}} + \kappa \right\}^{2} \ d \mu_{T}(x) \\
& \leq 4^{2n+3} \int_{B_{1} \setminus \bB_{\frac{1}{4}}} \left( \frac{x}{|x|} \cdot \cV^{T} \right)^{2} d \mu_{T}(x).
\end{aligned}$$ 
Adding these inequality we infer that
$$\begin{aligned}
\int_{B_{1} \setminus \bB_{\frac{1}{4}}} \boq_{1}^{2} \ d \mu_{T} \leq & \int_{B_{1} \setminus \bB_{\frac{1}{4}}} \frac{x^{2}_{n+1}}{|x|^{2}} \ d \mu_{T}(x) \\
\leq & 4 \left( \frac{4 \varepsilon}{\sqrt{c_{35}}} + \kappa^{2} \right) \mu_{T}(B^{n}_{1} \setminus \bB_{\frac{1}{4}}) \\
& + 4^{2n+4} \int_{B_{1} \setminus \bB_{\frac{1}{4}}} \left( \frac{x}{|x|} \cdot \cV^{T} \right)^{2} \ d \mu_{T}(x).
\end{aligned}$$
We also compute, using $\sup_{C_{\frac{1}{4}} \cap \spt T} \boq_{1}^{2} \leq c_{35}^{-1} \bE_{S}(T,1)$ and $\bE_{S}(T,1)+\kappa_{T} \leq c_{36}^{-1} \bE_{C}(T,1/3),$ that
$$\int_{\bB_{\frac{1}{4}}} \boq_{1}^{2} \ d \mu_{T} \leq (c_{35}^{-1} c_{36}^{-1} \varepsilon^{2}) \mu_{T}(\bB_{\frac{1}{4}}) \leq 16 c^{-1}_{35} \varepsilon^{2} \mu_{T}(\bB_{\frac{1}{4}}).$$

Using \eqref{2.3(1)}, \eqref{4.1(1)} (with $T,\sigma$ replaced by $(\ETA_{\frac{1}{3} \sharp} T) \res B_{3},1/2$), \eqref{7.1(1)}, and \eqref{2.2(5)}, we conclude 
  
$$\begin{aligned} 
\varepsilon^{2} \leq & 3^{n+2} c_{12} \left( \kappa + \int_{C_{\frac{1}{2}}} \boq_{1}^{2} \ d \mu_{T} \right) \\ 
\leq & 3^{n+2} c_{12} \left( \kappa + \int_{B_{1}} \boq_{1}^{2} \ d \mu_{T} \right) \\ 
\leq & 3^{n+2} c_{12} \Big( \kappa + 16 c^{-1}_{35} \varepsilon^{2} \mu_{T}(\bB_{\frac{1}{4}}) + 64 c^{-1}_{35} \varepsilon^{2} \mu_{T}(B_{1} \setminus \bB_{\frac{1}{4}}) \\ 
& + 4 \kappa^{2} \mu_{T}(B_{1} \setminus \bB_{\frac{1}{4}}) + 4^{2n+4} \int_{B_{1} \setminus \bB_{\frac{1}{4}}} \left( \frac{x}{|x|} \cdot \cV^{T} \right)^{2} \ d \mu_{T}(x) \Big) \\ 
\leq & 3^{2n+4} c_{12} (1+ M \varpi_{n})(16 c^{-1}_{35} \varepsilon^{2} + \kappa) + 4^{3n+6} c_{12} (\bE_{S}(T,1) + c_{4} \kappa) \\ 
\leq & (\varepsilon^{2}/2) + \frac{c_{36}}{2} (\bE_{S}(T,1) + \kappa)
\end{aligned}$$ 
which is a contradiction. \end{proof}

\subsection{Remark} \label{sec7.2} 

As in section 7.2 of \cite{HS79}, we make a general observation about homogeneous degree one harmonic functions over $\bV.$ We use this in the proof of Theorem \ref{sec7.3} in order to apply Lemma \ref{sec6.4}.

\medskip

Suppose $h: \bV \rightarrow \bbR$ is harmonic and $h(\rho y) = \rho h(y)$ for all $\rho \in (0,1)$ and $y \in \bV.$ 
\begin{equation} \label{7.2(1)}
\begin{array}{ll}
\text{If $h$ is nonnegative, then $h$ has zero trace} \\
\text{(see section 26 of \cite{T75}) on $\bL.$}
\end{array}
\end{equation} 
\begin{equation} \label{7.2(2)}
\text{If $h$ has zero trace on $\bL,$ then $h= \beta \boq_{0} \res \bV$ for some $\beta \in \bbR.$}
\end{equation}
These are exactly as in 7.2(1)(2) of \cite{HS79}.

\medskip

\begin{proof} To verify \eqref{7.2(1)}, note that on the open hemisphere 
$$S^{n-1}_{+} = \{ y \in \bbR^{n}: |y|=1, \ y_{n}>0 \}$$
the spherical Laplacian has minimum eigenvalue $(n-1).$ Letting $h_{2}(y) = 2h(y/2)$ for $y \in S^{n-1}_{+}$ and choosing $\lambda>0$ so that the spherical domain $\Omega = \{ y \in S^{n-1}_{+}: \lambda y_{n} > h_{2}(y) \}$ is nonempty, we observe that $(\lambda \boq_{0} - h_{2})|_{\Omega}$ is a positive eigenfunction for the spherical Laplacian on $\Omega$ with eigenvalue $(n-1).$ If $h$ did not have zero trace on $\bL,$ then $S^{n-1}_{+} \setminus \Omega$ would have nonempty interior, which would contradict the strict monotonicity of the minimum eigenvalue of the spherical Laplacian with respect to the domain (see, for example, section 2.3 of \cite{BC76}). 

\medskip

Statement \eqref{7.2(2)} follows from the weak version of the Schwarz reflection principle. \end{proof}

\subsection{Theorem} \label{sec7.3}

We conclude here as in Theorem 7.3 of \cite{HS79}, although with $c_{37}$ with $n,M,m.$ For this, we require $m \geq 1,$ as we apply Theorem \ref{sec6.4}.

\begin{theorem}
Suppose $m \geq 1.$ There is a positive constant $c_{37}$ depending on $n,M,m$ so that if $T \in \sT,$
$$\begin{aligned}
\bE_{C}(T,1) + \kappa_{T} \leq & (2c_{34})^{-1} \text{ where $c_{34}$ is as in Lemma \ref{sec7.1},} \\
\bE_{C}(T,1/3) + \bE_{C}(T,1/3)^{-1} \kappa_{T} \leq & c_{37}^{-1}, \text{ and } \\
\bE_{C}(T,1/4) \leq & 2 \bE_{C}(\rot_{\theta \sharp}T,1/4) \text{ whenever } \theta \in (-1/8,1/8), \\
\end{aligned}$$
then
$$\bE_{C}(T,1/4) \leq c_{37} (\bE_{S}(T,1)+\kappa_{T}).$$
\end{theorem}

\begin{proof} There are only minor changes to the proof from section 7.3 of \cite{HS79}.

\medskip

Take for contradiction a sequence $\{ T_{\ok} \}_{\ok \in \bbN}$ in $\sT$ so that for each $\ok \in \bbN$
\begin{equation} \label{7.3(1)}
\begin{aligned}
\bE_{C}(T_{\ok},1) + \kappa_{T_{\ok}} \leq & (2c_{34})^{-1} \\
\bE_{C}(T_{\ok},1/4) \leq & 2 \bE_{C}(\rot_{\theta \sharp} T_{\ok},1/4) \text{ whenever } \theta \in (-1/8,1/8),
\end{aligned}
\end{equation}
and so that as $\ok \rightarrow \infty$
\begin{equation} \label{7.3(2)}
\begin{aligned}
\bE(T_{\ok},1/3) + \bE_{C}(T_{\ok},1/3)^{-1} \kappa_{T_{\ok}} & \rightarrow 0 \\
\bE_{C}(T_{\ok},1/4)^{-1} (\bE_{S}(T_{\ok},1)+\kappa_{T_{\ok}}) & \rightarrow 0;
\end{aligned}
\end{equation}
these are as 7.3(1)(2) of \cite{HS79}. Letting $S_{\ok} = (\ETA_{\frac{1}{3} \sharp} T_{\ok}) \res B_{3},$ then Remark \eqref{sec2.3} implies $S_{\ok} \in \sT.$  With 
$$\begin{array}{cc}
\varepsilon_{\ok} = \bE_{C}(S_{\ok},1)^{\frac{1}{2}} = \bE_{C}(T_{\ok},1/3)^{\frac{1}{2}}, & \kappa_{\ok} = \kappa_{S_{\ok}} \leq \frac{\kappa_{T_{\ok}}}{3}
\end{array}$$
we can by Lemma \ref{sec6.1} assume $\{S_{\ok}\}_{\ok \in \bbN}$ is a blowup sequence with associated harmonic blowups $f_{i},g_{j}.$ We can compute using \eqref{1.4(1)} and \eqref{2.2(1)} 
$$\begin{aligned} 
\limsup_{\ok \rightarrow \infty} & \ \varepsilon_{\ok}^{-2}(\bE_{S}(S_{\ok},1)+c_{4} \kappa_{\ok}) \\ 
\leq & \limsup_{\ok \rightarrow \infty} (4/3)^{n} \bE_{C}(T_{\ok},1/4)^{-1} \Big( e^{c_{1} \kappa_{\ok}} \bE_{S}(T_{\ok},1) \\
& + ( e^{c_{1} \kappa_{\ok}}-1 )\left(\frac{M+m}{2}\right) \varpi_{n} + c_{4} \kappa_{\ok} \Big) = 0. 
\end{aligned}$$ 
Thus, we may apply \eqref{5.4(3)} (with $T$ replaced by $S_{\ok}$) and Definition \ref{sec6.1} (with $T_{\ok}$ replaced by $S_{\ok}$) to deduce that for each $i \in \{ 1,\ldots,M\}$ and $j \in \{ 1,\ldots,m \}$
$$\int_{\bV} \left( \frac{\partial}{\partial r} \left( \frac{f_{i}(y)}{|y|} \right) \right)^{2} |y|^{2-n} \ d \cH^{n}(y) = 0 = \int_{\bW} \left( \frac{\partial}{\partial r} \left( \frac{g_{j}(y)}{|y|} \right) \right)^{2} |y|^{2-n} \ d \cH^{n}(y),$$
so that for all $\rho \in (0,1)$
$$\begin{array}{cc}
f_{i}(\rho y) = \rho f_{i}(y) \text{ for } y \in \bV, & g_{j}(\rho y) = \rho g_{j}(y) \text{ for } y \in \bW.
\end{array}$$
By \eqref{7.2(1)} the nonnegative function $f_{M}-f_{1}$ has zero trace on $\bL.$ Thus by Lemma \ref{sec6.2}, each
$$\begin{aligned}
|f_{i}| = & (|f_{i}| - \min \{ |f_{1}|,\ldots,|f_{M}| \} ) + \min \{ |f_{1}|,\ldots,|f_{M}| \} \\
\leq & (f_{M}-f_{1}) + \min \{ |f_{1}|,\ldots,|f_{M}| \}
\end{aligned}$$
has zero trace on $\bL$ for each $i \in \{ 1,\ldots,M \}.$ By \eqref{7.2(2)} we conclude $f_{i} = \beta_{i} \boq_{0}|_{\bV}$ for some $\beta_{i} \in \bbR.$ By Lemma \ref{sec6.2} we as well conclude $\sum_{j=1}^{m} g_{j}$ has zero trace on $\bL$ because $\sum_{i=1}^{M} f_{i}$ does. From \eqref{7.2(1)} it follows that each function
$$\begin{aligned}
m|g_{j}| = & \left| \sum_{\tilde{j}=1}^{m} (g_{j}-g_{\tilde{j}}) +\sum_{\tilde{j}=1}^{m} g_{\tilde{j}} \right| \\
\leq & m (g_{m}-g_{1}) + \left| \sum_{\tilde{j}=1}^{m} g_{\tilde{j}} \right|
\end{aligned}$$
has zero trace on $\bL$; hence $g_{j} = \gamma_{j} \boq_{0}|_{\bW}$ for some $\gamma_{j} \in \bbR$ by \eqref{7.2(1)}.

\medskip

From Lemma \ref{sec6.4} (since $m \geq 1$) we conclude that 
$$\beta_{1} = \beta_{2} = \ldots = \beta_{M} = \gamma_{1} = \ldots = \gamma_{m},$$
\begin{equation} \label{7.3(4)}
\lim_{\ok \rightarrow \infty} \sup_{C_{\frac{7}{8}} \cap \spt S_{\ok}} |\varepsilon^{-1}_{\ok} \boq_{1} - \beta_{1} \boq_{0}| = 0;
\end{equation}
this is 7.3(4) of \cite{HS79} (note that there is no equation 7.3(3) in \cite{HS79}, as it was mistakenly skipped).

\medskip

From Lemma \ref{sec7.1}, \eqref{7.3(2)},\eqref{7.3(4)}, and the inequalities
$$\begin{aligned}
\sup_{C_{\frac{1}{4}} \cap \spt T_{\ok}} \boq_{1}^{2} \leq & \sup_{C_{\frac{7}{8}} \cap \spt S_{\ok}} \boq_{1}^{2}, \\
(3/4)^{n} \bE_{C}(T_{\ok},1/4) \leq & \bE_{C}(T_{\ok},1/3) = \varepsilon^{2}_{\ok} \leq 3^{n} \bE_{C}(T_{\ok},1),
\end{aligned}$$
we deduce that $\beta_{1} \neq 0.$ Letting $\theta_{\ok} = \arctan(\beta_{1} \varepsilon_{\ok}),$ we infer from Remark \ref{sec4.2}
$$R_{\ok} = (\ETA_{\frac{1}{4} \sharp} \rot_{\theta_{\ok} \sharp} T) \res B_{3} \in \sT$$
for all $\ok \in \bbN$ sufficiently large, and from \eqref{7.3(4)} and \eqref{4.1(1)} that
\begin{equation} \label{7.3(5)}
\begin{aligned}
\limsup_{\ok \rightarrow \infty} & \ \varepsilon^{-2}_{\ok} \bE_{C}(\rot_{\theta_{\ok} \sharp} T_{\ok},1/4) \\
= & \limsup_{\ok \rightarrow \infty} \varepsilon^{-2}_{\ok} \bE_{C}(T_{\ok},1) \\
\leq & \limsup_{\ok \rightarrow \infty} (c_{12}c_{13} \sup_{C_{\frac{7}{6}} \cap \spt R_{\ok}} \varepsilon^{-2}_{\ok} \boq_{1}^{2} + c_{12} \varepsilon^{-2}_{\ok} \kappa_{\ok}) = 0;
\end{aligned}
\end{equation}
this is as in 7.3(5) of \cite{HS79}. On the other hand, by \eqref{4.1(2)} and \eqref{7.3(4)},
\begin{equation} \label{7.3(6)}
\begin{aligned}
\liminf_{\ok \rightarrow \infty} & \ \varepsilon^{-2}_{\ok} \bE_{C}(T_{\ok},1/4) \\
\geq & \liminf_{\ok \rightarrow \infty} \big( (8^{2n+1}c_{14}c_{15})^{-1} \sup_{C_{\frac{1}{8}} \cap T_{\ok}} \varepsilon^{-2}_{\ok} \boq_{1}^{2} - (8^{n+1} c_{15})^{-1} \varepsilon^{-2}_{\ok} \kappa_{\ok} \big) \\
= & (8^{2n+1}c_{14}c_{15})^{-1}(\beta_{1}/8)^{2} > 0;
\end{aligned}
\end{equation}
this is as in 7.3(6) of \cite{HS79}. Statements \eqref{7.3(5)} and \eqref{7.3(6)} now contradict \eqref{7.3(1)} for $\ok$ sufficiently large. \end{proof}

\section{Boundary regularity of harmonic blowups} \label{sec8}

We remark as in \cite{DS93}, that a well-known key step in various regularity proofs for minimizing currents is to derive a decay rate $\bE_{C}(\rot_{\vartheta \sharp} T,\rho) \approx \rho^{\alpha}$ as $\rho \searrow 0$ for $T \in \sT$ with $\bE_{C}(T,1)+\kappa_{T}$ sufficiently small for a suitable angle $\vartheta$; this is done in Theorem \ref{sec9.2}. However, proving this decay rate for the cylindrical excess requires we prove all harmonic blowups corresponding to all blowup sequences are given by $C^{2}$ functions over $\bV \cup \bL$ and $\bW \cup \bL$; this is Theorem \ref{sec8.2}. We only need to make a clarification to the end of the proof of Lemma \ref{sec8.1}, which is as in Lemma 8.1 of \cite{HS79}. We introduce $c_{38},\ldots,c_{43}$ now depending on $n,M,m.$

\medskip

We assume throughout this section that $m \geq 1.$

\subsection{Lemma} \label{sec8.1}

To proof of Theorem \ref{sec8.2} relies on an application of the Hopf boundary point lemma. For this, we must first show that all harmonic blowups are roughly differentiable at the origin, in the sense of the following lemma.

\begin{lemma}
Suppose $m \geq 1.$ If $\{ T_{\ok} \}_{\ok \in \bbN}$ is a blowup sequence in $\sT$ with associated harmonic blowups $f_{i},g_{j}$ then
$$\begin{array}{cc}
\sup_{y \in \bV \cap B^{n}_{\rho}} \frac{|f_{i}(y)|}{|y|} < \infty, & \sup_{y \in \bW \cap B^{n}_{\rho}} \frac{|g_{j}(y)|}{|y|} < \infty,
\end{array}$$
for each $\rho \in (0,1),$ $i \in \{ 1,\ldots,M \},$ and $j \in \{ 1,\ldots,m \}.$
\end{lemma}

\begin{proof} For each $\ok \in \bbN,$ letting $\varepsilon_{\ok} = \bE_{C}(T_{\ok},1)^{\frac{1}{2}}$ and $\kappa_{\ok} = \kappa_{T_{\ok}},$ choose for each $\sigma \in (0,1/12]$ a $\theta(\ok,\sigma) \in [-1/8,1/8]$ so that
\begin{equation} \label{8.1(1)}
\bE_{C}(\rot_{\theta(\ok,\sigma) \sharp}T,\sigma/4) \leq 2 \bE_{C}(\rot_{\theta \sharp} T,\sigma/4) \text{ whenever } |\theta| \leq 1/8;
\end{equation}
this is as in 8.1(1) of \cite{HS79}. From \eqref{4.1(1)},\eqref{4.1(2)}, \eqref{4.2(3)}, \eqref{1.4(1)}, and \eqref{6.1(1)},\eqref{6.1(2)}, we infer that
\begin{equation} \label{8.1(2)}
\theta(\ok,\sigma) \rightarrow 0 \text{ and } \bE_{C}(\rot_{\theta(\ok,\sigma) \sharp} T_{\ok},\sigma) \rightarrow 0 \text{ as } \ok \rightarrow \infty;
\end{equation}
this is as in 8.1(2) of \cite{HS79}.

\medskip

We now consider two possibilities:
\begin{enumerate}
 \item[] \emph{Case 1.} $\bE_{C}(\rot_{\theta(\ok,\sigma) \sharp} T_{\ok}, \sigma/3) < \varepsilon^{2}_{\ok}$ for infinitely many $\ok \in \bbN.$
 \item[] \emph{Case 2.} $\bE_{C}(\rot_{\theta(\ok,\sigma) \sharp} T_{\ok}, \sigma/3) \geq \varepsilon^{2}_{\ok}$ for all sufficiently large $\ok \in \bbN.$
\end{enumerate}

\medskip

In \emph{Case 2}, using \eqref{8.1(1)},\eqref{8.1(2)}, \eqref{4.2(3)},\eqref{4.2(4)}, and \eqref{6.1(1)},\eqref{6.1(2)},\eqref{6.1(3)},\eqref{6.1(4)} we can choose $N_{\sigma} \in \bbN$ so that, for all $\ok \geq N_{\sigma},$ we have $\kappa_{\ok} \leq \varepsilon^{2}_{\ok},$
$$\begin{aligned} 
S_{\ok} = & (\rot_{\theta(\ok,\sigma) \sharp} \ETA_{\sigma \sharp} T_{\ok}) \res B_{3} \in \sT \\
\bE_{C}(S_{\ok},1) + \kappa_{S_{\ok}} \leq & (4 \sigma)^{-n} \bE_{C}((\ETA_{\frac{1}{4} \sharp} \rot_{\theta(\ok,\sigma) \sharp} T_{\ok}) \res B_{3},1) + \sigma^{\alpha} \kappa_{\ok} \\ 
\leq & (2c_{34})^{-1} \\ 
\bE_{C}(S_{\ok},1/3)+\bE_{C}(S_{\ok},1/3)^{-1} \kappa_{S_{\ok}} \leq & \left(\frac{3}{4\sigma} \right)^{n} \bE_{C}((\ETA_{\frac{1}{4} \sharp} \rot_{\theta(\ok,\sigma) \sharp} T_{\ok}) \res B_{3},1) \\ 
& + \varepsilon^{-2}_{\ok} \sigma^{\alpha} \kappa_{\ok} \leq c^{-1}_{37}, \\ 
\bE_{C}(S_{\ok},1/4) \leq & 2 \bE_{C}(\rot_{\theta \sharp} S_{\ok},1/4) \text{ whenever } |\theta| \leq 1/8. 
\end{aligned}$$ 
These are exactly the assumptions needed to apply Theorem 7.3 with $T = S_{\ok}$ for $\ok \geq N_{\sigma}$, and so we conclude by \eqref{2.3(2)}, \eqref{1.6(1)}, and \eqref{6.1(2)}
$$\begin{aligned} 
\bE_{C}(\rot_{\theta(\ok,\sigma) \sharp}, \sigma/4) = & \bE_{C}(S_{\ok},1/4) \\
\leq & c_{37} (\bE_{S}(S_{\ok},1) + \kappa_{S_{\ok}}) \\ 
\leq & c_{37} (\bE_{S}(T_{\ok},\sigma) + \sigma^{\alpha} \kappa_{\ok}) \\ 
\leq & c_{38} (\bE_{S}(T_{\ok},1) + \kappa_{\ok}] 
\\ \leq & c_{39} \varepsilon^{2}_{\ok}.
\end{aligned}$$
Here, $c_{38}$ depends on $n,M,m$ as we used 
$$\bE_{S}(T_{\ok},\sigma) \leq e^{c_{1} \kappa_{\ok}} \bE_{S}(T_{\ok},1) + (e^{c_{1} \kappa_{\ok}}-1) \left(\frac{M+m}{2}\right) \varpi_{n}$$ 
by \eqref{2.2(1)}; thus, $c_{39}$ also depends on $n,M,m.$

\medskip

If instead \emph{Case 1} occurs so that $\bE_{C}(\rot_{\theta(\ok,\sigma) \sharp} T_{\ok},\sigma/3) < \varepsilon^{2}_{\ok}$ for infinitely many $\ok \in \bbN,$ then we use \eqref{1.4(1)} to conclude that in either \emph{Case 1} or \emph{Case 2} 
$$\bE_{C}(\rot_{\theta(\ok,\sigma) \sharp} T_{\ok},\sigma/4) \leq c_{40} \varepsilon^{2}_{\ok}$$ 
for infinitely many $\ok \in \bbN,$ where $c_{40}$ depends on $n,M,m.$ Since $S_{\ok} \in \sT,$ then we can apply \eqref{2.3(1)}, \eqref{4.1(2)} with $T,\sigma$ replaced by $(\rot_{\theta(\ok,\sigma) \sharp} \ETA_{\frac{\sigma}{4} \sharp} T_{\ok}) \res B_{3},1/20,$ as well as \eqref{6.1(2)} to conclude
\begin{equation} \label{8.1(3)}
\sup_{C_{\frac{\sigma}{5}} \cap \spt \rot_{\theta(\ok,\sigma) \sharp} T_{\ok}} |\boq_{1}| \leq c_{41} \varepsilon_{\ok} \sigma
\end{equation}
for infinitely many $\ok \in \bbN,$ with $c_{41}$ now depending on $n,M,m$; this is as in 8.1(3) of \cite{HS79}.

\medskip

Letting $\overline{y} = (y_{1},\ldots,y_{n-1},-y_{n})$ for $y \in \bbR^{n}$ and letting $v^{(\ok)}_{i},w^{(\ok)}_{j}$ be associated with $T_{\ok}$ as in Definition \ref{sec6.1}, we note that $\theta(\ok,\sigma) \rightarrow 0$ as $\ok \rightarrow \infty,$ and we use \eqref{8.1(3)} to estimate, for each $\tau \in (0,1),$
$$\begin{array}{cc}
|v^{(\ok)}_{i}(y)-v^{(\ok)}_{\tilde{i}}(y)| \leq 2c_{41} \varepsilon_{\ok} \sigma & \text{ for } y \in \bV_{\tau} \cap B^{n}_{\frac{\sigma}{5}}, \\
|w^{(\ok)}_{j}(y)-w^{(\ok)}_{\tilde{j}}(y)| \leq 2c_{41} \varepsilon_{\ok} \sigma & \text{ for } y \in \bW_{\tau} \cap B^{n}_{\frac{\sigma}{5}}, \\
|v^{(\ok)}_{i}(y)+w^{(\ok)}_{j}(\overline{y})| \leq 2c_{41} \varepsilon_{\ok} \sigma & \text{ for } y \in \bV_{\tau} \cap B^{n}_{\frac{\sigma}{5}}, \\
\end{array}$$
for infinitely many $\ok \in \bbN$ and for all $\{ i,\tilde{i} \} \subseteq \{ 1,\ldots,M \}$ and all $\{ j,\tilde{j} \} \subseteq \{ 1,\ldots,m \}.$

\medskip

From the arbitrariness of $\sigma$ and \eqref{6.1(3)},\eqref{6.1(4)} we infer that
\begin{equation} \label{8.1(4)}
\begin{array}{cc}
|f_{i}(y)-f_{\tilde{i}}(y)| \leq 8c_{41} |y| & \text{ for } y \in \bV \cap B^{n}_{\frac{1}{60}}, \\
|g_{j}(y)-g_{\tilde{j}}(y)| \leq 8c_{41} |y| & \text{ for } y \in \bW \cap B^{n}_{\frac{1}{60}}, \\
|f_{i}(y)+g_{j}(\overline{y})| \leq 8c_{41} |y| & \text{ for } y \in \bV \cap B^{n}_{\frac{1}{60}}, \\
\end{array}
\end{equation}
for all $\{ i,\tilde{i} \} \subseteq \{ 1,\ldots,M \}$ and all $\{ j,\tilde{j} \} \subseteq \{ 1,\ldots,m \}$; this is 8.1(4) of \cite{HS79}.

\medskip

We may also apply Lemma \ref{sec6.2} to see that the functions
$$\begin{array}{lll}
\Pi:B^{n}_{1} \rightarrow \bbR, & \mathcal{P}:B^{n}_{1} \rightarrow \bbR & \\
\Pi|_{\bV} = \sum_{i=1}^{M} f_{i}, & \Pi|_{\bW} = \sum_{j=1}^{m} g_{j}, & \Pi|_{\bL}=0, \\
\mathcal{P}(y)=\Pi(y)-\Pi(\overline{y}) & \text{for } y \in B^{n}_{1} &
\end{array}$$
have locally square integrable weak gradients. Since $\mathcal{P}$ is odd in the second variable and $\mathcal{P}|_{\bV \cup \bW}$ is harmonic, $\mathcal{P}$ is, by the weak version of the Schwarz reflection principle, harmonic on all of $B^{n}_{1}$; hence, for each $\rho \in (0,1)$
\begin{equation} \label{8.1(5)}
\sup_{y \in B^{n}_{\rho}} \frac{|\mathcal{P}(y)|}{|y|} < \infty;
\end{equation}
this is 8.1(5) of \cite{HS79}.

\medskip

On the other hand, for $y \in \bV$ 
$$\begin{aligned} 
\mathcal{P}(y) = & \sum_{i=1}^{M} f_{i}(y) - \sum_{j=1}^{m} g_{j}(\overline{y}) \\
= & \sum_{i=1}^{M} f_{i}(y) - \sum_{j=1}^{m} (f_{j}(y)+ g_{j}(\overline{y})) + \sum_{j=1}^{m} f_{j}(y) \\ 
= & 2 \sum_{i=1}^{m} f_{i}(y) + \sum_{i=m+1}^{M} f_{i} - \sum_{j=1}^{m} (f_{j}(y)+ g_{j}(\overline{y})) \\ 
= & (M+m) f_{\tilde{i}}(y) + 2 \sum_{i=1}^{m} (f_{i}(y)-f_{\tilde{i}}(y)) + \\
& + \sum_{i=m+1}^{M} (f_{i}(y)-f_{\tilde{i}}(y)) - \sum_{j=1}^{m} (f_{j}(y)+ g_{j}(\overline{y})) 
\end{aligned}$$ 
for each $\tilde{i} \in \{1,\ldots,M\},$ and for any $j \in \{1,\ldots,m\}$ 
$$g_{j}(\overline{y}) = (f_{1}(y)+g_{j}(\overline{y})) - f_{1}(y).$$ 
The theorem thus follows by \eqref{6.1(5)} and \eqref{8.1(4)},\eqref{8.1(5)}. \end{proof}

\subsection{Theorem} \label{sec8.2}

The following theorem will be instrumental in showing the necessary decay rate Theorem \ref{sec9.2} for the excess; in particular, $\beta$ as given below will be used to find the necessary angle $\vartheta$ as in Theorem \ref{sec9.2}. As such, we give a slightly more thorough conclusion, in stating specifically that $Df(0)=(0,\beta)=Dg(0),$ than as in the statement of Theorem 8.2 of \cite{HS79}, to which Theorem \ref{sec8.2} is analogous.

\begin{theorem}
Suppose $m \geq 1.$ If $\{ T_{\ok} \}_{\ok \in \bbN}$ is a blowup sequence in $\sT$ with associated blowups $f_{i},g_{j},$ then there exists two functions $f \in C^{2}(\bV \cup \bL)$ and $g \in C^{2}(\bW \cup \bL)$ and $\beta \in \bbR$ such that
$$\begin{array}{lll}
f|_{\bV} = f_{1} = f_{2} = \ldots = f_{M}, & & g|_{\bW} = g_{1} = \ldots = g_{m} \\
f|_{\bL} = 0 = g|_{\bL}, & \text{and} & Df(0)= (0,\beta) = Dg(0).
\end{array}$$
\end{theorem}

\begin{proof} The proof is precisely the same as in section 8.2 of \cite{HS79}, with only minor notational changes.

\medskip

By Lemma \ref{sec8.1}, the harmonic functions $f^{(\rho)}_{i},g^{(\rho)}_{j}$ defined for $i \in \{ 1,\ldots,M \},$ $j \in \{ 1,\ldots,m \},$ and $\rho \in (0,1/4]$ by
$$\begin{array}{ll}
f^{(\rho)}_{i}(y) = f_{i}(\rho y)/\rho & \text{for } y \in \bV, \\
g^{(\rho)}_{j}(y) = g_{j}(\rho y)/\rho & \text{for } y \in \bW,
\end{array}$$ 
are uniformly bounded. By Theorem 2.11 of \cite{GT83}, there is a decreasing sequence $\{ \rho_{l} \}_{l \in \bbN} \subset (0,1/4]$ so that $\rho_{l} \rightarrow 0$ as $l \rightarrow \infty$ and bounded harmonic functions $f^{\ast}_{1},\ldots,f^{\ast}_{M}$ on $\bV$ and $g^{\ast}_{1},\ldots,g^{\ast}_{m}$ on $\bW$ such that for all $i \in \{ 1,\ldots,M \}$ and $j \in \{ 1,\ldots,m \}$
$$\begin{array}{llll}
f^{(\rho_{l})}_{i} \rightarrow f^{\ast}_{i} & \text{and} & Df^{(\rho_{l})}_{i} \rightarrow Df^{\ast}_{i} & \text{pointwise on } \bV \text{ as } l \rightarrow \infty, \\ 
g^{(\rho_{l})}_{j} \rightarrow g^{\ast}_{j} & \text{and} & Dg^{(\rho_{l})}_{j} \rightarrow Dg^{\ast}_{j} & \text{pointwise on } \bW \text{ as } l \rightarrow \infty.
\end{array}$$
Noting that
$$\begin{aligned}
\sum_{i=1}^{M} & \int_{\bV} \left( \frac{\partial}{\partial r} \left( \frac{f_{i}(y)}{|y|} \right) \right)^{2} |y|^{2-n} \ d \cH^{n}(y) \\
& + \sum_{j=1}^{m} \int_{\bW} \left( \frac{\partial}{\partial r} \left( \frac{g_{j}(y)}{|y|} \right) \right)^{2} |y|^{2-n} \ d \cH^{n}(y) < \infty
\end{aligned}$$
by \eqref{5.4(3)} and \eqref{6.1(2)}, we use Fatou's lemma to compute for each $i \in \{ 1,\ldots,M \}$
$$\begin{aligned}
\int_{\bV} \left( \frac{\partial}{\partial r} \left( \frac{f^{\ast}_{i}(y)}{|y|} \right) \right)^{2} & |y|^{2-n} \ d \cH^{n}(y) \\
\leq & \liminf_{l \rightarrow \infty} \int_{\bV} \left( \frac{\partial}{\partial r} \left( \frac{f^{(\rho_{l})}_{i}(y)}{|y|} \right) \right)^{2} |y|^{2-n} \ d \cH^{n}(y) \\
\leq & \liminf_{l \rightarrow \infty} \int_{\bV \cap B^{n}_{\rho_{l}}} \left( \frac{\partial}{\partial r} \left( \frac{f_{i}(y)}{|y|} \right) \right)^{2} |y|^{2-n} \ d \cH^{n}(y) = 0,
\end{aligned}$$
hence $\frac{\partial}{\partial r} \left( \frac{f^{\ast}_{i}(y)}{|y|} \right) = 0$ for all $y \in \bV.$ Similarly, $\frac{\partial}{\partial r} \left( \frac{g^{\ast}_{j}(y)}{|y|} \right) = 0$ for all $y \in \bW.$ As in the proof of Theorem \ref{sec7.3}, it follows from Remark \ref{sec7.2} that each $f^{\ast}_{i}$ (respectively, $g^{\ast}_{j}$) is the restriction to $\bV$ (respectively, $\bW$) of some multiple of the linear function $\boq_{0}(y)=y_{n}.$ In the next paragraph, we shall verify that all of these multiples coincide.

\medskip

Define for each $l \in \bbN$ 
$$S^{l}_{\ok} = (\ETA_{\rho_{l} \sharp} T_{\ok}) \res B_{3},$$ 
and choose, by Remark \ref{sec2.3} and \eqref{6.1(1)},\eqref{6.1(2)}, an integer $N_{l} \geq l$ so that $S^{l}_{\ok} \in \sT$ whenever $\ok \geq N_{l}.$ Noting that, for each $i \in \{ 1,\ldots,M \},$ $j \in \{ 1,\ldots,m \},$ and $l \in \bbN$
$$\begin{array}{ll}
\varepsilon^{-1}_{\ok} v^{S^{l}_{\ok}}_{i} \rightarrow f^{(\rho_{l})}_{i} & \text{uniformly on compact subsets of } \bV, \\
\varepsilon^{-1}_{\ok} w^{S^{l}_{\ok}}_{j} \rightarrow g^{(\rho_{l})}_{j} & \text{uniformly on compact subsets of } \bW
\end{array}$$
as $\ok \rightarrow \infty,$ we may find integers $\ok_{l} \geq N_{l}$ so that 
\begin{equation} \label{8.2(1)} 
\begin{aligned} 
\max \left\{ \sup_{\bV_{\frac{1}{2}}} \left| f^{(\rho_{l})}_{1} \right|, \right. \ & \left. \sup_{\bV_{\frac{1}{2}}} \left| f^{(\rho_{k})}_{M} \right|, \ \sup_{\bW_{\frac{1}{2}}} \left| g^{(\rho_{l})}_{1} \right|, \ \sup_{\bW_{\frac{1}{2}}} \left| g^{(\rho_{l})}_{m} \right| \right\} \\ 
& \leq \sup_{C_{\frac{1}{2}} \cap \spt S^{l}_{\ok_{l}}} \varepsilon^{-1}_{\ok_{l}} |\boq_{1}|+\frac{1}{l}, 
\end{aligned} 
\end{equation}
and, by applying Lemma \ref{sec6.3} to $\{ T_{\ok} \}_{\ok \in \bbN}$ with $y=0$ and $\sigma = 3 \rho_{l},$ so that
\begin{equation} \label{8.2(2)} 
\begin{aligned} 
& \sup_{C_{3} \cap \spt S^{l}_{\ok_{l}}} \varepsilon^{-1}_{\ok_{l}} |\boq_{1}| \\ 
& \ \ \leq 3 \max \left\{ \sup_{\bV} \left|f^{(\frac{\rho_{l}}{3})}_{1} \right|, \sup_{\bV} \left|f^{(\frac{\rho_{l}}{3})}_{M} \right|, \sup_{\bW} \left| g^{(\frac{\rho_{l}}{3})}_{1} \right|, \sup_{\bW} \left| g^{(\frac{\rho_{l}}{3})}_{m} \right| \right\}+\frac{1}{l}, 
\end{aligned} 
\end{equation} 
and so that, with $S^{\ast}_{l} = S^{l}_{\ok_{l}},$
\begin{equation} \label{8.2(3)}
\begin{array}{cccc}
\varepsilon^{-1}_{\ok_{l}} v^{S^{\ast}_{l}}_{i} \rightarrow f^{\ast}_{i} & \text{and} & \varepsilon^{-1}_{\ok_{l}} w^{S^{\ast}_{l}}_{j} \rightarrow g^{\ast}_{j} & \text{as } l \rightarrow \infty;
\end{array}
\end{equation}
these three assumptions are as in 8.2(1)(2)(3) of \cite{HS79}. In case not all functions $f^{\ast}_{i},g^{\ast}_{j}$ are identically zero, using \eqref{4.1(1)},\eqref{4.1(2)} (with $T,\sigma$ replaced by $S^{\ast}_{l},1/2$) we see that \eqref{8.2(1)},\eqref{8.2(2)} along with \eqref{2.3(2)} and \eqref{6.1(2)} imply
$$0 < \liminf_{l \rightarrow \infty} \varepsilon^{-1}_{\ok_{l}} \bE_{C}(S^{\ast}_{l},1)^{\frac{1}{2}} \leq \limsup_{l \rightarrow \infty} \varepsilon^{-1}_{\ok_{l}} \bE_{C}(S^{\ast}_{l},1)^{\frac{1}{2}} < \infty,$$
and we may use Definition \ref{sec6.1} and \eqref{8.2(2)} to obtain a positive number $\lambda$ and a blowup sequence $\{ S^{\ast}_{l} \}_{l \in \bbN}$ whose associated harmonic blowups are $\lambda f^{\ast}_{i},\lambda g^{\ast}_{j}.$  It follows from Lemma \ref{sec6.2} and Lemma \ref{sec6.4} (as in the proof of Theorem \ref{sec7.3}) that, in any case
$$\begin{array}{cc} 
f^{\ast}_{1} = f^{\ast}_{2} = \ldots = f^{\ast}_{M} = \beta \boq_{0}|_{\bV} & g^{\ast}_{1} = \ldots = g^{\ast}_{m} = \beta \boq_{0}|_{\bW}
\end{array}$$
for some $\beta \in \bbR.$ 

\medskip

This now implies that the nonnegative harmonic functions $f_{M}-f_{1}$ and $g_{m}-g_{1}$ satisfy the conditions
$$\liminf_{t \searrow 0} \frac{f_{M}(0,t)-f_{1}(0,t)}{t} = 0 = \liminf_{t \searrow 0} \frac{g_{m}(0,-t)-g_{1}(0,-t)}{t}.$$
By the Hopf boundary point lemma (see Lemma 3.4 of \cite{GT83}), we conclude
$$f_{M}-f_{1} \equiv 0 \text{ on } \bV \text{ and } g_{m}-g_{1} \equiv 0 \text{ on } \bW;$$
hence,
$$f_{1} = f_{2} = \ldots = f_{M} \text{ on } \bV \text{ and } g_{1} = \ldots = g_{m} \text{ on } \bW.$$
By Lemma \ref{sec6.2} these functions all have zero trace on $\bL.$ Thus, there exist by the weak version of the Schwarz reflection principle functions $f \in C^{2}(\bV \cup \bL)$ and $g \in C^{2}(\bW \cup \bL)$ such that
$$\begin{array}{ccc} 
f|_{\bV} = f_{1} = f_{2} = \ldots = f_{M}, & g|_{\bW} = g_{1} = \ldots = g_{m}, & f|_{\bL} = 0 = g|_{\bL}.
\end{array}$$
Moreover, $Df(0)=(0,\beta)=Dg(0).$ \end{proof}

\subsection{Remark} \label{sec8.3}

Suppose $m \geq 1.$ By the Schwarz reflection principle, well-known $L^{2}$ a priori estimates for harmonic functions (see 5.2.5 of \cite{F69}), together with \eqref{6.1(5)} taking $\rho=1/2,$ we conclude there are positive constants $c_{42},c_{43}$ depending on $n,M,m$ so that for any harmonic blowups $f,g$ as in Theorem \ref{sec8.2}
\begin{equation} \label{8.3(1)}
\begin{aligned}
|\beta| & = |Df(0)| = |Dg(0)| \\
\leq & c_{42} \min \left\{ \left( \int_{\bV \cap B^{n}_{\frac{1}{2}}} |f|^{2} \ d \cH^{n} \right)^{\frac{1}{2}}, \left( \int_{\bW \cap B^{n}_{\frac{1}{2}}} |g|^{2} \ d \cH^{n} \right)^{\frac{1}{2}} \right\} \leq c_{43}, \\
\end{aligned}
\end{equation}
\begin{equation} \label{8.3(2)}
\begin{aligned}
|f(y) & - y \cdot Df(0)| \\
\leq & c_{42} \left( \int_{\bV \cap B^{n}_{\frac{1}{2}}} |f|^{2} \ d \cH^{n} \right)^{\frac{1}{2}} |y|^{2} \leq c_{43} |y|^{2} \text{ for } y \in \bV \cap B^{n}_{\frac{1}{4}},
\end{aligned}
\end{equation}
\begin{equation} \label{8.3(3)}
\begin{aligned}
|g(y) & - y \cdot Dg(0)| \\
\leq & c_{42} \left( \int_{\bW \cap B^{n}_{\frac{1}{2}}} |g|^{2} \ d \cH^{n} \right)^{\frac{1}{2}} |y|^{2} \leq c_{43} |y|^{2} \text{ for } y \in \bW \cap B^{n}_{\frac{1}{4}};
\end{aligned}
\end{equation}
these are as in 8.3(1)(2)(3) of \cite{HS79} (note that section 8.3 of \cite{HS79} is mislabeled in \cite{HS79} as section 8.2).

\section{Excess growth estimate} \label{sec9}

\medskip

Theorem \ref{sec9.2} is the central cylindrical excess decay lemma we need to prove Corollary \ref{sec9.3}, which together with the Hopf-type boundary point lemma from section \ref{sec10.1} directly implies our main result Theorem \ref{main}. All of the results and proofs hold from section 9 of \cite{HS79} without change. We introduce $c_{44},c_{45},c_{46}$ now depending on $n,M,m,$ and we assume $m \geq 1.$

\subsection{Theorem} \label{sec9.1}

As noted before, proving a result such as Theorem \ref{sec9.2} is standard in proving regularity results for area-minimizing currents. Moreover, the proof of Theorem \ref{sec9.2} is a standard iteration argument. For this, we first require Theorem \ref{sec9.1}. The proof of the following is exactly the same as in section 9.1 of \cite{HS79}.

\begin{theorem}
Suppose $m \geq 1.$ There is a constant $c_{44} \geq 1$ depending on $n,M,m$ so that corresponding to each $T \in \sT$ with $\max \{ \bE_{C}(T,1), c_{44} \kappa_{T} \} \leq c^{-1}_{44},$ there exists $\theta \in \bbR$ for which
$$\begin{aligned}
|\theta| \leq & c_{43} \bE_{C}(T,1)^{\frac{1}{2}} \\
\bE_{C}(\rot_{\theta \sharp}T,\tau) \leq & \tau \max \{ \bE_{C}(T,1),c_{44} \kappa_{T} \},
\end{aligned}$$
where $\tau = (c_{16}(1+c_{43}))^{-2/\alpha},$ with $c_{16}$ as in \eqref{4.2(4)} and $c_{43}$ as in Remark \ref{sec8.3}; thus $\tau$ depends only on $n,M,m.$
\end{theorem}

\begin{proof} If the theorem were false, then there would exist for each $\ok \in \bbN$ a current $T_{\ok} \in \sT$ so that, with $\varepsilon_{\ok} = \bE_{C}(T_{\ok},1)^{\frac{1}{2}}$ and $\kappa_{\ok} = \kappa_{T_{\ok}},$
\begin{equation} \label{9.1(1)}
\max \{ \varepsilon^{2}_{\ok}, \ok \kappa_{\ok} \} \leq \ok^{-1}
\end{equation}
\begin{equation} \label{9.1(2)}
\bE_{C}(\rot_{\theta \sharp} T_{\ok},\tau) > \tau \max \{ \varepsilon^{2}_{\ok}, \ok \kappa_{\ok} \} \text{ whenever } |\theta| \leq c_{43} \varepsilon_{\ok};
\end{equation}
these are as in 9.1(1)(2) of \cite{HS79}.

\medskip

We assume as in 9.1(3) of \cite{HS79} that
\begin{equation} \label{9.1(3)}
\ok \kappa_{\ok} < \tau^{-n-1} \varepsilon^{2}_{\ok} \text{ for all } \ok \in \bbN,
\end{equation}
otherwise \eqref{9.1(2)} is, by \eqref{1.4(1)}, contradicted with $\theta=0.$

\medskip

Replacing $\{ T_{\ok} \}_{\ok \in \bbN}$ by a subsequence, we may assume by Definition \ref{sec6.1} that $\{ T_{\ok} \}_{\ok \in \bbN}$ is a blowup sequence with associated harmonic blowups $f,g$ as in Theorem \ref{sec8.2}. By Theorem \ref{sec8.2} and \eqref{8.3(1)},
\begin{equation} \label{9.1(4)}
Df(0)=(0,\beta) = Dg(0) \text{ for some } \beta \in [-c_{43},c_{43}];
\end{equation}
this is as in 9.1(4) of \cite{HS79}. Letting $\theta_{\ok} = \arctan (\beta \varepsilon_{\ok})$ so that
\begin{equation} \label{9.1(5)}
|\theta_{\ok}| \leq |\beta| \varepsilon_{\ok} \leq c_{43} \varepsilon_{\ok},
\end{equation}
as in 9.1(5) of \cite{HS79}, we infer from \eqref{9.1(1)} and \eqref{4.2(4)} that
\begin{equation} \label{9.1(6)}
\begin{array}{ccc}
(\ETA_{\tau \sharp} \rot_{\theta_{\ok} \sharp} T_{\ok}) \res B_{3} \in \sT & \text{ and } & \kappa_{(\ETA_{\tau \sharp} \rot_{\theta_{\ok} \sharp} T_{\ok}) \res B_{3}} \leq \tau^{\alpha} \kappa_{\ok}
\end{array}
\end{equation}
for all $\ok \in \bbN$ sufficiently large; this is as in 9.1(6) of \cite{HS79}.

\medskip

Noting that $f|_{\bL}=0=g_{\bL}$ and that the functions $\varepsilon^{-1}_{\ok}v^{T_{\ok}}_{i},\varepsilon^{-1}_{\ok}w^{T_{\ok}}_{j}$ converge uniformly as $\ok \rightarrow \infty$ on $\bV_{\sigma},\bW_{\sigma}$ respectively, for each $\sigma \in (0,1)$ we conclude from Lemma \ref{sec6.3} that
$$\begin{aligned}
\lim_{\ok \rightarrow \infty} \sup_{C_{\frac{1}{2}} \cap \bop^{-1}(\bV) \cap \spt T_{\ok}} |\varepsilon^{-1}_{\ok} \boq_{1} - f \circ \bop| = & 0, \\
 \lim_{\ok \rightarrow \infty} \sup_{C_{\frac{1}{2}} \cap \bop^{-1}(\bW) \cap \spt T_{\ok}} |\varepsilon^{-1}_{\ok} \boq_{1} - g \circ \bop| = & 0.
\end{aligned}$$
From these equations, \eqref{9.1(4)}, and \eqref{8.3(2)},\eqref{8.3(3)} we deduce that
$$\begin{aligned}
\limsup_{\ok \rightarrow \infty} & \sup_{C_{2} \cap \spt ( \ETA_{\tau \sharp} \rot_{\theta_{\ok}} T_{\ok})} \varepsilon^{-1}_{\ok} |\boq_{1}| \\
\leq & \tau^{-1} \limsup_{\ok \rightarrow \infty} \sup_{C_{3 \tau} \cap \spt T_{\ok}} |\varepsilon^{-1}_{\ok} \boq_{1} - \beta \boq_{0}| \\
\leq & \tau^{-1} \limsup_{\ok \rightarrow \infty} \Big( \sup_{C_{3 \tau} \cap \bop^{-1}(\bV) \cap \spt T_{\ok}} |f \circ \bop - \beta \boq_{0}| \\
& + \sup_{C_{3 \tau} \cap \bop^{-1}(\bW) \cap \spt T_{\ok}} |g \circ \bop - \beta \boq_{0}| \Big) \\
\leq & 18 c_{43} \tau.
\end{aligned}$$
Using this estimate, \eqref{9.1(1)},\eqref{9.1(3)},\eqref{9.1(6)}, and \eqref{4.1(1)} (with $\sigma \nearrow 1$), we conclude that
$$\begin{aligned}
\limsup_{\ok \rightarrow \infty} \varepsilon^{-2}_{\ok} \bE_{C}(\rot_{\theta_{\ok} \sharp} T_{\ok},\tau) = & \limsup_{\ok \rightarrow \infty} \varepsilon^{-2}_{\ok} \bE_{C}(\ETA_{\tau \sharp} \rot_{\theta_{\ok} \sharp} T_{\ok},1) \\
\leq & (18)^{2} c_{12}c_{13}c^{2}_{43} \tau^{2} < \theta,
\end{aligned}$$
which, along with \eqref{9.1(5)}, contradicts \eqref{9.1(2)}. \end{proof}

\subsection{Theorem} \label{sec9.2}

Finally, we attend to the key excess decay estimate.

\begin{theorem}
Suppose $m \geq 1,$ and let $c_{43},c_{44},$ and $\tau$ be as in Remark \ref{sec8.3} and Theorem \ref{sec9.1}. For any $T \in \sT$ with $\max \{ \bE_{C}(T,1),c_{44} \kappa_{T} \} \leq c^{-1}_{44} \tau^{\alpha},$ there exists a $\vartheta \in [-2c_{43}c^{-1/2}_{44} \tau^{\alpha/2},2c_{43}c^{-1/2}_{44} \tau^{\alpha/2}]$ such that
$$\bE_{C}(\rot_{\vartheta \sharp} T,\rho) \leq c^{-1}_{44} \vartheta^{-(n+2) \alpha} \rho^{\alpha} \text{ for } \rho \in (0,\tau/4).$$
\end{theorem}

\begin{proof} The proof is a standard inductive argument using Theorem \ref{sec9.1}, and we closely follow the proof in section 9.2 of \cite{HS79}.

\medskip

By Theorem \ref{sec9.1} and \eqref{4.2(4)} we may inductively select real numbers $\{ \theta_{\ok} \}_{\ok \in \bbN}$ and rectifiable currents $\{ T_{\ok} \}_{\ok \in \bbN}$ so that, after setting $T_{0}=T,$ we have for each $\ok \in \bbN$
$$\begin{aligned}
|\theta_{\ok}| \leq & c_{43} \bE_{C}(T_{\ok-1},1)^{\frac{1}{2}}, \\
T_{\ok} = & (\ETA_{\tau \sharp} \rot_{\theta_{\ok} \sharp} T_{\ok-1}) \res B_{3} \in \sT, \text{ and} \\
\max \{ \bE_{C}(T_{\ok},1),c_{44} \kappa_{T_{\ok}} \} \leq & \tau^{\alpha} \max \{ \bE_{C}(T_{\ok-1},1),c_{44} \kappa_{T_{\ok-1}} \} \leq c^{-1}_{44} \tau^{(j+1) \alpha}.
\end{aligned}$$
In inductively checking that $T_{\ok}$ belongs to $\sT,$ we use \eqref{4.2(4)} and the definition of $\tau$ in Theorem \ref{sec9.1} to estimate
\begin{equation} \label{9.2(1)}
|\theta_{\ok}| \leq c_{43} c^{-1/2}_{44} \tau^{\ok \alpha/2},
\end{equation}
\begin{equation} \label{9.2(2)}
\bE_{C}(T_{\ok -1},1)+\kappa_{T_{\ok-1}} \leq 2 \max \{ \bE_{C}(T_{\ok-1},1), \kappa_{T_{\ok-1}} \} \leq 2c^{-1}_{44} \tau^{j \alpha};
\end{equation}
these are as in 9.2(1)(2) of \cite{HS79}. Letting $\vartheta_{\ok} = \sum_{l=1}^{\ok} \theta_{l}$ and $\vartheta = \sum_{l=1}^{\infty} \theta_{l},$ we notice that by \eqref{9.2(1)}
$$|\vartheta| \leq 2c_{43}c^{-1/2}_{44} \tau^{\alpha/2}.$$
For any $\rho \in (0,\tau/4),$ we may choose $\ok \in \bbN$ so that $\rho \in [\frac{\tau^{\ok+1}}{4},\frac{\tau^{j}}{4}),$ and apply \eqref{9.2(1)},\eqref{9.2(2)} (with $\ok$ replaced by $\ok+1$), \eqref{4.2(3)},\eqref{4.2(4)} (with $T,\theta$ replaced by $T_{\ok},\vartheta-\vartheta_{\ok}$) and \eqref{1.4(1)} to estimate
$$\begin{aligned}
\bE_{C}(\rot_{\vartheta \sharp}T,\rho) \leq & \tau^{-n \alpha} \bE_{C}(\rot_{\vartheta \sharp}T, \tau^{\ok}/4) = 4^{n} \tau^{-n \alpha} \bE_{C}(\ETA_{\frac{1}{4} \sharp} \rot_{\vartheta \sharp} T,\tau^{j}) \\
= & 4^{n} \tau^{-n \alpha} \bE_{C}(\ETA_{\frac{1}{4} \sharp} \rot_{(\vartheta-\vartheta_{\ok}) \sharp}T_{\ok},1) \\
\leq & \tau^{-(n+1) \alpha} \left( \left( \sum_{l=\ok+1}^{\infty} |\theta_{l}|^{2} \right) + \bE_{C}(T_{\ok},1) + \kappa_{T_{\ok}} \right) \\
\leq & c^{-1}_{44} \theta^{-(n+2) \alpha} \rho^{\alpha}.
\end{aligned}$$ \end{proof}

\subsection{Corollary} \label{sec9.3}

The following is essentially Theorem \ref{main}, prior to an application of the Hopf-type boundary point Lemma \ref{sec10.1}.

\begin{corollary}
Suppose $m \geq 1.$ If $T$ and $\vartheta$ are as in Theorem \ref{sec9.2}, if $a = \alpha/(2n+6),$ and if
$$\begin{aligned}
\tilde{V} = &  \{ y \in B^{n}_{\delta}: y_{n} > |y|^{1+a} \}, \\
\tilde{W} = & \{ y \in B^{n}_{\delta}: y_{n} < -|y|^{1+a} \},
\end{aligned}$$
where $\delta = ((4c_{24})^{-2n-3} \tau^{(n+2) \alpha})^{\frac{2}{\alpha}}$ and $c_{24},\tau$ are as in Theorem \ref{sec5.4} and Theorem \ref{sec9.1}, then
$$\begin{aligned}
\bop^{-1}(\tilde{V}) \cap \spt \rot_{\vartheta \sharp} T = \bigcup_{i=1}^{M} \graph{\tilde{V}}{\tilde{v}_{i}}, \\
\bop^{-1}(\tilde{W}) \cap \spt \rot_{\vartheta \sharp} T = \bigcup_{j=1}^{m} \graph{\tilde{W}}{\tilde{w}_{j}}
\end{aligned}$$
for some $\tilde{v}_{i} \in C^{1,a}(\Clos{\tilde{V}}),$ $\tilde{w}_{j} \in C^{1,a}(\Clos{\tilde{W}})$ such that $\tilde{v}_{i}|_{\Tilde{V}},$ $\tilde{w}_{j}|_{\tilde{W}}$ satisfy the minimal surface equation, $D \tilde{v}_{i}(0)=0=D \tilde{w}_{j}(0)=0,$ and
$$\begin{array}{cc}
\tilde{v}_{1} \leq \tilde{v}_{2} \leq \ldots \leq \tilde{v}_{M}, & \tilde{w}_{1} \leq \ldots \leq \tilde{w}_{m}.
\end{array}$$ 
\end{corollary}

\begin{proof} For each $\rho \in (0,\delta)$ we note that, as in the proof of Theorem \ref{sec9.2}, the current $S_{\rho} = (\ETA_{\rho \sharp} \rot_{\vartheta \sharp}T) \res B_{3}$ belongs to $\sT$ and that, by Theorem \ref{sec5.4}, Theorem \ref{sec9.2}, and \eqref{2.3(2)},
\begin{equation} \label{9.3(1)}
\begin{aligned}
\sigma_{S_{\rho}} \leq & c_{24} (\bE_{C}(S_{\rho},1)+\kappa_{S_{\rho}})^{\frac{1}{2n+3}} \leq c_{24} ((\tau^{-(n+2) \alpha} + \tau^{\alpha}) \rho^{\alpha})^{\frac{1}{2n+3}} \\
\leq & c_{24} (2 \tau^{-(n+2) \alpha} \rho^{\alpha/2})^{\frac{1}{2n+3}} \rho^{a} \leq \rho^{a}/4;
\end{aligned}
\end{equation}
this is as in 9.3(1) of \cite{HS79}. Applying \eqref{5.4(1)} (with $T=S_{\rho}$ and $l=1,2$), we infer that
\begin{equation} \label{9.3(2)}
\sup_{\bV_{\rho^{a/4}}} |Dv^{S_{\rho}}_{i}| \leq c_{25} (2 \theta^{-(n+2) \alpha} \rho^{\alpha})^{1/2} (\rho^{a}/4)^{-1} \leq c_{45} \rho^{a},
\end{equation}
\begin{equation} \label{9.3(3)}
\sup_{\bV_{\rho^{\beta/4}}} |D^{2}v^{S_{\rho}}_{i}| \leq c_{25} (2 \theta^{-(n+2) \alpha} \rho^{\alpha})^{1/2} (\rho^{a}/4)^{-2} \leq c_{46} \rho^{a},
\end{equation}
for each $i = 1,\ldots,M,$ with $c_{45},c_{46}$ depending on $n,M,m$; these are is in 9.3(2)(3) of \cite{HS79}.

\medskip

Next, observe that if we take $y \in \tilde{V}$ and if we take $\rho \in (0,2|y|),$ then (by definition of $\tilde{V}$)
\begin{equation} \label{9.3(4)}
\begin{array}{cc}
y_{n}/\rho > \rho^{a}/4 & (\sigma_{S_{\rho}} \text{ by } \eqref{9.3(1)});
\end{array}
\end{equation}
this is as in 9.3(4) of \cite{HS79}. Hence, using \eqref{9.3(1)}, it follows that
$$\bop^{-1}(\tilde{V}) \cap \spt \rot_{\vartheta \sharp} T = \bigcup_{i=1}^{M} \graph{\tilde{V}}{\tilde{v}_{i}}$$
for some functions $\tilde{v}_{i}: \tilde{V} \rightarrow \bbR$ which are well-defined by the condition
\begin{equation} \label{9.3(5)}
\tilde{v}_{i}(y) = \rho v^{S_{\rho}}_{i}(y/\rho) \text{ whenever } y \in \tilde{V} \text{ and } \frac{3}{2} |y| \leq \rho \leq 2 |y|;
\end{equation}
this is as in 9.3(5) of \cite{HS79}. By Theorem \eqref{sec5.4}, $\tilde{v}_{i} \leq \tilde{v}_{2} \leq \ldots \leq \tilde{v}_{M}$ and each $\tilde{v}_{i}$ satisfies the minimal surface equation on $\tilde{V}.$ Moreover, we may use \eqref{9.3(4)},\eqref{9.3(5)},\eqref{9.3(2)},\eqref{9.3(3)} (with $\rho = 2|y|$) to estimate
\begin{equation} \label{9.3(6)}
|D \tilde{v}_{i}(y)| \leq 2c_{45} |y|^{a},
\end{equation}
\begin{equation} \label{9.3(7)}
|D^{2} \tilde{v}_{i}(y)| \leq 2c_{46} |y|^{a} (2|y|)^{-1} = c_{46} |y|^{a-1}
\end{equation}
for any $y \in \tilde{V};$ these are as in 9.3(6)(7) of \cite{HS79}. 

\medskip

To obtain the desired H$\ddot{\text{o}}$lder, we suppose $y,\tilde{y} \in \tilde{V}$ and consider the two possibilities:
\begin{enumerate}
 \item[] \emph{Case 1.} $\max \{ |y|,|\tilde{y}| \} \leq 2|y-\tilde{y}|.$ Here we infer from \eqref{9.3(6)} that
\begin{equation} \label{9.3(8)}
\begin{aligned}
|D\tilde{v}_{i}(y)-D\tilde{v}_{i}(\tilde{y})| \leq & |D\tilde{v}_{i}(y)|+|D\tilde{v}_{i}(\tilde{y})| \\
\leq & 2c_{45} \left( |y|^{a} + |\tilde{y}|^{a} \right) \\
\leq & 4c_{45} |y-\tilde{y}|^{a};
\end{aligned}
\end{equation}
this is as in 9.3(8) of \cite{HS79}.
 \item[] \emph{Case 2.} $\max \{ |y|,|\tilde{y}| \} > 2|y-\tilde{y}|.$ Here
$$|y+t(\tilde{y}-y)| > |y-\tilde{y}| \text{ whenever } t \in [0,1],$$
and so, by \eqref{9.3(7)},
\begin{equation} \label{9.3(9)}
\begin{aligned}
|D \tilde{v}_{i}(y)-D\tilde{v}_{i}(\tilde{y})| \leq & |y-\tilde{y}| \int_{0}^{1} |D^{2}\tilde{v}_{i}(y+t(\tilde{y}-y))| \ dt \\
\leq & c_{46} |y-\tilde{y}|^{a};
\end{aligned}
\end{equation}
\end{enumerate}
this is 9.3(9) of \cite{HS79}.

\medskip

By \eqref{9.3(8)},\eqref{9.3(9)}, and \eqref{9.3(6)} each function $\tilde{v}_{i}$ extends uniquely to a member of $C^{1,a}(\Clos{\tilde{V}})$ with $D \tilde{v}_{i}(0)=0.$

\medskip

The argument to show the existence of $\tilde{w}_{i} \in C^{1,a}(\Clos{\tilde{W}})$ is similar. \end{proof}

\section{A Hopf-type boundary point lemma} \label{sec10}

Here we give a short proof of a Hopf-type boundary point lemma for divergence-form elliptic equations, which was first established by Finn and Gilbarg; see Lemma 7 of \cite{FG57}. This is as in section 10 of \cite{HS79}.

\medskip

We use in section \ref{sec10.1} positive constants $c_{47},c_{48},c_{49}$ from applying Theorem 5.5.5'(b) of \cite{M66}; while these constants do not depend only on $n,M,m,$ we will only use these constants in section \ref{sec10.1}.

\subsection{Lemma} \label{sec10.1}

\begin{lemma}
If $a \in (0,1),$ $\Omega$ is a connected $C^{1,a}$ domain in $\bbR^{n},$ $a_{kl} \in C^{0,a}(\Clos{\Omega})$ for $k,l \in \{ 1,\ldots,n \},$
$$\sum_{k,l=1}^{n} a_{kl}(y) \xi_{k} \xi_{l} \geq |\xi|^{2} \text{ for } y \in \Omega \text{ and } \xi = (\xi_{1},\ldots,\xi_{n}) \in \bbR^{n},$$
$u \in C^{1,a}(\Clos{\Omega},[0,\infty))$ is a weak solution of the equation
$$\sum_{k,l=1}^{n} D_{k}(a_{kl} D_{l} u) = 0$$
on $\Omega,$ $0 \in \partial \Omega,$ and $u(0)=0,$ then
$$\text{either } u \equiv 0 \text{ or } \mathbf{n} \cdot Du(0)>0$$
where $\mathbf{n}$ is the interior unit normal to $\Omega$ at $0.$
\end{lemma}

\begin{proof} We follow the proof of Lemma 10.1 of \cite{HS79}, with only changes in notation.

\medskip

By making a nonsingular $C^{1,a}$ transformation of coordinates near $0$ we may assume $\Omega \cap B^{n}_{1} = \bV$; hence $\mathbf{n} \cdot Du(0) = D_{n}u(0).$ Let
$$A = B^{n}_{\frac{1}{2}}(e_{n}/2) \setminus \bB^{n}_{\frac{1}{4}}(e_{n}/2)$$
and let $\phi$ be the analytic function on $\Clos{A}$ for which 
$$\begin{array}{ll}
\phi|_{\partial B^{n}_{\frac{1}{2}}(\frac{e_{n}}{2})} \equiv 0 & \phi|_{\partial \bB^{n}_{\frac{1}{4}}(\frac{e_{n}}{2})} \equiv 1 \\
\sum_{k,l=1}^{n} a_{kl}(0) D_{k} D_{l} \phi \equiv 0. &
\end{array}$$
By the classical Hopf argument (see, for example, Theorem 3.6 of \cite{GT83}) for nondivergence form equations,
\begin{equation} \label{10.1(1)}
D_{n} \phi(0)>0;
\end{equation}
this is as in 10.1(1) of \cite{HS79}.

\medskip

Next, for each $\epsilon \in (0,1),$ let
$$u^{\epsilon}(y) = u(\epsilon y) \text{ and } a^{\epsilon}_{kl}(y) = a_{kl}(\epsilon y) \text{ for } y \in \bV \text{ and } \{ k,l \} \subseteq \{ 1,\ldots,n \},$$
and note that $a^{\epsilon}_{kl}(0) = a_{kl}(0)$ and that $u^{\epsilon}$ is a weak solution of the equation
\begin{equation} \label{10.1(2)}
\sum_{k,l=1}^{n} D_{k} (a^{\epsilon}_{kl} D_{l} w) = 0
\end{equation}
on $\bV$; this is as in 10.1(2) of \cite{HS79}. By section 23 of \cite{T75} and Theorem 5.5.4' of \cite{M66}, we may choose $\phi_{\epsilon} \in C^{1,a}(\Clos{A})$ a weak solution of \eqref{10.1(2)} on $A$ so that $\phi_{\epsilon}|_{\partial A} = \phi|_{\partial A}.$ Using the functions
$$\zeta^{\epsilon}_{k}(y) = - \sum_{l=1}^{n} (a^{\epsilon}_{kl}(y)-a^{\epsilon}_{kl}(0)) D_{l} \phi(y) \text{ for } y \in \Clos{A} \text{ and } k \in \{ 1,\ldots,n \},$$
we observe that $\psi_{\epsilon} = \phi_{\epsilon}-\phi \in C^{1,a}(\Clos{A})$ is a weak solution of the equation
\begin{equation} \label{10.1(3)}
\sum_{k,l=1}^{n} D_{k} (a^{\epsilon}_{kl} D_{l} w) = \sum_{k=1}^{n} D_{k} \zeta^{\epsilon}_{k}
\end{equation}
on $A$ with $\psi_{\epsilon}|_{\partial A} = 0$; this is as in 10.1(3) of \cite{HS79}. We infer from Theorem 5.5.5'(b) of \cite{M66}
\begin{equation} \label{10.1(4)}
\| \psi_{\epsilon} \|_{C^{1,a}(\Clos{A})} \leq c_{47} \left( \sum_{k=1}^{n} \| \zeta^{\epsilon}_{k} \|_{C^{0,a}(\Clos{A})} + \int_{A} |\psi_{\epsilon}| \ d \cH^{n} \right),
\end{equation}
where $c_{47}$ is a constant independent of $\epsilon$; this is as in 10.4 of \cite{HS79}. Moreover, using the coercivity of the operator on the left of \eqref{10.1(3)} as in section 23 of \cite{T75}, we readily verify that
\begin{equation} \label{10.1(5)}
\int_{A} |\psi_{\epsilon}|^{2} \ d \cH^{n} \leq c_{48} \sum_{k=1}^{n} \int_{A} |\zeta^{\epsilon}_{k}|^{2} \ d \cH^{n} \leq c_{49} \sum_{k=1}^{n} |\zeta^{\epsilon}_{k}|^{2}_{C^{0,a}(\Clos{A})}
\end{equation}
with $c_{48},c_{49}$ independent of $\epsilon$; this is as in 10.1(5) of \cite{HS79}. Combining \eqref{10.1(4)},\eqref{10.1(5)}, and Schwartz's inequality, we conclude that
$$|D_{n} \phi_{\epsilon}(0)-D_{n} \phi(0)| \leq \| \psi_{\epsilon} \|_{C^{1,a}(\Clos{A})} \leq c_{47}(1+\sqrt{c_{49}}) \sum_{l=1}^{n} \| \zeta^{\epsilon}_{l} \|_{C^{0,a}(\Clos{A})} \rightarrow 0$$
as $\epsilon \rightarrow 0.$ Thus we may, by \eqref{10.1(1)}, fix $\epsilon>0$ sufficiently small so that
\begin{equation} \label{10.1(6)}
D_{n} \phi_{\epsilon}(0) > \frac{1}{2} D_{n} \phi(0) > 0;
\end{equation}
this is as in 10.1(6) of \cite{HS79}.

\medskip

In case $u|_{\Omega}$ is strictly positive we can choose $\lambda \in (0,1)$ so that $(u^{\epsilon} - \lambda \phi_{\epsilon})|_{\partial A} \geq 0.$ Since $u^{\epsilon}-\lambda \phi_{\epsilon}$ is a solution of \eqref{10.1(2)}, we may infer from the weak maximum principle (see, for example, Theorem 3.6 of \cite{GT83}) that $D_{n}(u^{\epsilon}-\lambda \phi_{\epsilon})(0) \geq 0$; hence, by \eqref{10.1(6)},
$$D_{n}u(0) = \epsilon D_{n}u^{\epsilon}(0) \geq \epsilon \lambda D_{n} \phi_{\epsilon}(0) > 0.$$

\medskip

To complete the proof we will verify that if $u|_{\Omega}$ is not identically zero, then $u|_{\Omega}$ is strictly positive. Otherwise, there is $\bB_{\rho}(y) \subset \Omega$ and a point $\tilde{y} \in \partial B_{\rho}(y)$ so that $u|_{B_{\rho}(y)}>0$ and $u(\tilde{y})=0.$ Since $u \geq 0$ we infer $Du(\tilde{y})=0,$ contradicting the argument above with $\Omega$ replaced by $B_{\rho}(y).$ \end{proof}

\section{Concluding Theorem \ref{main}} \label{proofmain}

Having established Corollary \ref{sec9.3}, the proof of Theorem \ref{main} follows exactly as in the first part of Theorem 11.1 of \cite{HS79}. Recall that $m \geq 1$ in Theorem \ref{main}.

\medskip

\begin{proof}[Proof of Theorem \ref{main}:] Suppose $T \in \sR^{n}(B_{3})$ satisfies $(\ast),(\ast \ast)$ from Theorem \ref{main}. Choose $r_{\ok} \rightarrow 0$ with $r_{k} < 1$ so that as currents 
$$\ETA_{r_{\ok} \sharp} T \rightarrow M \bbE^{n} \res \{ y \in \bbR^{n} : y_{n} > 0 \} + m \bbE^{n} \res \{ y \in \bbR^{n}: y_{n} < 0 \}.$$ 
By 5.4.2 of \cite{F69} 
$$\lim_{\ok \rightarrow \infty} \sup_{\bB_{r_{\ok}} \cap \spt T} \boq_{1}/r_{\ok} = 0,$$ 
and so we can choose $\ok \in \bbN$ sufficiently large so that 
$$\begin{aligned} 
T_{\ok} & = (\ETA_{r_{\ok} \sharp} T) \res B_{3} \in \sT \\ 
\max \{ \bE_{C}(T_{\ok},1), c_{44} \kappa_{T_{\ok}} \} & \leq c_{44}^{-1}. 
\end{aligned}$$ 
It follows that Theorem \ref{sec9.2} holds for $T_{\ok}$ specifically with $\vartheta = 0.$ By Lemma \ref{sec10.1} we conclude in applying Theorem \ref{sec9.2} to $T_{\ok}$ that $\tilde{v}_{1} = \tilde{v}_{2} = \ldots = \tilde{v}_{M}$ and $\tilde{w}_{1} = \ldots = \tilde{w}_{m}.$ Together with \eqref{9.3(7)}, we now have Theorem \ref{main}. \end{proof}
  
\appendix

\section{Appendix} 

In this section we present some calculations based on the homotopy formula 4.1.9 of \cite{F69}. The identities presented in this Appendix will be needed in \eqref{1.6(1)}, Lemma \ref{sec4.1}, Lemma \ref{sec6.2}, and Lemma \ref{sec6.4}. 

\medskip

We begin with the following lemma, which follows from the constancy theorem (see Theorem 26.27 of \cite{S83}) and induction, which will be used to prove Lemma \ref{appendixlemma2}.

\begin{appendixlemma} \label{appendixlemma1} 
Let $\rho \in (0,\infty)$ and $\sigma \in (0,1).$ Suppose $P \in \sR_{n}(\bbR^{n})$ is nonzero with
\begin{equation} \label{appendixlemma1spt}
\begin{aligned}
\spt P \cap \{ y \in \bbR^{n} & : |(y_{1},\ldots,y_{n-1})| < \rho, |y_{n}| < 1 \} \\
& \subset \{ y \in \bbR^{n} : |(y_{1},\ldots,y_{n-1})| < \rho, |y_{n}| < \sigma \}
\end{aligned}
\end{equation}
and 
\begin{equation} \label{appendixlemma1boundary}
\begin{aligned} 
\partial P \res \{ y \in \bbR^{n} : & |(y_{1},\ldots,y_{n-1})| < \rho, |y_{n}| < 1 \} \\ 
= & (-1)^{n} \sum_{\ell=1}^{N} m_{\ell} \Phi_{P,\ell \sharp}(\bbE^{n-1} \res B^{n-1}_{\rho}) \\
& + (-1)^{n-1} m_{0} \bbE^{n-1} \res B^{n-1}_{\rho} 
\end{aligned}
\end{equation}
where $N,m_{0},m_{1},\ldots,m_{N} \in \bbN$ satisfy $\sum_{\ell=1}^{N} m_{\ell} = m_{0},$ and for each $\ell = 1,\ldots,N$ the map $\Phi_{P,\ell} \in C^{1}(B^{n-1}_{\rho};\bbR^{n})$ is given by $\Phi_{P,\ell}(z) = (z,\varphi_{P,\ell}(z))$ where $\varphi_{P,\ell} \in C^{1}(B^{n-1}_{\rho})$ with $\sup_{B^{n-1}_{\rho}}|\varphi_{P,\ell}| < \sigma.$

\medskip

Then there is a nonempty $K \subseteq \bbN$ such that for each $k \in K$ there is an open nonempty connected set $O_{P,k} \subset \{ y \in \bbR^{n}: |(y_{1},\ldots,y_{n-1})| < \rho, |y_{n}|< \sigma \}$ and an integer $m_{O_{P,k}} \neq 0$ so that 
$$P \res \{ y \in \bbR^{n} : |(y_{1},\ldots,y_{n-1})| < \rho, |y_{n}| < 1 \} = \sum_{k \in K} m_{O_{P,k}} \bbE^{n} \res O_{P,k}.$$
Moreover, $m_{O_{P,k}} \in [-m_{0},0)$ if $O_{P,k} \cap \{ y \in \bbR^{n}: y_{n}>0 \} \neq \emptyset,$ while $m_{O_{P,k}} \in (0,m_{0}]$ if $O_{P,k} \cap \{ y \in \bbR^{n}: y_{n}<0 \} \neq \emptyset.$ 
\end{appendixlemma}

\medskip

\begin{proof} We prove this by induction on $n.$ Note first that the constancy theorem (see Theorem 26.27 of \cite{S83}) implies
\begin{equation} \label{appendixlemma1decomposition}
P \res \{ y \in \bbR^{n} : |(y_{1},\ldots,y_{n-1})| < \rho, |y_{n}| < 1 \} = \sum_{k \in K} m_{O_{P,k}} \bbE^{n} \res O_{P,k} 
\end{equation} 
for some nonempty $K \subseteq \bbN,$ where $m_{O_{P,k}} \neq 0$ is an integer and $O_{P,k}$ is a nonempty open connected subset of
$$\{ y \in \bbR^{n} : |(y_{1},\ldots,y_{n-1})| < \rho, |y_{n}| < \sigma \} \setminus \bigcup_{\ell=1}^{N} \Phi_{P,\ell}(B^{n-1}_{\rho})$$ 
for each $k \in K.$ We now begin our proof by induction.

\medskip

{\bf n=2:} Define $\tilde{\varphi}_{P,1},\ldots,\tilde{\varphi}_{P,N} \in C((-\rho,\rho))$ so that
$$\tilde{\varphi}_{P,1} \leq \tilde{\varphi}_{P,2} \leq \ldots \leq \tilde{\varphi}_{P,N-1} \leq \tilde{\varphi}_{P,N}$$
and so that for each $z \in (-\rho,\rho)$ we have
$$\{ \tilde{\varphi}_{P,\tilde{\ell}}(z) \}_{\tilde{\ell}=1}^{N} = \{ \varphi_{P,\ell}(z) \}_{\ell=1}^{N}.$$
Take $O_{P,k}$ as in \eqref{appendixlemma1decomposition}, and suppose $O_{P,k} \cap \{ y \in \bbR^{2}: y_{2}>0 \} \neq \emptyset.$ The constancy theorem together with \eqref{appendixlemma1spt},\eqref{appendixlemma1boundary} imply there is an open interval $I_{k} \subset (-\rho,\rho)$ and an $\tilde{\ell}_{k} \in \{1,\ldots,N\}$ so that 
$$O_{P,k} = \{ y \in \bbR^{2} : y_{1} \in I_{k}, \ y_{2} \in (\max \{ 0,\tilde{\varphi}_{P,\tilde{\ell}_{k}-1}(y_{1})\},\tilde{\varphi}_{P,\tilde{\ell}_{k}}(y_{1})) \}.$$

\medskip

First, suppose $\tilde{\varphi}_{P,\tilde{\ell}_{k}}(z) = \tilde{\varphi}_{P,N}(z)$ for each $z \in I_{k}.$ It follows we can find $\ell_{1},\ldots,\ell_{N-\tilde{\ell}_{k}+1} \in \{1.\ldots,N \}$ so that 
$$\varphi_{P,\ell_{1}}(z) = \ldots = \varphi_{P,\ell_{N-\tilde{\ell}_{k}+1}}(z) = \tilde{\varphi}_{P,N}(z)$$
for each $z \in I_{k},$ and hence 
$$\begin{aligned}
\partial O_{P,k} \cap \{ y \in \bbR^{2} & : \max \{ 0,\tilde{\varphi}_{P,\tilde{\ell}_{k}-1}(y_{1}) \} < y_{2} \} \\
& = \Phi_{P,\ell_{1}}(I_{k}) = \ldots = \Phi_{P,\ell_{N-\tilde{\ell}_{k}+1}}(I_{k}).
\end{aligned}$$
From this $m_{O_{P,k}} = - (m_{\ell_{1}}+\ldots+m_{\ell_{N-\tilde{\ell}_{k}+1}})$ follows, and so $m_{O_{P,k}} \in [-m_{0},0).$

\medskip

Second, suppose $\tilde{\varphi}_{P,\tilde{\ell}_{k}}(z) = \tilde{\varphi}_{P,N-1}(z)$ for each $z \in I_{k},$ but $\tilde{\varphi}_{P,N-1}(z) < \tilde{\varphi}_{P,N}(z)$ for some $z \in I_{k}.$ We can thus find an open interval $\tilde{I}_{k} \subset I_{k}$ and an $O_{P,\tilde{k}}$ from \eqref{appendixlemma1decomposition} disjoint from $O_{P,k}$ so that 
$$\begin{aligned}
\partial O_{P,\tilde{k}} \cap & \{ y \in \bbR^{2} : y_{1} \in \tilde{I}_{k},y_{2}>0 \} \\
& = \{ (z,\tilde{\varphi}_{P,N-1}(z)): z \in \tilde{I}_{k} \} \cup \{ (z,\tilde{\varphi}_{P,N}(z)) : z \in \tilde{I}_{k} \}.
\end{aligned}$$ 
The previous paragraph applied to $O_{P,\tilde{k}}$ implies there are $\ell_{1},\ldots,\ell_{N-\tilde{\ell}_{k}+1} \in \{1,\ldots,N \}$ so that 
$$m_{O_{P,k}} = - (m_{\ell_{1}}+\ldots+m_{\ell_{N-\tilde{\ell}_{k}+1}}),$$
and from this $m_{O_{P,k}} \in [-m_{0},0)$ follows.

\medskip

Third, we can argue inductively that $m_{O_{P,k}} \in [-m_{0},0)$ whenever $O_{P,k} \cap \{ y \in \bbR^{2}: y_{2}>0 \} \neq \emptyset.$ By likewise first considering $\tilde{\varphi}_{P,1},$ we can show $m_{O_{P,k}} \in (0,m_{0}]$ whenever $O_{P,k} \cap \{ y \in \bbR^{2}: y_{2}<0 \} \neq \emptyset.$ This shows the lemma in case $n=2.$

\medskip

{\bf n $>$ 2.} With $\bop_{1}:\bbR^{n+1} \rightarrow \bbR$ given by $\bop_{1}(x)=x_{1},$ take any of the Lebesgue almost-every $t \in (-\rho,\rho)$ such that the slice
$$\langle P,\bop_{1},t \rangle = \partial \big( P \res \{ x \in \bbR^{n+1}: x_{1} < t \} \big) - (\partial P) \res \{ x \in \bbR^{n+1}: x_{1} < t \}$$ 
exists, by 4.3.6 of \cite{F69}. Then the following two facts also hold for Lebesque almost-every $t \in (-\rho,\rho).$ First, by \eqref{appendixlemma1spt},
$$\begin{aligned}
\spt & \langle P,\bop_{1},t \rangle \cap \{ y \in \bbR^{n} : |(y_{1},\ldots,y_{n-1})| < \rho, |y_{n}| < 1 \} \\
& = \spt \langle P,\bop_{1},t \rangle \cap \{ y \in \bbR^{n} : y_{1}=t, |(y_{2},\ldots,y_{n-1})| < \rho-t, |y_{n}| < 1 \} \\
& \subset \{ y \in \bbR^{n}: y_{1}=t, |(y_{2},\ldots,y_{n-1})| < \rho-t, |y_{n}| < \sigma \}.
\end{aligned}$$
Second, by \eqref{appendixlemma1boundary} and Lemma 28.5(3) of \cite{S83},
$$\begin{aligned} 
\partial \langle & P,\bop_{1},t \rangle \res \{ y \in \bbR^{n}: y_{1}=t,|(y_{2},\ldots,y_{n-1})| < \rho-t, |y_{n}| < 1 \} \\ 
= & (-1)^{n-1} \sum_{\ell=1}^{N} m_{\ell} \Phi_{P,\ell \sharp} (\bbE^{t,n-2} \res \{ z \in \bbR^{n-1}: z_{1}=t, |(z_{2},\ldots,z_{n-1})| < \rho-t \}) \\
& + (-1)^{n-2} m_{0} \bbE^{t,n-2} \res \{ z \in \bbR^{n-1}: z_{1}=t, |(z_{2},\ldots,z_{n-1})| < \rho-t \}
\end{aligned}$$
where $\bbE^{t,n-2}$ is the $(n-2)$-dimensional current in $\bbR^{n}$ given by 
$$\bbE^{t,n-2}(\omega) = \int_{\{ z \in \bbR^{n-1}: z_{1}=t \}} \langle \omega, e_{2} \wedge \ldots \wedge e_{n-1} \rangle \ d \cH^{n-2} \text{ for } \omega \in \mathcal{D}^{n-2}(\bbR^{n}).$$
These two facts imply by induction that whenever $\langle P,\bop_{1},t \rangle \neq 0$ 
$$\begin{aligned}
\langle P,\bop_{1},t \rangle & \res \{ y \in \bbR^{n} : y_{1}=t, |(y_{2},\ldots,y_{n-1})| < \rho-t, |y_{n}| < 1 \} \\
& = \sum_{k \in K^{t}} m_{O^{t}_{P,k}} \bbE^{t,n-1} \res O^{t}_{P,k}
\end{aligned}$$ 
where $\bbE^{t,n-1}_{t}$ is the $(n-1)$-dimensional current in $\bbR^{n}$ given by 
$$\bbE^{t,n-1}(\omega) = \int_{\{ y \in \bbR^{n}: y_{1}=t \}} \langle \omega, e_{2} \wedge \ldots \wedge e_{n} \rangle \ d \cH^{n-1} \text{ for } \omega \in \mathcal{D}^{n-1}(\bbR^{n}),$$
and where $K^{t} \subseteq \bbN$ is a nonempty set such that for each $k \in K^{t}$ the set $O^{t}_{P,k} \subset \{ y \in \bbR^{t,n-1}: |(y_{2},\ldots,y_{n-1})| < \rho-t, |y_{n}| < \sigma \}$ is a (nonempty) open connected set, and $m_{O^{t}_{P,k}} \neq 0$ is an integer so that $m_{O^{t}_{P,k}} \in [-m_{0},0)$ whenever $O^{t}_{P,k} \cap \{ y \in \bbR^{t,n-1}: y_{n}>0 \} \neq \emptyset$ while $m_{O^{t}_{P,k}} \in (0,m_{0}]$ whenever $O^{t}_{P,k} \cap \{ y \in \bbR^{t,n-1}: y_{n}<0 \} \neq \emptyset.$

\medskip

For each $k \in K$ (as in \eqref{appendixlemma1decomposition}), we can choose $t \in (-\rho,\rho)$ so that $O_{P,k} \cap \{ y \in \bbR^{n}: y_{1}=t \} \neq \emptyset$ and so that $\langle P,\bop_{1},t \rangle \neq 0$ exists. Thus $m_{O_{P,k}}= m_{O^{t}_{P,\tilde{k}}}$ for some $\tilde{k} \in K^{t}.$ We conclude the lemma. \end{proof}

The following calculations will be used throughout, and hence we collect them here. In particular, \eqref{appendixlemma2projectionmass} is instrumental in checking that the proofs of \cite{HS79} carry over to the more general setting of Theorem \ref{main}; see the proof of Lemma \ref{sec6.4}, analogous to Lemma 6.4 of \cite{HS79}. On the other hand, the fact that \eqref{appendixlemma2projectionmass} holds in general only with $m \geq 1$ means we can only presently prove Theorem \ref{main} with $m \geq 1.$

\begin{appendixlemma} \label{appendixlemma2} 
Let $M \in \bbN,$ $m \in \{ 0,\ldots,M-1 \},$  and let $\alpha \in (0,1].$ Suppose  $T \in \sT = \sT(M,m,\alpha)$ (see Definition \ref{sec1.5,1.6}). With $q: \bbR \times \bbR^{n+1} \rightarrow \bbR^{n+1}$ given by
$$q(t,x) = (x_{1},\ldots,x_{n-1},tx_{n},tx_{n+1}) \text{ for } (t,x) \in \bbR \times \bbR^{n+1},$$ 
define $Q_{T} \in \sR_{n}(\bbR^{n+1})$ by
$$Q_{T} = q_{\sharp} \left( (\bbE^{1} \res [0,1]) \times (\partial T \res C_{2}) \right)$$ 
(see 4.1.9 of \cite{F69}). Then for every $r \in (0,2)$ 
\begin{equation} \label{appendixlemma2retraction} 
\begin{aligned} 
\bM(Q_{T} \res C_{r}) \leq & (M-m) \kappa_{T} \varpi_{n-1} r^{n+\alpha} \\
& \times \left( \frac{\alpha}{2} \right) \left( 1+\frac{\alpha^{2} \kappa_{T}^{2}}{4} r^{2 \alpha}+ \frac{\alpha^{4} \kappa_{T}^{4}}{16} r^{4 \alpha} \right)^{\frac{1}{2}}, \\ 
\bM (\bop_{\sharp} Q_{T} \res C_{r}) \leq & (M-m) \kappa_{T} \varpi_{n-1} r^{n+\alpha}, 
\end{aligned}
\end{equation} 
and we as well have
\begin{equation} \label{appendixlemma2identity}
\begin{aligned}
\partial Q_{T} \res C_{r} & = \partial T \res C_{r} - (-1)^{n}(M-m) \bbE^{n-1} \res C_{r} \\ 
(\bop_{\sharp} T) \res \bB^{n}_{r} & = (M \bbE^{n} \res \bV + m \bbE^{n} \res \bW + \bop_{\sharp} Q_{T}) \res \bB^{n}_{r}. 
\end{aligned}
\end{equation} 
If $m \geq 1,$ $\sigma \in (0,1/2),$ and $\kappa_{T} < \sigma$ (see \eqref{defderivative}), then
\begin{equation} \label{appendixlemma2projectionmass} 
\begin{aligned}
\bM ((\bop_{\sharp} T - \bbE^{n}) \res & \{ y \in B^{n}_{\frac{1}{2}} : |y_{n}| < \sigma \} ) \\ 
= & \bM ( \bop_{\sharp} T \res \{ y \in B^{n}_{\frac{1}{2}} : |y_{n}| < \sigma \}) \\ 
& - \bM (\bbE^{n} \res \{ y \in B^{n}_{\frac{1}{2}} : |y_{n}| < \sigma \}).
\end{aligned}
\end{equation} 
\end{appendixlemma} 

\begin{proof} First, we compute by 4.1.9 of \cite{F69} (see as well the proof of 26.23 of \cite{S83}) and \eqref{defboundary},\eqref{defderivative} (as in the end of the proof of Lemma \ref{sec3.2})
$$\begin{aligned} 
\bM(Q_{T} \res C_{r}) \leq &  \int \sqrt{\boq_{0}^{2}+\boq_{1}^{2}} \ d \mu_{\partial T \res C_{r}} \\ 
\leq & \left( \frac{\alpha}{2} \right) \kappa_{T} r^{1+\alpha} \mu_{\partial T}(C_{r}) \\ 
\leq & (M-m) \kappa_{T} \varpi_{n-1} r^{n+\alpha} \\
& \times \left( \frac{\alpha}{2} \right) \left( 1+\frac{\alpha^{2} \kappa_{T}^{2}}{4} r^{2\alpha}+ \frac{\alpha^{4} \kappa_{T}^{4}}{16} r^{4 \alpha} \right)^{\frac{1}{2}}. \end{aligned}$$ 
By a similar calculation for $\bM(\bop_{\sharp} Q_{T} \res C_{r}),$ we conclude \eqref{appendixlemma2retraction}.

\medskip

Second, for any $r \in (0,2)$ we have by 4.1.8-9 of \cite{F69} (see as well 26.22 of \cite{S83}, the homotopy formula) the first identity in \eqref{appendixlemma2identity} and
$$\partial ( \bop_{\sharp} T - M \bbE^{n} \res \bV - m \bbE^{n} \res \bW - \bop_{\sharp} Q_{T}) \res \bB^{n}_{r} = 0,$$
recalling that $\sum_{\ell=1}^{N} m_{\ell} = M-m$ by \eqref{defboundary}. This proves the second identity in \eqref{appendixlemma2identity}, by the constancy theorem and \eqref{defprojection}.

\medskip

Third, suppose $m \geq 1$ and with $\sigma \in (0,1/2)$ assume $\kappa_{T} < \sigma.$ Let $P = (\bop_{\sharp} Q_{T}) \res B^{n}_{1}.$ If $P=0,$ then \eqref{appendixlemma2projectionmass} readily follows from \eqref{appendixlemma2identity} (since $1 \leq m \leq (M-1)$). Now suppose $P \neq 0,$ then we first note that \eqref{defderivative},\eqref{defprojection}, and \eqref{appendixlemma2identity} imply 
$$\begin{aligned}
\spt P \cap \{ y \in \bbR^{n} & : |(y_{1},\ldots,y_{n-1})| < 1/2, |y_{n}| < 1 \} \\
& \subset \{ y \in \bbR^{n}: |(y_{1},\ldots,y_{n-1})| < 1/2, |y_{n}| < \sigma \}.
\end{aligned}$$
This together with \eqref{defboundary},\eqref{appendixlemma2identity} means we apply Lemma \ref{appendixlemma1} with $\rho = 1/2$ and $m_{0} = M-m$ to conclude
$$P \res \{ y \in \bbR^{n}: |(y_{1},\ldots,y_{n-1})|<1/2, |y_{n}| < 1 \} = \sum_{k \in K} m_{O_{P,k}} \bbE^{n} \res O_{P,k},$$ 
where we recall $m_{O_{P,k}} \in [-(M-m),0)$ if $O_{P,k} \cap \{ y \in \bbR^{n}: y_{n}>0 \} \neq \emptyset$ while $m_{O_{P,k}} \in (0,(M-m)]$ if $O_{P,k} \cap \{ y \in \bbR^{n}: y_{n}<0 \} \neq \emptyset.$ By further writing
$$\begin{aligned}
P \res \{ y \in \bbR^{n} & : |(y_{1},\ldots,y_{n-1})|<1/2, |y_{n}| < 1 \} \\
= & \sum_{\{ k \in K: O_{P,k} \cap \bW = \emptyset \}} m_{O_{P,k}} \bbE^{n} \res O_{P,k} \\
& + \sum_{\{ k \in K: O_{P,k} \cap \bW \neq \emptyset \}} m_{O_{P,k}} \bbE^{n} \res (O_{P,k} \cap \bV) \\
& + \sum_{\{ k \in K: O_{P,k} \cap \bW \neq \emptyset \}} m_{O_{P,k}} \bbE^{n} \res (O_{P,k} \cap \bW),
\end{aligned}$$ 
then using \eqref{appendixlemma2identity} we can compute each side of \eqref{appendixlemma2projectionmass} in terms of $M,m,\{ m_{O_{P,k}} \}_{k \in K}$ in order to verify equality (for this, $m \geq 1$ is needed). \end{proof}

The following calculation, \eqref{appendixlemma3RT}, is used in the proof of Lemma \ref{sec4.1}. We prove it here for the sake of cleaner exposition.

\begin{appendixlemma} \label{appendixlemma3} 
Let $M \in \bbN,$ $m \in \{ 0,\ldots,M-1 \},$ $\alpha \in (0,1],$ and suppose  $T \in \sT = \sT(M,m,\alpha)$ (see Definition \ref{sec1.5,1.6}). Also suppose $\tau \in (0,1)$ and that $\kappa_{T} \leq 4 \cdot 3^{n} (1+M \varpi_{n}) \tau^{2}$ (see \eqref{defderivative}). Let $\phi \in C^{1}(\bbR^{n};[0,1])$ satisfy
$$\begin{aligned}
\phi(y) & = 0 & \text{ if } & |y| \leq 1 \\ 
0 < \phi(y) & <1 & \text{ if } & 1 < |y| < 1 + \tau \\
\phi(y) & = 1 & \text{ if } & 1+\tau \leq |y| \\
|D\phi(y)| & \leq 3/\tau & \text{ for} & \text{ all } y \in \bbR^{n},
\end{aligned}$$
and define $F : \bbR^{n+1} \rightarrow \bbR^{n+1}$ and $h:\bbR \times \bbR^{n+1} \rightarrow \bbR^{n+1}$ by
$$F(x) = (\bop(x),f(\bop(x)) x_{n+1}) \text{ and } h(t,x) = (1-t)F(x)+tx.$$
Then $R_{T} = h_{\sharp} ([0,1] \times \partial T)$ satisfies the mass bound
\begin{equation} \label{appendixlemma3RT} 
\bM(R_{T}) \leq \left( \frac{\sqrt{21}}{8}+ 2^{\frac{9n-7}{2}} 3^{n^{2}-\frac{1}{2}} \right) (M-m) \varpi_{n-1} (1+M \varpi_{n})^{n-1} \kappa_{T}.
\end{equation} 
\end{appendixlemma} 

\begin{proof} We compute using 4.1.9 of \cite{F69} (see also the end of section \ref{sec3.2}), 
$$\begin{aligned} 
\bM(R_{T}) & \leq \int (1-\phi)|\boq_{1}| \max \left\{ 1, \left(1+\phi^{2}+ \boq_{1}^{2} |D\phi|^{2} \right)^{\frac{n-1}{2}} \right\} \ d \mu_{\partial T} \\ 
& \begin{aligned} 
\leq \Big( & \frac{\alpha}{2} \kappa_{T} \Big) \mu_{\partial T}(C_{1}) \\
+ & \left( \frac{\alpha}{2} \kappa_{T} \right) (1+\tau)^{1+\alpha} \left(2+ \left( \frac{\alpha^{2}}{4} \right) \kappa_{T}^{2} (1+\tau)^{2+2\alpha} \left( \frac{9}{\tau^{2}}\right) \right)^{\frac{n-1}{2}} \\
& \times \mu_{\partial T} (C_{1+\tau} \setminus C_{1})
\end{aligned} \\ 
& \leq \frac{1}{2} \left( \mu_{\partial T}(C_{1}) + 2^{2} \left(2+ \frac{36 \kappa_{T}^{2}}{\tau^{2}} \right)^{\frac{n-1}{2}} \mu_{\partial T}(C_{1+\tau} \setminus C_{1}) \right) \kappa_{T} \\ 
& \begin{aligned} 
\leq \frac{1}{2} \Big( & \mu_{\partial T}(C_{1})  + 2^{2} \left( 2+ 2^{6} \cdot 3^{2n+2} (1+M \varpi_{n})^{2} \right)^{\frac{n-1}{2}} \mu_{\partial T} (C_{1+\tau} \setminus C_{1}) \Big) \kappa_{T} 
\end{aligned} \\ 
& \begin{aligned}
\leq \frac{1}{2} \Big( & \mu_{\partial T}(C_{1}) + 2^{\frac{7n-3}{2}} \cdot 3^{n^{2}-1} (1+M \varpi_{n})^{n-1} \mu_{\partial T}( C_{1+\tau} \setminus C_{1}) \Big) \kappa_{T} 
\end{aligned} \\ 
& \leq \frac{1}{2} \left( \frac{\sqrt{21}}{4}+ 2^{\frac{9n-5}{2}} 3^{n^{2}-\frac{1}{2}} \right) (M-m) \varpi_{n-1} (1+M \varpi_{n})^{n-1} \kappa_{T}. 
\end{aligned}$$ \end{proof}

We will also need the following lemma for the proof of Lemma \ref{sec6.1}. Again, we give it here for the sake of cleaner exposition.

\begin{appendixlemma} \label{appendixlemma4} 
Let $M \in \bbN,$ $m \in \{ 0,\ldots,M-1 \},$ $\alpha \in (0,1],$ and suppose $T \in \sT = \sT(M,m,\alpha)$ (see Definition \ref{sec1.5,1.6}). Also suppose there is a $\sigma \in (0,1)$ so that
\begin{equation} \label{appendixlemma4excess}
\bE_{C}(T,1) + \kappa_{T} < \frac{(\frac{2}{3})^{n} \varpi_{n}}{(1+(M-m) \varpi_{n-1}+c_{4})} \sigma^{n+1}
\end{equation}
with $c_{4}$ as in \eqref{2.2(4)},\eqref{2.2(5)}, and
\begin{equation} \label{appendixlemma4boundary}
\partial (T \res B_{1-\frac{\sigma}{6}}) \res C_{1-\frac{\sigma}{3}} = (\partial T) \res C_{1-\frac{\sigma}{3}}.
\end{equation}
Then
$$\bop_{\sharp} \big( (T \res B_{1-\frac{\sigma}{6}}) \res C_{1-\frac{\sigma}{3}} \big) = \bop_{\sharp} (T \res C_{1-\frac{\sigma}{3}}).$$
\end{appendixlemma}

\begin{proof} First, we compute using \eqref{2.2(5)} and \eqref{defmass}
$$\left( \frac{M+m}{2} \right) \varpi_{n} - \frac{\mu_{T}(\bB_{1-\frac{\sigma}{3}})}{\left( 1-\frac{\sigma}{3} \right)^{n}} = - \bE_{S} \left( T,1-\frac{\sigma}{3} \right) \leq c_{4} \kappa_{T}.$$
This together with Lemma 26.25 of \cite{S83}, \eqref{1.4(1)}, and \eqref{appendixlemma2retraction},\eqref{appendixlemma2identity} gives
$$\begin{aligned}
\bM \Big( \bop_{\sharp} (T \res & C_{1-\frac{\sigma}{3}} \setminus \bB_{1-\frac{\sigma}{3}}) \Big) \\
\leq & \mu_{T}(C_{1-\frac{\sigma}{3}} \setminus \bB_{1-\frac{\sigma}{3}}) \\
= & \left( 1-\frac{\sigma}{3} \right)^{n} \bE_{C} \left( T,1-\frac{\sigma}{3} \right) +  \mu_{\bop_{\sharp} T}(C_{1-\frac{\sigma}{3}}) - \mu_{T}(\bB_{1-\frac{\sigma}{3}}) \\
\leq & \bE_{C}(T,1) + \mu_{\bop_{\sharp} T}(C_{1-\frac{\sigma}{3}}) - \mu_{T}(\bB_{1-\frac{\sigma}{3}}) \\
\leq & \bE_{C}(T,1) + \left( \frac{M+m}{2} \right) \varpi_{n} \left( 1-\frac{\sigma}{3} \right)^{n} \\
& + (M-m) \kappa_{T} \varpi_{n-1} \left( 1-\frac{\sigma}{3} \right)^{n+\alpha} - \mu_{T}(\bB_{1-\frac{\sigma}{3}}) \\
\leq & \bE_{C}(T,1) + (M-m) \kappa_{T} \varpi_{n-1} \left( 1-\frac{\sigma}{3} \right)^{n+\alpha} + \left( 1-\frac{\sigma}{3} \right)^{n} c_{4} \kappa_{T} \\
\leq & (1+(M-m) \varpi_{n-1}+c_{4}) (\bE_{C}(T,1) + \kappa_{T}),
\end{aligned}$$
recalling as well $\sigma \in (0,1).$ Combining with \eqref{appendixlemma4excess} gives
$$\bM \Big( \bop_{\sharp} (T \res C_{1-\frac{\sigma}{3}} \setminus \bB_{1-\frac{\sigma}{3}}) \Big) \leq \left( \frac{2}{3} \right)^{n} \varpi_{n} \sigma^{n+1}.$$
On the other hand, the constancy theorem (see Theorem 26.27 of \cite{S83}) and \eqref{appendixlemma4boundary} imply
$$\bop_{\sharp} (T \res C_{1-\frac{\sigma}{3}} \setminus \bB_{1-\frac{\sigma}{3}}) = \tilde{m} \bbE^{n} \res \bB^{n}_{1-\frac{\sigma}{3}}$$
for some integer $\tilde{m}.$ The above estimate implies we need
$$|\tilde{m}| \left( \frac{2}{3} \right)^{n} \varpi_{n} \leq |\tilde{m}| \left( 1-\frac{\sigma}{3} \right)^{n} \varpi_{n} \leq \left( \frac{2}{3} \right)^{n} \varpi_{n} \sigma^{n+1}.$$
Since $\sigma \in (0,1)$ we conclude $\tilde{m}=0,$ from which the lemma follows. \end{proof}

\end{flushleft}
\end{document}